\newif\ifpreprint
\preprinttrue  

\ifpreprint

\documentclass[12pt]{article}
\usepackage{fullpage}

\title{Learning to Warm-Start Fixed-Point Optimization Algorithms}

\author{Rajiv Sambharya$^1$, Georgina Hall$^2$, Brandon Amos$^3$, \\and Bartolomeo Stellato$^1$}
\date{%
    $^1$Princeton University\\%
    $^2$INSEAD\\
    $^3$Meta AI\\[2ex]%
    \today
}

\usepackage[round,authoryear]{natbib}
\usepackage{amsmath,enumitem,mathtools,amsthm,amssymb}

\renewcommand\arraystretch{1.2}
\newtheorem{theorem}{Theorem}
\newtheorem{lemma}[theorem]{Lemma}
\newtheorem{corollary}[theorem]{Corollary}
\usepackage{hyperref}

\usepackage[algo2e,ruled]{algorithm2e}
\usepackage{subcaption}
\theoremstyle{definition}
\newtheorem{definition}{Definition}[section]

\else

\documentclass[twoside,11pt]{article}
\usepackage{blindtext}
\usepackage{jmlr2e}
\usepackage{lastpage}
\usepackage{amsmath}

\ShortHeadings{Learning to Warm-Start Fixed-point Optimization Algorithms}{Sambharya, Hall, Amos, and Stellato}
\firstpageno{1}

\renewcommand\arraystretch{1.2}

\title{Learning to Warm-Start Fixed-Point Optimization Algorithms}

\author{%
 \name Rajiv Sambharya \email rajivs@princeton.edu\\
 \addr Operations Research and Financial Engineering, Princeton University, Princeton, NJ, USA
 \AND
 \name Georgina Hall \email georgina.hall@insead.edu\\
 \addr Decision Sciences, INSEAD, Fontainebleau, France
 \AND
 \name Brandon Amos \email bda@meta.com\\
 \addr Meta AI, New York City, NY, USA
 \AND
 \name Bartolomeo Stellato \email bstellato@princeton.edu\\
 \addr Operations Research and Financial Engineering, Princeton University, Princeton, NJ, USA
}

\editor{My editor}
\jmlrheading{23}{2023}{1-\pageref{LastPage}}{9/23; Revised ??}{??}{21-0000}{Rajiv Sambharya and Georgina Hall and Brandon Amos and Bartolomeo Stellato}

\usepackage{caption}
\usepackage{graphicx}
\usepackage{subcaption}

\fi
\usepackage{makecell}
\usepackage{todonotes}

\definecolor{bdacolor}{RGB}{168, 141, 201}

\definecolor{ghcolor}{RGB}{48,213,200}

\usepackage{verbatim}
\usepackage{multicol}
\usepackage{algorithm}
\usepackage[noend]{algpseudocode} 
\usepackage{csvsimple}	
\usepackage{booktabs}   
\usepackage[flushleft]{threeparttable}  
\usepackage[round-mode=places, round-integer-to-decimal, round-precision=4,
    table-format = 1.4,
    table-number-alignment=center,
    round-integer-to-decimal]{siunitx}
\usepackage{adjustbox}  

\newcommand{\Tk}{T^k_{\theta}}

\newcommand{\op}{T_{\theta}}

\newcommand{\Lo}{\ell_{\theta}}

\newcommand{\hw}[1][w]{h_{#1}}
\newcommand{\hwp}[1][u]{h_{w + #1}}

\DeclareMathOperator*{\argmin}{argmin}

\newcommand{\Lok}[1][z]{\ell_{\theta}^{\rm{fp}}(T_{\theta}^k(#1))}

\newcommand{\Lot}[1][z]{\ell_{\theta}^{\rm{fp}}(T_{\theta}^t(#1))}
\newcommand{\Loit}[1][z]{\ell_{\theta_i}^{\rm{fp}}(T_{\theta_i}^t(#1))}
\newcommand{\rok}[1][z]{r_{\theta}(T_{\theta}^t(#1))}
\newcommand{\rokp}[1][z]{r_{\theta}(T_{\theta}^{t+1}(#1))}
\newcommand{\gok}[1][\gamma]{g_{#1,\theta}^t}
\newcommand{\goik}[1][\gamma]{g_{#1,\theta_i}^t}
\newcommand{\Ns}{\mathcal{N}(0,\sigma^2)}

\newcommand{\tw}{\tilde{w}}
\newcommand{\bh}{\bar{h}}

\newcommand{\prior}{\pi}
\newcommand{\paramdist}{Q}
\newcommand{\fth}{f_{\theta}}
\newcommand{\gth}{g_{\theta}}

\DeclareDocumentCommand{\T}{ O{z} O{} }{T\IfValueT{#2}{(#1,\theta_{#2})}\IfNoValueT{#2}{(#1,\theta)}}
\DeclareDocumentCommand{\Tj}{ O{k} O{z} O{} }{T^{#1}\IfValueT{#3}_{\theta_{#3}}{(#2)}}

\DeclareDocumentCommand{\CB}{ O{} }{C_{B_{\IfValueTF{#1}{\theta_{#1}}{\theta}}}}
\DeclareDocumentCommand{\RB}{ O{} }{R_{B_{\IfValueTF{#1}{\theta_{#1}}{\theta}}}}

\newcommand{\eg}{{\it e.g.}}
\newcommand{\ie}{{\it i.e.}}

\newcommand{\reals}{{\mbox{\bf R}}}
\newcommand{\complex}{{\mbox{\bf C}}}

\newcommand{\symm}{{\mbox{\bf S}}}  
\newcommand{\prob}{{\mathbf P}}

\newcommand{\Rank}{\mathop{\bf rank}}
\newcommand{\Tr}{\mathop{\bf tr}}
\newcommand{\diag}{\mathop{\bf diag}}

\newcommand{\dist}{{\bf dist{}}}
\newcommand{\fix}{\mathop{\bf fix{}}}
\newcommand{\vect}{\mathop{\bf vec{}}}

\newcommand{\dom}{\mathop{\bf dom}} 

\newcommand{\prox}{\textbf{prox}}
\newcommand{\card}{\textbf{card}}

\newcommand{\Sec}{Section}
\newcommand{\Thm}{Theorem}

\newcommand{\Eqn}{Equation}

\newcommand{\Lem}{Lemma}
\newcommand{\Def}{Definition}

\newcommand{\cblock}[3]{
  \hspace{-1.5mm}
  \begin{tikzpicture}
    [
    node/.style={square, minimum size=10mm, thick, line width=0pt},
    ]
    \node[fill={rgb,255:red,#1;green,#2;blue,#3}] () [] {};
  \end{tikzpicture}%
}

\newcommand{\ccircle}[3]{%
  \raisebox{0.55\height}{
  \begin{tikzpicture}[baseline=(node.base)]%
    \node[circle, fill={rgb,255:red,#1;green,#2;blue,#3}, yshift=14] (node) at (0,1) {};%
  \end{tikzpicture}%
  }
}

\newcommand{\linediamond}[3]{%
  \begin{tikzpicture}[baseline={(0,-.1)}]
    \coordinate (A) at (0.15,0);
    \coordinate (B) at (.25,-.1);
    \coordinate (C) at (.35,0);
    \coordinate (D) at (.25,.1);
    \draw[draw=none, fill={rgb,255:red,#1;green,#2;blue,#3}] (A) -- (B) -- (C) -- (D) -- cycle;
    \draw[line width=0.8pt, color={rgb,255:red,#1;green,#2;blue,#3}] (0,0) -- (0.5,0);
  \end{tikzpicture}%
  }

\newcommand{\linecircle}[3]{%
  \begin{tikzpicture}[baseline={(0,-.1)}]
    \draw[draw=none, fill={rgb,255:red,#1;green,#2;blue,#3}](0.25, 0) circle (.08);
    \draw[line width=0.8pt, color={rgb,255:red,#1;green,#2;blue,#3}] (0,0) -- (0.5,0);
  \end{tikzpicture}%
  }

\newcommand{\linesquare}[3]{%
  \begin{tikzpicture}[baseline={(0,-.1)}]
    \draw[draw=none, fill={rgb,255:red,#1;green,#2;blue,#3}](0.179, -.071) rectangle (.321,.071);
    \draw[line width=0.8pt, color={rgb,255:red,#1;green,#2;blue,#3}] (0,0) -- (0.5,0);
  \end{tikzpicture}%
  }

\newcommand{\linelefttri}[3]{%
  \begin{tikzpicture}[baseline={(0,-.1)}]
    \coordinate (A) at (0.17,0);
    \coordinate (B) at (.3,-.1);
    \coordinate (C) at (.3,.1);
    \draw[draw=none, fill={rgb,255:red,#1;green,#2;blue,#3}] (A) -- (B) -- (C) -- cycle;
    \draw[line width=0.8pt, color={rgb,255:red,#1;green,#2;blue,#3}] (0,0) -- (0.5,0);
  \end{tikzpicture}%
}

\newcommand{\linerighttri}[3]{
  \begin{tikzpicture}[baseline={(0,-.1)}]
    \coordinate (A) at (0.33,0);
    \coordinate (B) at (.2,-.1);
    \coordinate (C) at (.2,.1);
    \draw[draw=none, fill={rgb,255:red,#1;green,#2;blue,#3}] (A) -- (B) -- (C) -- cycle;
    \draw[line width=0.8pt, color={rgb,255:red,#1;green,#2;blue,#3}] (0,0) -- (0.5,0);
  \end{tikzpicture}%
}

\newcommand{\linedowntri}[3]{
  \begin{tikzpicture}[baseline={(0,-.1)}]
    \coordinate (A) at (0.25,-.09);
    \coordinate (B) at (.17,.06);
    \coordinate (C) at (.33,.06);
    \draw[draw=none, fill={rgb,255:red,#1;green,#2;blue,#3}] (A) -- (B) -- (C) -- cycle;
    \draw[line width=0.8pt, color={rgb,255:red,#1;green,#2;blue,#3}] (0,0) -- (0.5,0);
  \end{tikzpicture}%
}

\newcommand{\lineuptri}[3]{
  \begin{tikzpicture}[baseline={(0,-.1)}]
    \coordinate (A) at (0.25,.09);
    \coordinate (B) at (.17,-.06);
    \coordinate (C) at (.33,-.06);
    \draw[draw=none, fill={rgb,255:red,#1;green,#2;blue,#3}] (A) -- (B) -- (C) -- cycle;
    \draw[line width=0.8pt, color={rgb,255:red,#1;green,#2;blue,#3}] (0,0) -- (0.5,0);
  \end{tikzpicture}%
}

\newcommand{\legend}{\vspace{-1mm} \small \linedowntri{0}{0}{0} cold start \hspace{1mm}
  \linelefttri{166}{86}{40} nearest neighbor warm start \hspace{1mm}
  \\learned warm-start $k=$\{\hspace{-1mm}
      \linerighttri{152}{78}{163} $0$
      \linecircle{228}{26}{28} $5$
      \linesquare{55}{126}{184} $15$
      \linediamond{77}{175}{74} $60$
      \}\hspace*{-2mm}}

\newcommand{\legendprevsol}{\vspace{-1mm} \small \linedowntri{0}{0}{0} cold start \hspace{1mm}
 \linelefttri{166}{86}{40} nearest neighbor warm start \hspace{1mm}
  \lineuptri{255}{127}{0} previous solution warm start \hspace{1mm}
  \\learned warm-start $k=$\{\hspace{-1mm}
      \linerighttri{152}{78}{163} $0$
      \linecircle{228}{26}{28} $5$
      \linesquare{55}{126}{184} $15$
      \linediamond{77}{175}{74} $60$
      \}\hspace*{-2mm}}

\newcommand{\iters}{\small{(a) Mean iterations to reach a given fixed point residual (Fp res.)}}

\newcommand{\reduction}{\small{(b) Mean reduction in iterations from a cold start to a given fixed-point residual (Fp res.)}}

\newcommand{\scstiming}{\small{(c) Mean solve times (in milliseconds) in SCS with absolute and relative tolerances set to tol.
  }}

\newcommand{\osqptiming}{\small{(c) Mean solve times (in milliseconds) in OSQP with absolute and relative tolerances set to tol.
  }}

\newcommand{\figsize}{0.79}

\newenvironment{talign}
 {\let\displaystyle\textstyle\align}
 {\endalign}
\newenvironment{talign*}
 {\let\displaystyle\textstyle\csname align*\endcsname}
 {\endalign}

\newcommand{\Cgt}{C_{\gamma}(t)}
\newcommand{\Cgtt}{C_{\gamma / 2}(t)}
\newcommand{\quadcopterwidthjmlr}{0.2}
\newcommand{\quadcopterwidthpreprint}{0.25}
\makeatletter
\renewcommand{\eqref}[1]{\textup{\tagform@{\ref{#1}}}}
\makeatother

\begin{document}
\maketitle
\begin{abstract}%
We introduce a machine-learning framework to warm-start fixed-point optimization algorithms.
Our architecture consists of a neural network mapping problem parameters to warm starts, followed by a predefined number of fixed-point iterations.
We propose two loss functions designed to either minimize the fixed-point residual or the distance to a ground truth solution.
In this way, the neural network predicts warm starts with the end-to-end goal of minimizing the downstream loss.
An important feature of our architecture is its flexibility, in that it can predict a warm start for fixed-point algorithms run for any number of steps, without being limited to the number of steps it has been trained on.
We provide PAC-Bayes generalization bounds on unseen data for common classes of fixed-point operators: contractive, linearly convergent, and averaged.
Applying this framework to well-known applications in control, statistics, and signal processing, we observe a significant reduction in the number of iterations and solution time required to solve these problems, through learned warm starts.
\end{abstract}

\section{Introduction}
\label{sec:intro}

We consider {\it parametric fixed-point problems} of the form
\begin{equation}\label{prob:fp}
  \mbox{find} \; z \quad \mbox{ such that } \quad  z = \op(z),
\end{equation}
where $z \in \reals^p$ is the decision variable and $\theta \in  \Theta \subseteq \reals^d$ is the {\it problem parameter} defining each instance of~\eqref{prob:fp} via the \emph{fixed-point operator}~$\op$.
We assume that $\theta$ is drawn from an unknown distribution $Q$, accessible only via samples, and that for every $\theta \in \Theta$, problem~\eqref{prob:fp} is a  solvable (\ie, $\op$ admits a fixed-point) convex optimization problem.
Almost all convex optimization problems can be cast as finding a fixed-point of an operator~\citep{lscomo}, often representing the optimality conditions~\citep{cosmo,scs,osqp}.
To solve problem~\eqref{prob:fp}, we repeatedly apply the {operator} $\op$, obtaining the iterations
\begin{equation}\label{eq:fp_algo}
  z^{i+1} = \op(z^i).
\end{equation}
We assume that the iterations~\eqref{eq:fp_algo} converge to a fixed-point, that is, $\lim_{i \rightarrow \infty}\|z^{i}-z^\star(\theta)\|= 0$, where $z^\star(\theta)$ is a fixed point of $\op$. In practice, it is common to return an $\epsilon$-approximate solution, corresponding to a vector $z^i$ for which the {\it fixed-point residual}, $\|\op(z^i) - z^i\|_2$ is below $\epsilon.$
Many optimization algorithms correspond to fixed-point iterations of the form~\eqref{eq:fp_algo}; see~Table~\ref{table:fp_algorithms} for some examples.

\begin{table}[!h]
  \centering
  \caption{Many optimization algorithms can be written as fixed-point iterations.}
  \label{table:fp_algorithms}
  \adjustbox{max width=\textwidth}{
  \begin{threeparttable}
  \begin{tabular}{@{}llll@{}}
    \toprule[\heavyrulewidth]
    Algorithm & Problem & Iterates & Fixed-point operator $\op$ \\
    \midrule[\heavyrulewidth]
    Gradient descent & $\begin{array}{@{}ll}\min &\fth(z)\end{array}$ & $z^{i+1} = z^i - \alpha \nabla f_{\theta}(z^i)$ & $\op(z) = z - \alpha \nabla f_{\theta}(z)$ \\
    \midrule
    Proximal gradient descent
    & $\begin{array}{@{}ll}\min & \fth(z) + \gth(z)\end{array}$ & $z^{i+1} = \prox_{\alpha \gth} (z^i - \alpha \nabla f_{\theta}(z^i))$ & $\op(z) = \prox_{\alpha \gth} (z - \alpha \nabla f_{\theta}(z))$ \\
    \midrule
    \makecell[l]{ADMM \\ \citep{dr_splitting} \\ \citep{Gabay1976ADA}} & $\begin{array}{@{}ll}\min &\fth(u) + \gth(u)\end{array}$ & $\begin{aligned}
            \tilde{u}^{i+1} &= \prox_{g_{\theta}}(z^i)\\
            u^{i+1} &= \prox_{f_{\theta}}(2 \tilde{u}^{i+1} - z^i)\\
            z^{i+1} &= z^i + u^i - \tilde{u}^i
          \end{aligned}$ & $\op(z) = z + \prox_{f_{\theta}}(2\prox_{g_{\theta}}(z) - z)$ \\
    \midrule
    \makecell[l]{OSQP \\ \citep{osqp}} & $\begin{array}{@{}ll}
    \min & (1/2) x^T Px + c^T x\\
    \mbox{s.t.}& l \leq Ax \leq u \\
    &\\&\\
    \text{with}&\theta = (\vect(P),\vect(A), c, l, u)\\
    \end{array}$
    &
    $\begin{aligned}
      &\text{solve }
      Q x^{i+1} = \sigma x^i - c + A^T (\rho w^i - y^i)\\
      &w^{i+1} = \Pi_{[l,u]}(A x^i + \rho^{-1} y^i)\\
      &y^{i+1} = y^i + \rho(Ax^{i+1} - w^{i+1})\\
      &\\
      &\text{with}\quad Q = P + \sigma I + \rho A^T A &
    \end{aligned}$ &
    $\op(z)$ where $z=(x,Ax + \rho^{-1}y)$\\
    \midrule[\lightrulewidth]
    \makecell[l]{SCS \\ \citep{scs_quadratic}} & $\begin{array}{@{}ll}
    \min & (1/2) x^T Px + c^T x\\
    \mbox{s.t.}& Ax + s = b \\
    & s \in \mathcal{K} \\
    &\\
      \text{with}&\theta = (\vect(P), \vect(A), c, b)
    \end{array}$ & $\begin{aligned}
      & \text{solve }Q \tilde{u}^{i+1} = z^i\\
      &u^{i+1} = \Pi_{\mathcal{C}}(2 \tilde{u}^{i+1} - z^i)\\
      &z^{i+1} = z^i + u^{i+1} - \tilde{u}^{i+1}\\
      &\\
      &\text{with} \quad Q = \begin{bmatrix}
          P + I & A^T\\
          -A & I
        \end{bmatrix}\\
        &\hspace{3em} \mathcal{C} = \reals^n \times \mathcal{K}
    \end{aligned}$
    & $\begin{aligned}
        &\op(z) \text{ where } z \text{ is the} \\
        &\text{dual variable to } u=(x,y)\\
      \end{aligned}$
    \\
    \bottomrule[\heavyrulewidth]
  \end{tabular}
  \begin{tablenotes}
  \item  We denote~\prox{} as the proximal operator~\citep{prox_algos} and~$\vect$ as the vectorization operator stacking the columns of a matrix (See the notation paragraph in~\ref{sec:intro} for formal definitions). See Appendix~\ref{app:fp_algos} for more information on the algorithms in this table.
  \end{tablenotes}
  \end{threeparttable}
  }
\end{table}



\paragraph{Applications.}
{\it Parametric fixed-point problems} arise in several applications in machine learning, operations research, and engineering, where we repeatedly solve a problem of the form~\eqref{prob:fp} with varying parameter $\theta$.
For example, in optimal control, we update the inputs (\eg, propeller thrusts) as sensor signals (\eg, system state) and goals (\eg, desired trajectory) vary \citep[\Sec~7.1]{borrelli_mpc_book}.
Other examples include backtesting financial models~\citep{cvx_portfolio}, power flow optimization~\citep{ml_opf,ml_opf_smart}, and image restoration~\citep{learned_dictionaries_images}.
In non-convex optimization, finding a stationary point can also be cast as a fixed-point problem~\citep{admm_nonconvex,admm_nonconvex2}.
In game theory, finding the Nash equilibrium of a multi-player game can be formulated as a fixed-point problem under some mild assumptions on the utility functions of each player~\citep{nash_operator,mon_primer}.
Finding fixed-points are also important in other areas, such as finding the optimal policy of Markov decision processes~\citep{bellman_dp} and solving variational inequality problems~\citep{rockefeller_variational,bauschke2011convex}.

\paragraph{Acceleration.}
In spite of the widespread use of fixed-point iterative algorithms, they are known to suffer from slow convergence to high-accuracy solutions~\citep{zhang2020globally}.
Acceleration schemes \citep{zhang2020globally,anderson_acceleration,accel_survey,super_scs} are an active area of research designed to speed up the solving of fixed-point problems.
These methods, such as Anderson acceleration \citep{anderson_acceleration,zhang2020globally},
combine past iterates to generate the next one in order to improve the convergence behavior.
Although acceleration methods are known to work well in certain cases, such as Nesterov acceleration to solve smooth, convex optimization problems, it is still an open research question to design schemes that are robust and versatile.

\paragraph{Learning for optimization.}
Instead of designing acceleration methods for single problems, recent approaches take advantage of the parametric structure of fixed-point problems encountered in practice to learn efficient solution methods.
In particular, they {\it learn algorithm steps} using data from past solutions~\citep{amos_tutorial,l2o}.
Despite recent successes in a variety of fields, \eg, in sparse coding~\citep{lista,alista}, convex optimization~\citep{rlqp,neural_fp_accel_amos}, and meta-learning~\citep{learning_to_optimize_malik,maml}, most of these approaches lack {\it convergence guarantees} because they directly alter the algorithm iterations with learned variants~\citep{l2o,amos_tutorial}.
Although some efforts have been made to safeguard the learned iterations~\citep{safeguard_l2o,safeguard_convex, banert2021accelerated}, guaranteeing convergence for general learned optimizers is still a challenge.
In addition, most of these approaches do not provide {\it generalization guarantees} on unseen data~\citep{l2o,amos_tutorial}.

Another data-driven approach to reduce the number of iterations is to {\it learn warm starts} rather than the steps of the algorithm~\citep{mpc_primal_active_sets,warm_start_power_flow}.
An advantage to learning warm starts as opposed to algorithm steps is that this approach can be integrated with existing algorithms that provably converge from any starting point.
However, existing methods to learn warm starts still lack generalization guarantees.
They also decouple the learning procedure from the algorithm behavior after warm-starting, which can lead to suboptimality and infeasibility issues on unseen problem instances.


\paragraph{Our contributions.}
We present a learning framework that predicts warm starts for iterative algorithms of the form~\eqref{eq:fp_algo}, which solve parametric fixed-point problems of the type given in~\eqref{prob:fp}.
The framework consists of two modules.
The first module maps the parameter to a warm start via a neural network, and the second runs a predefined number of steps of the fixed-point algorithm.
We propose two loss functions.
The first one is the \emph{fixed-point residual} loss which directly penalizes the fixed-point residual of the output of the architecture.
The second one is the \emph{regression} loss which penalizes the distance between the output of the architecture and a given ground truth fixed-point (among possibly many).

Compared to existing literature on learning warm starts, we train our architecture by differentiating through the fixed-point iterations.
In this way, we construct warm-start predictions that perform well after a specific number of fixed-point iteration in an {\it end-to-end} fashion.
Furthermore, after training, our architecture allows the flexibility of selecting an arbitrary number of fixed-point iterations to perform and is not limited to the number it was originally trained on.


By combining operator theory with the PAC-Bayes framework~\citep{McAllester_PAC_Bayes,Shawe-Taylor_PAC_Bayes}, we provide two types of guarantees on the performance of our framework.
First, we give bounds on the fixed-point residual when we apply our framework to an arbitrary number of steps, larger than the number used during training.
Second, we provide generalization bounds to unseen problems for common classes of operators: contractive, linearly convergent, and averaged.


Finally, we apply our framework to a variety of algorithms including gradient descent, proximal gradient descent, and the alternating direction method of multipliers.
In our benchmarks, we show that our learned warm starts lead to a significant reduction in the required number of iterations used to solve the fixed-point problems.
We also demonstrate compatibility with state-of-the-art solvers by learning architectures specifically tailored to SCS~\citep{scs} and OSQP~\citep{osqp}, and inputting warm starts into the corresponding C implementations.

\paragraph{Notation.}
We denote the set of non-negative vectors of length $n$ as $\reals_+^n$, and the set of vectors with positive entries of length $n$ as $\reals_{++}^n$.
We let the set of $n \times n$ positive semidefinite and positive definite matrices be $\symm_+^n$ and $\symm_{++}^n$ respectively.
We define the set of fixed-points of the operator $T$, assumed to be non-empty, as $\fix T$.
For any closed and convex set $S$, we denote $\dist_{S}: \reals^n \rightarrow \reals$ to be the distance function, where $\dist_{S}(x) = \min_{s \in S} \|s - x\|_2$.
For any set $S \subset \reals^n$, we define the indicator function $\mathcal{I}_{S}: \reals^n \rightarrow \reals \cup \{+\infty\}$ where $\mathcal{I}_{S}(x) = 0$ if $x \in S$ and $\mathcal{I}_{S}(x) = +\infty$ otherwise.
We take $k$ applications of any single-valued operator $T$ to be $T^k: \reals^n \rightarrow \reals^n$.
For any matrix $A$, we denote its spectral norm and Frobenius norm with $\|A\|_2$ and $\|A\|_F$ respectively.
For a matrix $Z \in \reals^{m \times n}$, $\vect(Z)$ is the vector obtained by stacking the columns of $Z$.
For a symmetric matrix $Y \in \symm^{n}$, $\vect(Y)$ is the vector obtained by  taking the upper triangular entries of matrix $Y$.
We let the all-ones vectors of length $n$ be $\mathbf{1} \in \reals^n$.
The proximal operator, $\prox_h : \reals^n \rightarrow \reals^n$, of $h$ is defined as \citep{prox_algos}
\begin{equation*}
  \prox_h(v) = \argmin_x \Bigl(h(x) + (1/2)\|x - v\|_2^2\Bigr).
\end{equation*}



\paragraph{Outline.}
We structure the rest of the paper as follows.
In \Sec~\ref{sec:related_work}, we review some related work on learned solvers.
In \Sec~\ref{sec:framework}, we present our learning to warm-start framework.
In \Sec~\ref{sec:PAC_Bayes}, we provide generalization guarantees to unseen data for our method.
In \Sec~\ref{sec:pick_k}, we discuss choosing the right architecture, namely the choice of loss function and the number steps to train on.
\Sec~\ref{sec:numerical_experiments} presents various numerical benchmarks.

\section{Related work}\label{sec:related_work}

\paragraph{Learning warm starts.}
A common approach to reduce the number of iterations of iterative algorithms is to learn a mapping from problem parameters to high-quality initializations.
\cite{l2ws_l4dc} learn warm starts for Douglas-Rachford splitting to solve convex quadratic programs (QPs).
While this work conducts end-to-end learning, our work is more general in scope since we consider fixed-point problems rather than QPs.
Additionally, we provide generalization guarantees for more cases of operators, and finally, we add a regression loss.
In contrast to our approach, most of the techniques to learn warm starts don't consider the downstream algorithm in the warm start prediction.
\cite{warm_start_power_flow} and \cite{mak2023learning} use machine learning to warm-start the optimal power flow problem.
In the model predictive control (MPC)~\citep{borrelli_mpc_book} paradigm,~\cite{mpc_primal_active_sets} use a neural network to accelerate the optimal control law computation by warm-starting an active set method.
Other works in MPC use machine learning to predict an approximate optimal solution and, instead of using it to warm-start an algorithm, directly ensure feasibility and optimality.
\cite{mpc_constrained_neural_nets} and \cite{deep_learning_mpc_karg} use a constrained neural network architecture that guarantees feasibility by projecting its output onto the QP feasible region.
\cite{borrelli_mpc_primal_dual} uses a neural network to predict the solution while also certifying suboptimality of the output.
Our paper differs from these works in that the training of the neural network we propose is designed to minimize the loss after many fixed-point steps, allowing us to improve solution quality. 
Our work is also more general in scope since we consider general parametric fixed-point problems.
Finally, we provide generalization guarantees to unseen data which other works lack.



\paragraph{Learning algorithm steps for convex optimization.}
In the area of learning to optimize~\citep{l2o} or amortized optimization~\citep{amos_tutorial}, a parallel approach to learning warm starts consists in learning the algorithm steps themselves to solve convex optimization problems.
\cite{rlqp} and \cite{qp_accelerate_rho} use reinforcement learning to solve quadratic programs quickly by learning high-quality hyperparameters of algorithms.
\cite{neural_fp_accel_amos} learns to accelerate fixed-point problems that correspond to convex problems quickly.
One risk of some of these approaches 
is that convergence may not be guaranteed~\citep{amos_tutorial}.
To solve this problem, some works safeguard learned optimizers to guarantee convergence by reverting to a fallback update if the learned update starts to diverge~\citep{safeguard_convex,safeguard_l2o}.
Other strategies guarantee convergence by making sure that the learned algorithm does not deviate too much from a known convergent algorithm~\citep{banert2021accelerated} or by providing convergence rate bounds~\citep{learn_mirror}.
In addition to convergence challenges, approaches that learn algorithm steps generally do not have generalization guarantees to unseen data~\citep{amos_tutorial,l2o}.
Lastly, these methods generally cannot interface with existing algorithms that are written in C.

\paragraph{Learning algorithm steps beyond convex optimization.}
Many works have learned algorithm steps for problems outside of convex optimization.
For example, in non-convex optimization, \cite{sjolund2022graphbased} use graph neural networks~\citep{GNNBook2022} to accelerate algorithms to solve matrix factorization problems, and \cite{bai2022neural} learn the acceleratation scheme to solve fixed-point problems quickly.
The idea of learning algorithm steps has ventured beyond optimization.
There has been a surge in recent years to learn algorithm steps to solve \emph{inverse problems}, that is, problems where one wishes to recover a true signal, rather than minimizing an objective~\citep{l2o}.
This is typically done by embedding algorithm steps or reasoning layers~\citep{reasoning_layer} into a deep neural network and has been applied to various fields such as sparse coding~\citep{lista,alista,glista}, image restoration~\citep{Diamond2017UnrolledOW,denoise_prior,one_network}, and wireless communication~\citep{deep_learning_mimo,deep_unfolding_wireless}.
A widely used technique involves \emph{unrolling} algorithmic steps~\citep{algo_unrolling}, meaning differentiating through these steps to minimize a performance loss.
While we also unroll algorithm steps, our work is different in scope since we aim to solve optimization problems rather than inverse problems, and in method since we learn warm starts rather than algorithm steps.
Additionally, generalization and convergence remain issues in the context of learning to solve inverse problems~\citep{l2o,amos_tutorial}.

\paragraph{Learning surrogate optimization problems.}
Instead of solving the original parametric problem, several works aim to learn a surrogate model of large optimization problems.
Then, an approximate solution can be obtained by solving the simpler or smaller optimization problem.
For instance, \cite{surrogate_learning} learn a mapping to reduce the dimensionality of the decision variables in the surrogate problem.
\cite{kolouri_learn_cons} use a neural approximator with reformulation and relaxation steps to solve linearly constrained optimization problems.
Other works predict which constraints are active~\citep{learn_active_sets} and the value of the optimal integer solutions~\citep{voice_optimization,online_milliseconds}.
In contrast, our approach refrains from approximating any problem; instead, we warm-start the fixed-point iterations.
This allows us to clearly quantify the suboptimality achieved within a set number of fixed-point iterations.

\paragraph{Meta-learning.}
Meta-learning~\citep{hospedales2020metalearning,meta_learning_survey,ruder2017overview} or learning to learn overlaps with the learning for optimization literature when the tasks are general machine learning tasks~\citep{l2o}.
A wide array of works learn the update function to gradient-based methods to speed up machine learning tasks with a variety of techniques including reinforcement learning~\citep{learning_to_optimize_malik}, unrolled gradient steps~\citep{learn_learn_gd_gd}, and evolutionary strategies~\citep{VeLO}.
More in the spirit of our work, \cite{maml}~learn the initial model weights so that a new task can be learned after only a few gradient updates.
While the initialization of the model weights for their method is shared across the tasks, in our method we, instead, predict the warm start from the problem parameter.
This tailors our initialization to the specific parametric problem under consideration.
\paragraph{Algorithms with predictions.}
Another area that uses machine learning to improve algorithm performance is algorithms with predictions~\citep{algos_with_predictions,learned_index_structures,khodak2022learning}.
Here, algorithms take advantage of a possibly imperfect prediction of some aspect of the problem to improve upon worst-case analysis.
This idea has been applied to many problems such as ski-rental~\citep{ski_rental}, caching~\citep{caching}, and bipartite matching~\citep{fast_matching}.
Even though the prediction can be used to improve the warm start for algorithms~\citep{fast_matching,learned_algos_pred_ws_discrete_convex}, the task we consider is fundamentally different since we aim to solve parametric problems as quickly as possible rather than to take advantage of a prediction.

\paragraph{Generalization guarantees.}
The generalization guarantees we provide use a PAC-Bayes framework, which has been used in prior work in the amortized optimization setting~\citep{pac_bayes_gen_l2o,bartlett2022generalization}.
\cite{reasoning_layer} provide generalization guarantees for architectures with reasoning layers, using a local Rademacher complexity analysis.
However, to the best of our knowledge, generalization guarantees have not been obtained with methods that aim to solve fixed-point problems quickly.
Additionally, the bounds from these works mentioned above are obtained in methods where the algorithm steps are learned rather than the warm start.
Unlike~\citet{l2ws_l4dc} which focused on solving QPs, we obtain guarantees in the non-contractive case by using the PAC-Bayes framework rather than Rademacher complexity theory.

\section{Learning to warm-start framework}\label{sec:framework}
We now present our learning framework to learn warm starts to solve the parametric fixed-point problem~\eqref{prob:fp}.
A key feature of our framework is the inclusion of a predefined number of fixed-point steps within the architecture.
In this way, the warm-start predictions are tailored for the downstream algorithm, and we conduct end-to-end learning.
The section is organized as follows.
In \Sec~\ref{sec:illustrative_example}, we provide intuition as to why learning end-to-end can be beneficial through a small illustrative example.
In \Sec~\ref{sec:architecture}, we describe our architecture, and in \Sec~\ref{sec:loss_functions} we introduce the two different loss functions we consider.
A concise summary of these aspects is depicted in Figure~\ref{fig:nn_architecture}.

\subsection{An illustrative example}\label{sec:illustrative_example}
To build intuition, we provide a two-dimensional example that illustrates the importance of tailoring the warm-start prediction to the downstream algorithm.
Consider the problem,
\begin{equation}
  \begin{array}{ll}
  \label{prob:illustrative_ex}
  \mbox{minimize} & (1 / 2) z^TQ z\\
  \mbox{subject to} & z \geq 0, \\
  \end{array}
\end{equation}
where $Q = \diag(10, 1)$. We solve problem~\eqref{prob:illustrative_ex} using proximal gradient descent (see Table \ref{table:fp_algorithms}) with the iterates
\begin{equation*}
  z^{i+1} = \Pi(z^i - \alpha \nabla f(z^i))
\end{equation*}
where $\nabla f(z) = Qz$, $\alpha \in \reals_{++}$ is picked to get the fastest worst-case convergence rate~\citep{mon_primer}, and $\Pi$ is the projection onto the non-negative orthant.
The optimal solution for problem~\eqref{prob:illustrative_ex} is at the origin, and we consider three different warm starts shown in Figure~\ref{fig:illustrative_example}.
All three are equidistant to the optimal solution, but lead to different convergence behavior.
The purple warm start has the fastest convergence since the projection step clips non-negative values to zero.
The orange warm start converges more quickly than the green warm start due to the difference in scaling of the objective function along each axis.
This results in faster convergence for the orange warm start compared with the green one since the orange warm start is closer to the $z_1$ axis.
This example shows the necessity of considering the downstream algorithm when choosing a warm start.
All three warm starts in this case appear of equal quality as they are equidistant from $z^\star$, but when considering the downstream algorithm, there is a clear hierarchy in terms of convergence speed: purple takes the lead, followed by orange, then green.
\begin{figure}[!h]
  \centering
  \includegraphics[width=.6\columnwidth]{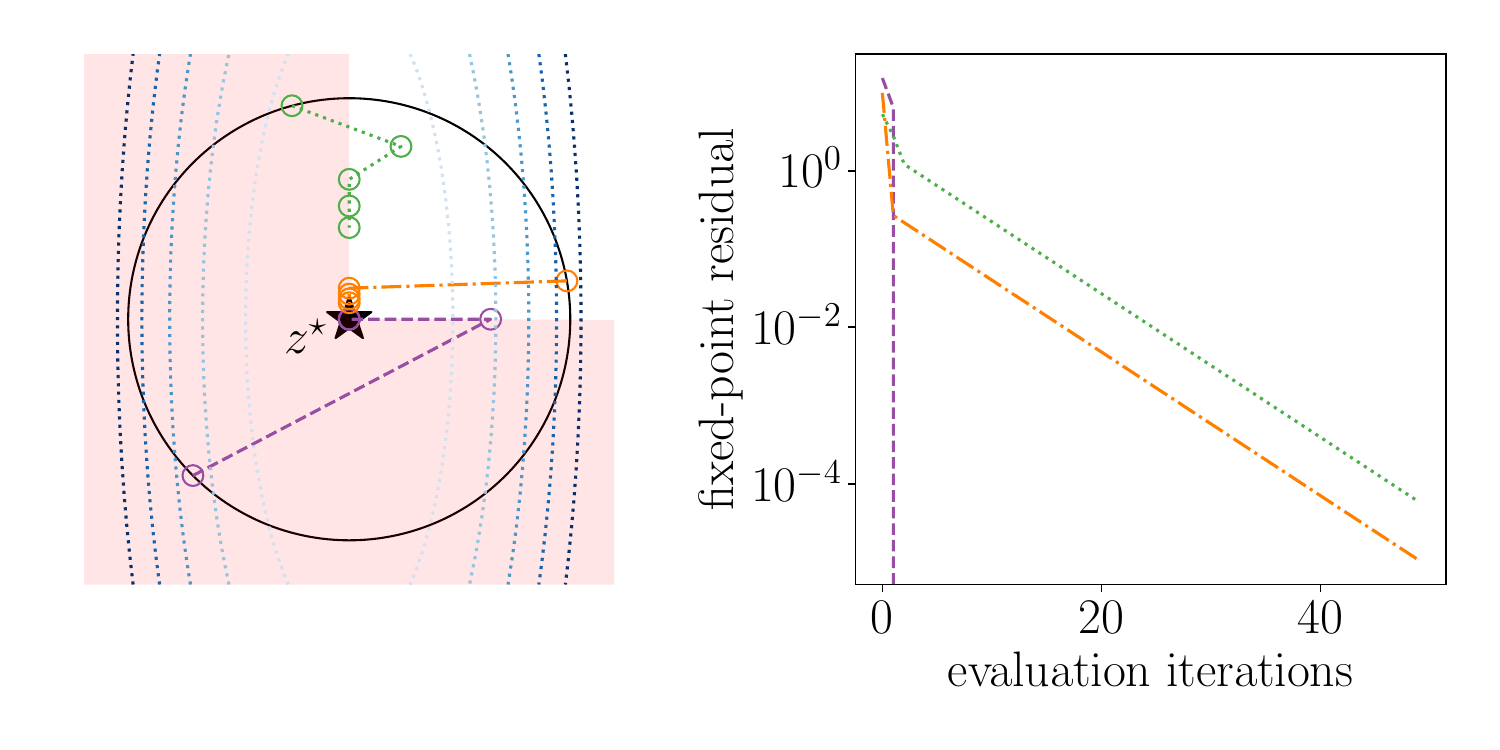}
  \vspace{-3mm}
  \caption{
  The iterates of proximal gradient descent to solve problem~\eqref{prob:illustrative_ex} with different warm starts.
  For three different warm starts equidistant to the optimal solution $z^\star$, we plot the first $5$ iterates on the left.
  The contour lines of the objective function are in blue and the infeasible region is shaded in pink.
  We plot the fixed-point residuals for the different warm starts on the right.
  Depending on the warm start, the convergence to the optimal solution, can vary greatly.
  }
  \label{fig:illustrative_example}
\end{figure}


\subsection{Learning to warm-start architecture}\label{sec:architecture}
Our learning architecture consists of two modules, a neural network with $L$ layers and~$k$ iterations of operator~$\op$; see Figure~\ref{fig:nn_architecture}.
\begin{figure}[!h]
  \centering
  \includegraphics[width=.75\columnwidth]{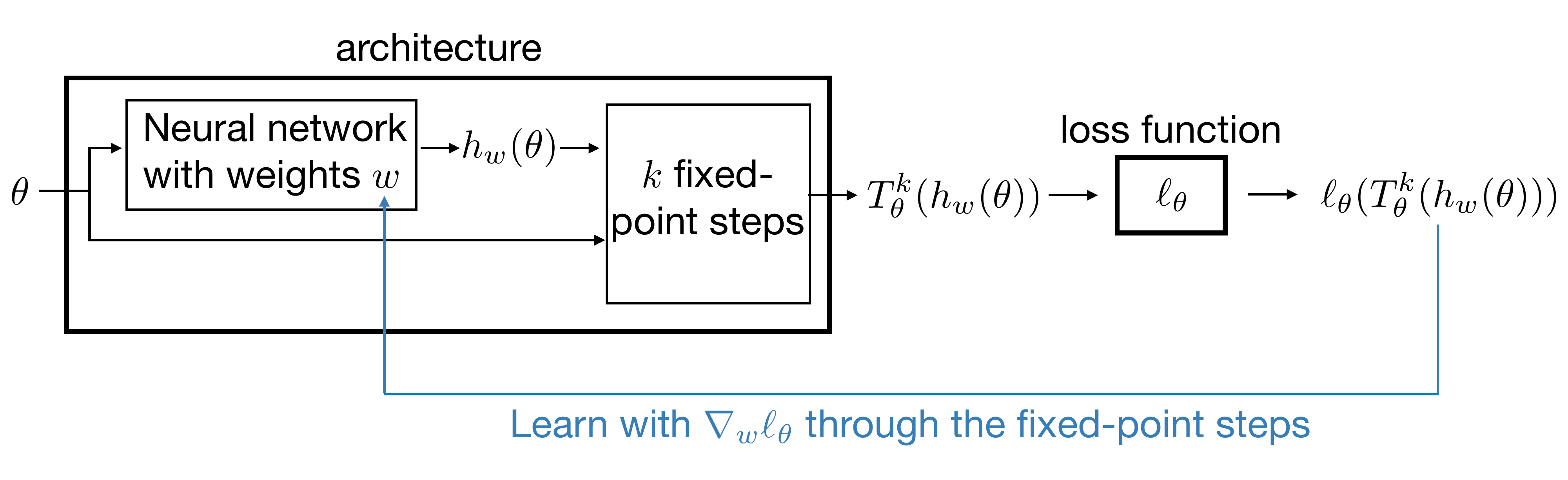}
  \caption{
  Illustration of the learning framework.
  The architecture consists of two modules: a neural network mapping the parameter $\theta$ to a warm start $\hw(\theta)$, and a second module executing $k$ fixed-point iterations starting from $\hw(\theta)$ to obtain the candidate solution $\Tj[k][\hw(\theta)]$.
  The fixed-point steps in the architecture depend on the parameter $\theta$, and have no learnable weights.
  There are two options for the loss function $\ell_{\theta}$: the fixed-point residual loss $\ell^{\rm{fp}}_{\theta}$, or the regression loss $\ell^{\rm{reg}}_{\theta}$.
  We backpropagate from the loss through the fixed-point iterates to learn the neural network weights $w$.}
  \label{fig:nn_architecture}
  \end{figure}
The neural network uses ReLU activation functions defined as $\phi(z) = \max(0, z)$ element-wise.
We let $w=\{W_i\}_{i=1}^L$ be the neural network weights for each layer where $W_i \in \reals^{m_i \times n_i}$.
Our warm-start prediction is computed as
\begin{equation}\label{eq:hw}
  \hw(\theta) = W_L \phi(W_{L-1} \phi(\ldots \phi(W_1 \theta))).
\end{equation}
While we do not explicitly represent bias terms, we can include them by appending a new column to matrices $W_i$ for $i=1,\dots,L,$ and a $1$ to the input vector.
The warm-start prediction $\hw(\theta) \in \reals^p$ feeds into the fixed-point algorithm parametrized by $\theta$.
The second part of our architecture consists of $k$ applications of the operator $\op$ to the warm start $\hw(\theta)$.
The final output is the candidate solution $\Tj[k][\hw(\theta)]$.

\subsection{Loss functions}\label{sec:loss_functions}
\paragraph{Training for $k$ steps.}
We propose two loss functions to analyze the output of our learning to warm-start architecture, $\Tj[k][\hw(\theta)]$.
The first one is the \emph{fixed-point residual loss}
\begin{equation}\label{eq:fp_loss}
  \ell_{\theta}^{\rm{fp}}(z) = \|\op(z) - z\|_2,
\end{equation}
which measures of the distance to convergence of the fixed-point algorithm~\eqref{eq:fp_algo}~\citep[\Sec~2.4]{lscomo}.
The second one is the~\emph{regression loss}
\begin{equation}\label{eq:reg_loss}
  \ell_{\theta}^{\rm{reg}}(z) = \|z - z^\star(\theta)\|_2,
\end{equation}
where $z^\star(\theta)$ is a known (possibly non-unique) fixed-point of $\op$.
The learning problem is
\begin{equation}\label{prob:learning}
  \begin{array}{ll} \mbox{minimize} &  \mathbf{E}_{\theta \sim \paramdist} \ell_{\theta}(\Tk(\hw(\theta))),
  \end{array}
\end{equation}
where $\ell_{\theta}$ is either $\ell_{\theta}^{\rm{fp}}$ or $\ell_{\theta}^{\rm{reg}}$, and $k$ is the number of fixed-point iterations in our architecture.
Note that choosing $k=0$ decouples the learning procedure from the downstream algorithm, thereby making our architecture no longer end-to-end.

It is generally infeasible to evaluate the objective in problem~\eqref{prob:learning} because the distribution~$Q$ is unknown.
Instead, we minimize its empirical estimate over training data hoping to attain generalization to unseen data.
We leverage stochastic gradient descent (SGD) methods to efficiently train the neural network weights, by constructing stochastic approximations of the gradient of the empirical risk~\citep{opt_for_ml}.
To compute such gradient estimates, we use automatic differentiation~\citep{autodiff} techniques to differentiate through the $k$ fixed-point iterations.
We note that due to the inclusion of ReLU layers and projection steps in the fixed-point algorithms (\eg,~the projection step in OSQP), there are non-differentiable mappings in the architecture.
At non-differentiable points, SGD uses subgradients~\citep{rockefeller_variational} to estimate directional derivatives of the loss.
By tailoring the warm-start prediction to the downstream fixed-point algorithm, our framework constitutes an end-to-end learning scheme.


\paragraph{Testing for $t$ steps.}
We now evaluate the learned model with $t$ fixed-point iterations ($t$ possibly different from $k$ used during training) on an unseen parameter~$\theta$.
While we consider two different loss functions for training, we always measure the test performance on unseen problems by the fixed-point residual since it is a standard measure of progress~\citep[\Sec~2.4]{lscomo}.
To analyze the generalization of our architecture, we define the \emph{risk} as the following function of $t$:
\begin{equation}\label{eq:risk}
R^t(w) = \mathbf{E}_{\theta \sim \paramdist}\Lo^{\rm{fp}}(\Tj[t][\hw(\theta)]).
\end{equation}
Since we only access the distribution $Q$ via $N$ samples $\theta_1, \dots, \theta_N$, we define the \emph{empirical risk} over training data as
\begin{equation}\label{eq:emp_risk}
  \hat{R}^t(w) = \frac{1}{N} \sum_{i=1}^N \Loit[\hw(\theta_i)].
  \end{equation}




\section{PAC-Bayes generalization bounds} \label{sec:PAC_Bayes}
In this section, we provide generalization bounds for our approach using the PAC-Bayes framework \citep{Shawe-Taylor_PAC_Bayes, McAllester_PAC_Bayes}.
More specifically, we provide a generalization guarantee on the risk in \Eqn~\eqref{eq:risk} after any number of evaluation steps $t$ ($t$ need not be equal to the number of fixed-point steps $k$ taken during training).

First, we introduce preliminary results and definitions needed for our proofs in \Sec~\ref{subsec:prelims}. In particular, we define the \emph{marginal fixed-point residual}, a key ingredient of our proof technique, which measures the maximum fixed-point residual incurred by a warm start when subjected to a bounded perturbation.
Then, we derive our main generalization bound result, \Thm~\ref{thm:gen_pac}, in \Sec~\ref{subsec:gen_bounds}.
Finally, in \Sec~\ref{subsec:perturb_bounds}, we specialize \Thm~\ref{thm:gen_pac} to three different cases of operators: contractive, linearly convergent, and averaged.

\subsection{Preliminaries}\label{subsec:prelims}
In this subsection, we introduce our marginal fixed-point residual in \Eqn~\eqref{eq:marg_fp} and McAllester's bound in inequality~\eqref{eq:_pac_bayes}.

\paragraph{Marginal fixed-point residual.}
We define the marginal fixed-point residual to be the worst-case fixed-point residual for a warm start subjected to a  bounded perturbation:
\begin{equation}\label{eq:marg_fp}
  \gok(z) = \max_{\|\Delta\|_2 \leq \gamma} \Lot[z + \Delta].
\end{equation}
Similarly, we define the marginal risk and marginal empirical risk in the same way as for the non-marginal case from \Sec~\ref{sec:framework} with
\begin{equation*} 
  R^t_{\gamma}(w) = \mathbf{E}_{\theta \sim \paramdist}\gok(\hw(\theta)) \quad \text{and} \quad \hat{R}^t_{\gamma}(w) = \frac{1}{N} \sum_{i=1}^N \goik(\hw(\theta_i)).
\end{equation*}
Setting $\gamma = 0$ corresponds to the original fixed-point residual and risk functions, \ie, $g_{0,\theta}^t(z) = \Lot[z]$,
$\hat{R}^t_0(w)=\hat{R}^t(w),$ and $R^t_0(w)=R^t(w)$ from Equations~\eqref{eq:risk} and \eqref{eq:emp_risk}.

\paragraph{McAllester's bound.}
The PAC-Bayesian framework provides generalization bounds randomized predictors, as opposed to a learned single predictor.
Randomized predictors are obtained by sampling in a set of basic predictors based on a specific probability distribution~\citep{pac_bayes_intro}.
This is especially useful in our setting because we can manipulate the bounds on the randomized predictors into bounds on our learned predictors.

In our case, $\hw$ from~\Eqn~\eqref{eq:hw} corresponds to the fixed warm-start prediction parameterized by the weights of the neural network $w \in \mathcal{W}$ where $\mathcal{W}$ is a set of possible weights.
We aim to bound $R^t(w)$, the risk after $t$ fixed-point steps from \Eqn~\eqref{eq:risk}, in terms of empirical quantities.
To do so, we consider perturbations of the neural network weights given by the random variable $u$ whose distribution may also depend on the training data.
Now, we have a distribution of predictors $\hwp$, where $w$ is fixed and $u$ is random.
Given a prior distribution $\prior$ over the set of predictors that is independent of the training data, the expected marginal risk of the randomized predictor $\mathbf{E}_u[R^t_{\gamma}(w+u)]$ can be bounded as~\citep{McAllester_simplified}
\begin{equation}\label{eq:_pac_bayes}
  \mathbf{E}_u[R^t_{\gamma}(w + u)] \leq \mathbf{E}_u[\hat{R}^t_{\gamma}(w + u)] + 2 C_{\gamma}(t) \sqrt{\frac{2({\rm KL}(w + u || \prior) + \log(2N/\delta))}{N - 1}},
\end{equation}
with probability at least $1 - \delta$.
Here ${\rm KL}(p||\pi)$ is the KL-divergence between the distributions $p=w + u$ and $\pi$,
\begin{equation*}
  {\rm KL}(p || \pi) = \int_{-\infty}^{\infty} p(x)\log \biggl(\frac{p(x)}{\pi(x)}\biggr).
\end{equation*}
The quantity $\Cgt$ upper bounds the fixed-point residual after $t$ steps, \ie,
\begin{equation*}
    g_{\gamma,\theta}^t(h_{w}(\theta))\le \Cgt,\quad \forall \theta \in \Theta, \; w \in \mathcal{W}.
\end{equation*}
Note that the marginal risk and empirical marginal risk lie in the range $[0,C_{\gamma}(t)]$.
In order to bound $\Cgt$, we will consider predictors where the distance from warm start to the set of fixed-points is upper bounded by $D$:
\begin{equation}\label{eq:D_bound}
    \dist_{\fix \op}(\hw(\theta)) \leq D\quad \forall \theta \in \Theta, \; w \in \mathcal{W}. 
\end{equation}
In \Sec~\ref{subsec:perturb_bounds}, we bound $\Cgt$ in terms of $t$, $\gamma$, $D$, and properties of the operator~ $\op$. 

\subsection{Generalization bounds}\label{subsec:gen_bounds}
In this subsection, we use the marginal fixed-point residual and the McAllester bound from~\Sec~\ref{subsec:prelims} to bound the generalization gap.
We first transform the McAllester bound in \eqref{eq:_pac_bayes}, which provides a generalization bound on the expected marginal risk of the randomized predictor, to a bound on the risk with the following lemma.
\begin{lemma}\label{lemma:marg}
  Let $\hw : \Theta \rightarrow \reals^p$ be any predictor to a warm start learned from the training data such that $g_{\gamma / 2,\theta}^t(\hw(\theta)) \leq \Cgtt, \; \forall \theta \in \Theta$.
  Let $\hw$ be any learned predictor parametrized by $w$ and $\prior$ be any distribution that is independent from the training data.
  Then, for any $\delta,\gamma > 0$, with probability at least $1 - \delta$ over a training set of size~$N$ and for any random perturbation $u$ such that ${\mathbf{P}(\max_{\theta \in \Theta} \|h_{w + u}(\theta) - h_{w}(\theta)\|_2 \leq \gamma / 2) \geq 1 / 2}$ we have
  \begin{equation*}
    R^t(w) \leq \hat{R}^t_{\gamma}(w) + 4 \Cgtt \sqrt{\frac{{\rm KL}(w + u|| \prior) + \log(6N / \delta)}{N - 1}}.
  \end{equation*}
\end{lemma}
See Appendix~\ref{proof:marg} for the proof.
In the above expression, $w$ is fixed and $u$ is a random variable.
This lemma bears resemblance to \citet[\Lem~1]{PAC_Bayes_nn}, and the proof is nearly identical.
Next, we use \Lem~\ref{lemma:marg} to obtain generalization bounds with our main theorem.

\begin{theorem}\label{thm:gen_pac}
    Assume that $\|\theta\|_2 \leq B$ for all $\theta \in \Theta$.
    Let $\hw: \Theta \rightarrow \reals^p$ be an $L$-layer neural network with ReLU activations where $g_{\gamma / 2,\theta}^t(\hw(\theta)) \leq \Cgtt, \; \forall \theta \in \Theta$.
    Let $c = B^2 L^2 \bh \log(L \bh) \Pi_{j=1}^L \|W_j\|_2^2 \sum_{i=1}^L \|W_i\|_F^2 / \|W_i\|_2^2$ and let $\bar{h} = \max_i n_i$ be the largest number of output units in any layer.
  Then for any $\delta,\gamma > 0$ with probability at least $1 - \delta$ over a training set of size $N$,
  \begin{equation}\label{eq:gen_pac}
    R^t(w) \leq \hat{R}^t_{\gamma}(w) + \begin{cases}
      \mathcal{O} \Biggl(\Cgtt \sqrt{\frac{c + \log(\frac{L N}{\delta})}{\gamma^2 N}}\Biggr) & \text{if  } \Pi_{j=1}^L \|W_j\|_2 \geq \frac{\gamma} {2 B}\\
      \Cgtt \sqrt{\frac{\log(1 / \delta)}{2N}} & \text{else.}
    \end{cases}
  \end{equation}
\end{theorem}
See Appendix~\ref{proof:gen_pac} for the proof.
With \Thm~\ref{thm:gen_pac}, we bound the risk in terms of the empirical marginal risk and a penalty term.
The main case is when the weights are sufficiently large: $\Pi_{j=1}^L \|W_j\|_2 \geq \gamma / (2 B)$.
In this case, we use the PAC-Bayesian framework to provide the generalization bound.
We directly use the perturbation bound from \citet[\Lem~2]{PAC_Bayes_nn} which bounds the change in the warm start $\hw(\theta)$ with respect to the change in the neural network weights $w$.
In the other case, if $\Pi_{j=1}^L \|W_j\|_2 \leq \gamma / (2 B)$, then the warm start $\hw(\theta)$ is close to the zero vector.
Here, we leverage Hoeffding's inequality to get the generalization bound.

As $t \rightarrow \infty$, the generalization gap in \Thm~\ref{thm:gen_pac} approaches zero since $\Cgtt$ goes to zero.
Intuitively, this happens because the algorithm is run until convergence.
On the other hand, as $N \rightarrow \infty$, the second term in each of the cases disappears and the generalization gap becomes the difference between the marginal empirical risk and the risk for a fixed $\gamma$.
Our bounds also generalizes the setting where the warm start is not learned.
Setting all of the weights to zero corresponds to warm-starting every problem from the zero vector.
In this case, with high probability, $R^t(0) \leq \hat{R}^t(0) + \Cgtt \sqrt{\log(1 / \delta) / (2N)}$.






\subsection{Bounding the empirical marginal risk}\label{subsec:perturb_bounds}
Theorem~\ref{thm:gen_pac} bounds the risk $R^t(w)$ in terms of the empirical marginal risk $\hat{R}^t_{\gamma}(w)$ plus a penalty term.
In this subsection, we use operator theory to bound two things: i) $\hat{R}^t_{\gamma}(w)$, thus removing the dependency on the marginal component, and ii) $\Cgtt$ in terms of $D$ given by \Eqn~\eqref{eq:D_bound}.
We first assume that the operator $\op$ is non-expansive, which is a common characteristic of solving convex problems~\citep{mon_primer}.
\begin{definition}[Non-expansive operator]
  An operator $T$ is non-expansive if
  \begin{equation*}
    \|Tx - Ty\|_2 \leq \|x - y\|_2, \quad \forall x, y \in \dom T.
  \end{equation*}
\end{definition}
Since non-expansiveness is not enough to guarantee convergence, we break our analysis into three different cases of fixed-point operators which converge: contractive in \Sec~\ref{subsec:contractive}, linearly convergent in \Sec~\ref{subsec:lin_conv}, and averaged in \Sec~\ref{subsec:averaged}. 
By using the different properties of each, we can bound the marginal fixed-point residual after $t$ steps, $g_{\gamma, \theta}^t(z)$ defined in~\eqref{eq:marg_fp}.
Since the empirical marginal risk is the average of these marginal fixed-point residuals, we can remove the dependence on the empirical marginal risk in \Thm~\ref{thm:gen_pac}.
The sets of the three different types of operators are not mutually exclusive as seen in  the set relationships depicted in Figure~\ref{fig:operators}.
The contractive case provides the strongest bounds, followed by the linearly convergent case, and then the averaged case.

\begin{figure}[!h]
  \vspace*{-4mm}
  \centering
    \includegraphics[width=0.5 \linewidth]{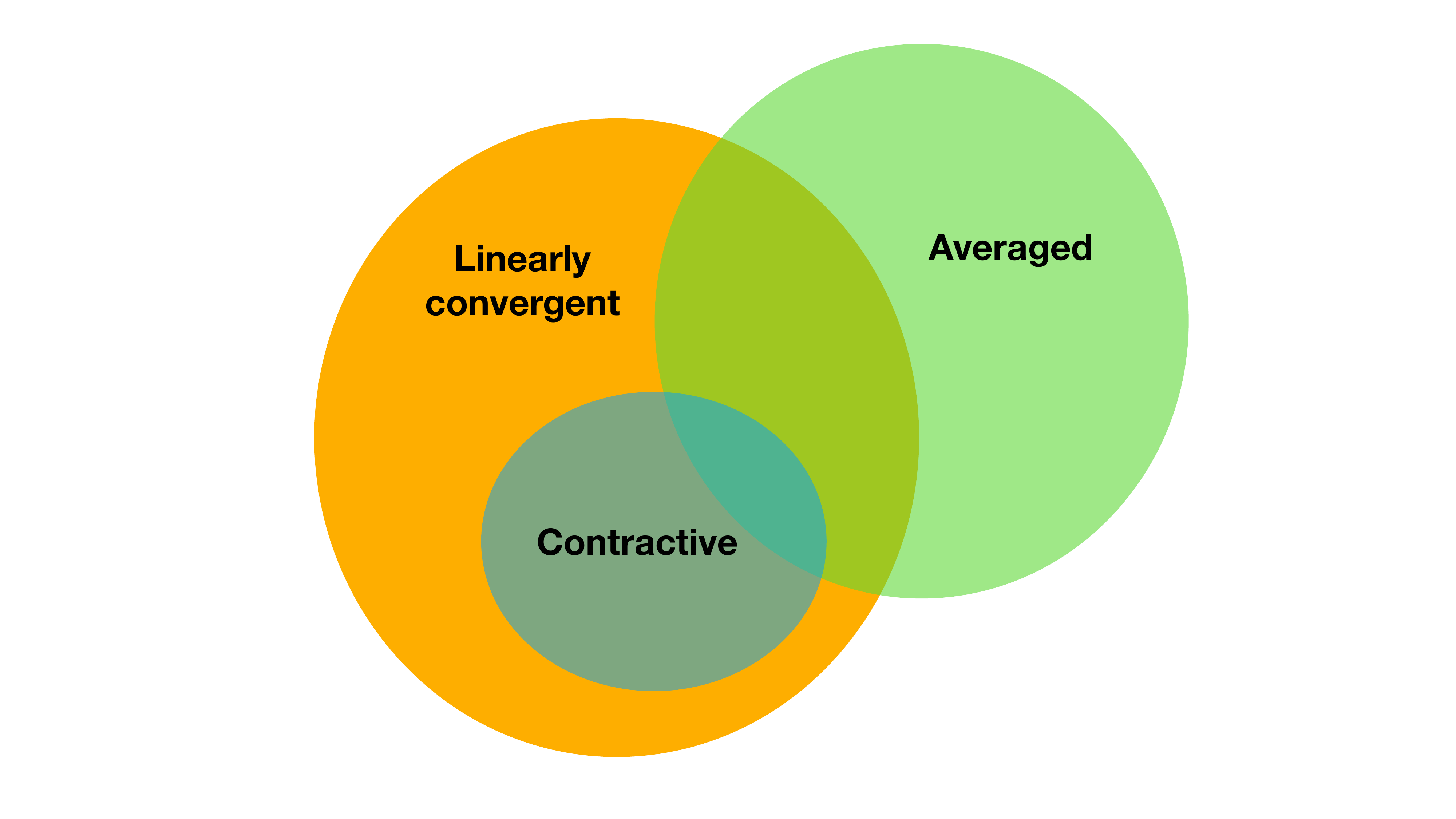}
    \caption{The set relationship between the different types of operators we consider in this section.}
    \label{fig:operators}
\end{figure}
To help in the subsequent analysis, we define the following functions which give the distance to optimality and marginal distance to optimality:
\begin{equation}\label{eq:r_f_def}
  r_{\theta}(z) = \dist_{\fix \op}(z), \quad f_{\gamma, \theta}^t(z) = \max_{\|\Delta\|_2 \leq \gamma} r_{\theta}(\op^t(z + \Delta)).
\end{equation}
We give the following lemma to relate the fixed-point residual to the distance to optimality.
\begin{lemma}\label{lem:fp_reg_bound}
For any non-expansive operator, $\op$,
    \begin{equation*}
        \ell^{\rm fp}_{\theta}(z) \leq 2r_{\theta}(z).
    \end{equation*}
\end{lemma}
See Appendix~\ref{sec:fp_2_reg_proof} for the proof.


\subsubsection{Contractive operators}\label{subsec:contractive}
We first consider contractive operators which give the strongest perturbation bounds.
\begin{definition}[$\beta$-contractive operator]\label{def:contractive}
  An operator $T$ is $\beta$-contractive for $\beta \in (0, 1)$ if
  \begin{equation*}
    \|Tx - Ty\|_2 \leq \beta \|x - y\|_2 \quad \forall x, y \in \dom T.
  \end{equation*}
\end{definition}
If $\op$ is $\beta$-contractive, then
\begin{equation}\label{eq:perturbation_contractive}
    g_{\gamma, \theta}^t(z) \leq \Lot[z] + 2 \beta^t \gamma,
  \end{equation}
which follows from $\ell^{\rm fp}_{\theta}(\Tj[t][\cdot])$ being $2 \beta^t$-Lipschitz~\citep[Appendix A.1]{l2ws_l4dc}.
In the contractive case, we remove the marginal risk dependency with the following corollary.
\begin{corollary}\label{cor:contractive}
    We define $B$ and $\bar{h}$ as in \Thm~\ref{thm:gen_pac}.
    Let $\op$ be $\beta$-contractive for any $\theta \in \Theta$. 
    Let $\hw$ be an $L$-layer neural network with ReLU activations such that~\eqref{eq:D_bound} holds with bound $D$. 
    Let $c = B^2 L^2 \bh \log(L \bh) \Pi_{j=1}^L \|W_j\|_2^2 \sum_{i=1}^L \|W_i\|_F^2 / \|W_i\|_2^2$.
    Then for any $\delta,\gamma > 0$ with probability $\geq 1 - \delta$ over a training set of size $N$,
  \begin{equation*}
    R^t(w) \leq \hat{R}^t(w) + 2 \beta^t \gamma + \begin{cases}
      \mathcal{O} \Biggl(\beta^t (D + \frac{\gamma} {2})  \sqrt{\frac{c  + \log(\frac{L N}{\delta})}{\gamma^2 N}}\Biggr) & \text{if  } \Pi_{j=1}^L \|W_j\|_2 \geq \frac{\gamma} {2 B}\\
      2 \beta^t (D + \frac{\gamma} {2}) \sqrt{\frac{\log(1 / \delta)}{2N}} & \text{else}
    \end{cases}
  \end{equation*}
\end{corollary}
\begin{proof}
    We remove the marginal dependence by applying inequality~\eqref{eq:perturbation_contractive} to get
    \begin{equation*}
        \hat{R}^t_{\gamma}(w) = \frac{1}{N} \sum_{i=1}^N   g_{\gamma, \theta}^t(\hw(\theta_i)) \leq 2 \beta^t \gamma + \frac{1}{N} \sum_{i=1}^N \ell^{\rm fp}_{\theta}(\op^t(\hw(\theta_i))) = \hat{R}^t(w) + 2 \beta^t \gamma.
    \end{equation*}
We bound the worst-case fixed-point residual as $\Cgtt \leq 2 \beta^t (D + \gamma / 2)$ which comes from $C_0(t) \leq 2 \beta^tD$~\citep[Appendix A.1]{l2ws_l4dc} and the inequality
\begin{equation}\label{eq:gen_tri}
    \dist_{\fix \op}(\hw(\theta) + \Delta) \leq \|\Pi_{\fix \op}(\hw(\theta)) - (\hw(\theta) + \Delta)\|_2 \leq \dist_{\fix \op}(\hw(\theta)) +  \|\Delta\|_2.
\end{equation}
Here, $\Pi_{\fix \op}$ is the projection on the set $\fix \op$.
The first inequality in~\eqref{eq:gen_tri} uses the definition of the distance function $\dist_{\fix \op}$, and the second uses the triangle inequality.
\end{proof}

\subsubsection{Linearly convergent operators}\label{subsec:lin_conv}
Now, we consider a broader category of operators, linearly convergent operators.
\begin{definition}[$\beta$-linearly convergent operator]\label{def:lin_conv}
  An operator $T$ is $\beta$-linearly convergent for $\beta \in [0, 1)$ if
  \begin{equation*}
    \dist_{\fix T}(Tx) \leq \beta \dist_{\fix T}(x) \quad \forall x \in \dom T.
  \end{equation*}
\end{definition}
If the operator, $\op$, is not contractive, then we get the weaker property that $\ell^{\rm fp}_{\theta}(\Tj[t][\cdot])$ is $2$-Lipschitz.
To provide tighter perturbation bounds, we first establish the following lemma.

\begin{lemma}\label{lemma:nonexp_r}
  For any non-expansive operator $\op$ and for any $t \geq 0$,
  \begin{equation*}
    |\rok[z] - \rok[w]| \leq 2\|z-w\|_2.
  \end{equation*}
\end{lemma}
See Appendix~\ref{proof:nonexp_r} for the proof.
We now use \Lem~\ref{lemma:nonexp_r} and the linear convergence guarantee, \Def~\ref{def:lin_conv}, to bound $\gok$ in terms of empirical quantities.

\begin{lemma}\label{lem:perturbation_lin_conv}
  Assume that $\op$ is $\beta$-linearly convergent where $\beta \in (0, 1)$.
  Then the following bounds hold for all $t \geq 0$:
  \begin{equation*}\label{eq:perturbation_lin_conv_1}
    f_{\gamma,\theta}^{t}(z) \leq r_{\theta}(\op^{t}(z)) + 2\gamma, \quad f_{\gamma,\theta}^{t+1}(z) \leq \beta f_{\gamma,\theta}^t(z)
  \end{equation*}
  \begin{equation*}
    g_{\gamma,\theta}^t(z) \leq 2 f_{\gamma,\theta}^t(z).
  \end{equation*}
\end{lemma}


\begin{proof}
  The inequality $f_{\gamma,\theta}^{t+1}(z) \leq \beta f_{\gamma,\theta}^t(z)$ comes from
  \begin{equation*}
      f_{\gamma,\theta}^{t+1}(z) = r_{\theta}(\op^{t+1}(z + \Delta^\star)) \leq \beta r_{\theta}(\op^{t}(z + \Delta^\star)) \leq \beta f_{\gamma,\theta}^t(z),
  \end{equation*}
  where $\|\Delta^\star\|_2 \leq \gamma$ is the maximizer to $f_{\gamma,\theta}^{t+1}(z)$.
  The first inequality comes from \Def~\ref{def:lin_conv} and the second from \eqref{eq:r_f_def}.
  The inequality $f_{\gamma,\theta}^{t+1}(z) \leq \rokp[z] + 2 \|\Delta\|_2$ in \Lem~\ref{lem:perturbation_lin_conv} follows from \Lem~\ref{lemma:nonexp_r}.
  The final inequality in \Lem~\ref{lem:perturbation_lin_conv} is derived as follows:
  \begin{equation*}
      g_{\gamma,\theta}^t(z) = \ell^{\rm fp}_{\theta}(\op^t(z + \Delta^\star)) \leq 2 r_{\theta} (\op^t(z + \Delta^\star)) \leq 2 f_{\gamma,\theta}^t(z).
  \end{equation*}
  Here, $\|\Delta^\star\|_2 \leq \gamma$ is the maximizer for $g_{\gamma,\theta}^t(z)$, and \Lem~\ref{lem:fp_reg_bound} gives the first inequality.
\end{proof}
Using \Lem~\ref{lem:perturbation_lin_conv}, we can bound the marginal empirical risk for the linearly convergent case.
For the $\beta$-linearly convergent case, $\Cgtt$ is bounded by $2 \beta^t (D + \gamma / 2)$ which uses~\eqref{eq:gen_tri} and $ C_0(t) \leq 2 r_{\theta}(\op^t(z)) \leq 2 \beta^t D$.
The first inequality comes from \Lem~\ref{lem:fp_reg_bound} and the second inequality follows from \Def~\ref{def:lin_conv}.

\subsubsection{Averaged operators}\label{subsec:averaged}
Lastly, we consider the averaged operator case which in general gives sublinear convergence.
\begin{definition}[$\alpha$-averaged operator]\label{def:averaged}
  An operator $T$ is $\alpha$-averaged for $\alpha \in (0, 1)$ if there exists a non-expansive operator $R$ such that $T = (1 - \alpha) I + \alpha R$.
\end{definition}
\begin{lemma}\label{lem:averaged}
  Let $\op$ be an $\alpha$-averaged.
  Then the following bound holds:
  \begin{equation*}
    g_{\gamma,\theta}^{t}(z) \leq \min_{j=0, \dots, t} \sqrt{\frac{\alpha}{(1 - \alpha) (t-j+1)}} (r_{\theta}(\Tj[j][z]) + 2 \gamma) \quad \text{ for } t \geq 0 .
  \end{equation*}
\end{lemma}
\begin{proof}
Let $\bar{\alpha}_{t,j} = \sqrt{\frac{\alpha}{(1 - \alpha) (t-j+1)}}$.
There exists $\|\Delta^\star\|_2 \leq \gamma$ such that the equality below holds by definition of the marginal fixed-point residual.
    \begin{equation}\label{eq:avged_proof_eq}
        g_{\gamma,\theta}^t(z) = \ell^{\rm fp}_{\theta}(\op^t(z + \Delta^\star)) \leq \bar{\alpha}_{t,j} (r_{\theta}(\Tj[j][z + \Delta^\star])) \leq \bar{\alpha}_{t,j} f_{\gamma,\theta}^j(z) \leq \bar{\alpha}_{t,j} (r_{\theta}(\Tj[j][z]) + 2 \gamma)
    \end{equation}
The three inequalities comes from~\citet[\Thm~1]{lscomo}, the definition of $f_{\gamma,\theta}^j(z)$ in~\eqref{eq:r_f_def}, and \Lem~\ref{lem:perturbation_lin_conv} respectively.
\Eqn~\eqref{eq:avged_proof_eq} holds for all $0 \leq j \leq t$.
\end{proof}
Using \Lem~\ref{lem:averaged}, we can bound the marginal empirical risk for the averaged case.
We bound the worst-case marginal fixed-point residual with $\Cgtt \leq \sqrt{\alpha / ((1 - \alpha)(t+1))}(D + \gamma)$ which follows from \Lem~\ref{lem:averaged} by letting $z = \hw(\theta)$ and $j=0$.
Then the inequality holds for every $\theta \in \Theta$.


\section{Choosing the right architecture}\label{sec:pick_k}
In this section, we discuss how the number of fixed-point steps the model is trained on, $k$, and the loss function affect performance.

\subsection{Bounds on the fixed-point residual for $t$ evaluation steps}\label{sec:gen_iters}
In this subsection, we derive bounds on the fixed-point residual after $t$ steps, $\Lot$, in terms of the loss after $k$ steps, $\ell_{\theta}(\op^k(z))$ where $k < t$.
A summary of these results is given in Table~\ref{tab:gains_summary} where we provide the bound for each of the two loss functions.
The bounds in Table~\ref{tab:gains_summary} using the fixed-point residual loss in the denominator are given by either applying the definition of contractiveness in the contractive case or non-expansiveness in the other two cases.
To get the bounds in Table~\ref{tab:gains_summary} when using the regression loss in the denominator, we first establish the inequality
\begin{equation}\label{eq:fp_reg_bound}
  \ell_{\theta}^{\rm{fp}}(z) \leq 2 \ell_{\theta}^{\rm{reg}}(z),
\end{equation}
for any non-expansive operator $\op$.
This result follows from \Lem~\ref{lem:fp_reg_bound} since $r_{\theta}(z) \leq \ell^{\rm reg}_{\theta}(z)$.
The results in the contractive and linearly convergent cases follow from applying the definition of each and inequality~\eqref{eq:fp_reg_bound}.
In the averaged case, we directly apply \citet[\Thm~1]{lscomo}.
Unless the operator is contractive, the results from Table~\ref{tab:gains_summary} indicate that stronger bounds can be obtained from using the regression loss.




\begin{table}[!h]
    \centering
    \small
    \caption{Bounds for the ratios of testing at $t$ steps and training at $k$ steps.
    Here, we bound the ratio of the fixed-point residual after $t$ steps and the loss after $k$ steps where $t > k$.
    The value in the table provides the bound, \eg, for a $\beta$-contractive operator, $\Lot / \ell^{\rm{reg}}_{\theta}(\Tj) \leq 2 \beta^{t-k}$.}
    \begin{tabular}{lcc}
    \toprule
        {\bf Operator}
        & $\displaystyle \frac{\Lot}{\Lok}$ & $\displaystyle \frac{\Lot}{\ell^{\rm{reg}}_{\theta}(\Tj)}$  \\
        \midrule
        $\beta$-contractive & $\beta^{t - k}$& $2 \beta^{t-k}$\\
        $\beta$-linearly convergent & $1$ & $2 \beta^{t-k}$\\
        $\alpha$-averaged & $1$& $\sqrt{\frac{\alpha}{(1 - \alpha) (t - k + 1)}}$\\
        \bottomrule
    \end{tabular}
    \label{tab:gains_summary}
\end{table}

\subsection{Training for the fixed-point residual vs regression loss}\label{sec:fp_vs_reg}
The fixed-point residual~\eqref{eq:fp_loss} and regression~\eqref{eq:reg_loss} losses, align with the main distinction of learning methods mentioned in \citet[\Sec~2.2]{amos_tutorial} which splits between learning strategies that penalize suboptimality directly and those that penalize the distance to known ground truth solutions.
The primary advantages of using our fixed-point residual loss are twofold: i) there is no need to compute a ground truth solution $z^\star(\theta)$ for each problem instance before training, and ii) the loss exactly corresponds to the evaluation metric, the fixed-point residual.
On the other hand, there are two main advantages to using the regression loss: i) the regression loss uses the global information of the ground truth solution $z^\star(\theta)$, while the fixed-point residual loss exploits only local information, and ii) as mentioned in \Sec~\ref{sec:gen_iters}, stronger bounds on future iterations can be obtained when using the regression loss.


\section{Numerical experiments}\label{sec:numerical_experiments}
We now illustrate our method on examples of fixed-point algorithms from Table~\ref{table:fp_algorithms}.
We implemented our architecture in the JAX library~\citep{jax} using the Adam~\citep{adam} optimizer to train.
We use $10000$ training problems and evaluate on $1000$ test problems for the examples except the first one in~\Sec~\ref{sec:gd_experiments}.
In our examples, we conduct a hyperparameter sweep over learning rates of either $10^{-3}$ or $10^{-4}$, and architectures with $0,1,$ or $2$ layers with $500$ neurons each.
We decay the learning rate by a factor of $5$ when the training loss fails to decrease over a window of $10$ epochs.
All computations were run on the Princeton HPC Della Cluster and each example could be trained under $5$ hours.
The code to reproduce our results is available at
\begin{equation*}
\text{\url{https://github.com/stellatogrp/l2ws_fixed_point}}.
\end{equation*}

\paragraph{Baselines.}
We compare our learned warm start, for both the fixed-point residual loss and the regression loss functions, against the following initialization approaches:
\begin{description}
	\item[Cold start.] We initialize the fixed-point algorithm for a test problem with parameter $\theta$ with the prediction $h_{w_{\rm cs}}(\theta)$ where $w_{\rm cs}$ has been randomly initialized.
	\item[Nearest-neighbor warm start.] The nearest-neighbor warm start initializes the test problem with an optimal solution of the nearest of the training problems measured by distance in terms of its parameter $\theta \in \reals^d$.
    In most of our examples, the parametrized problems are sufficiently far apart so that the nearest-neighbor initializations do not significantly improve upon the cold start.
\end{description}
In every experiment, we plot the average of the fixed-point residuals of the test problems for varying $t$ as defined in \Sec~\ref{sec:loss_functions}.
Additionally, we plot the average \emph{gain} of each initialization relative to the cold start across the test problems.
This gain for a given parameter $\theta$ corresponds to the ratio
\begin{equation*}
  \frac{\Lot[h_{w_{\rm cs}}(\theta)]}{\Lot[\hw(\theta)]},
\end{equation*}
where $\hw$ is the initialization technique in question and $h_{w_{cs}}$ is the cold-start predictor described above.
Importantly, we code exact replicas of the OSQP and SCS algorithms in JAX.
This allows us to input the learned warm starts into the corresponding C implementations; moreover, we report the solve times in milliseconds to reach various tolerances for the experiments we run with OSQP in \Sec~\ref{sec:osqp_experiments} and SCS in \Sec~\ref{sec:scs_experiments}.

\subsection{Gradient descent}\label{sec:gd_experiments}
\subsubsection{Unconstrained QP}
We first consider a stylized example to illustrate why unrolling fixed-point steps can significantly improve over a decoupled approach, where $k=0$.
Consider the problem
\begin{equation*}
  \begin{array}{ll}
  \label{prob:gd_example}
  \mbox{minimize} & (1 / 2)z^T P z + c^T z,
  \end{array}
\end{equation*}
where $P \in \symm^n_{++}$, and $c \in \reals^n$ are the problem data and $z \in \reals^n$ is the decision variable.
The parameter is $\theta = c$.

\paragraph{Numerical example.}
We consider a small example where $n=20$.
We have a single hidden layer with $10$ neurons, and $100$ training problems.
Let $P \in \symm_{++}^{n}$ be a diagonal matrix where the first $10$ diagonals take the value $100$ and the last ten take the value of $1$.
Let $\theta = c \in \reals^n$.
Here, each $\theta_i$ is sampled according to the uniform distribution $\psi_i \mathcal{U}[-10,10]$, where $\psi_i = 10000 \text{ if } i \leq 10 \text{ else } 1$.
The idea is that the first $10$ indices of the optimal solution $z^\star(\theta)$ vary much more than the last $10$, but the first $10$ indices of $z$ will converge much faster.

\paragraph{Results.}
Figure~\ref{fig:unconstrained_qp_results} and Table~\ref{tab:unconstrained_qp_table} show the convergence behavior of our method.
The decoupled approaches prioritize minimizing the error to predict the first $10$ indices and fail to improve on the cold start.
By unrolling these gradient steps, our learning framework with $k > 0$ is able to adapt the warm start to take advantage of the downstream algorithm.
These gains remain constant as the number of evaluation steps increases.

\begin{figure}[!h]
  \centering
    \includegraphics[width=\figsize\linewidth]{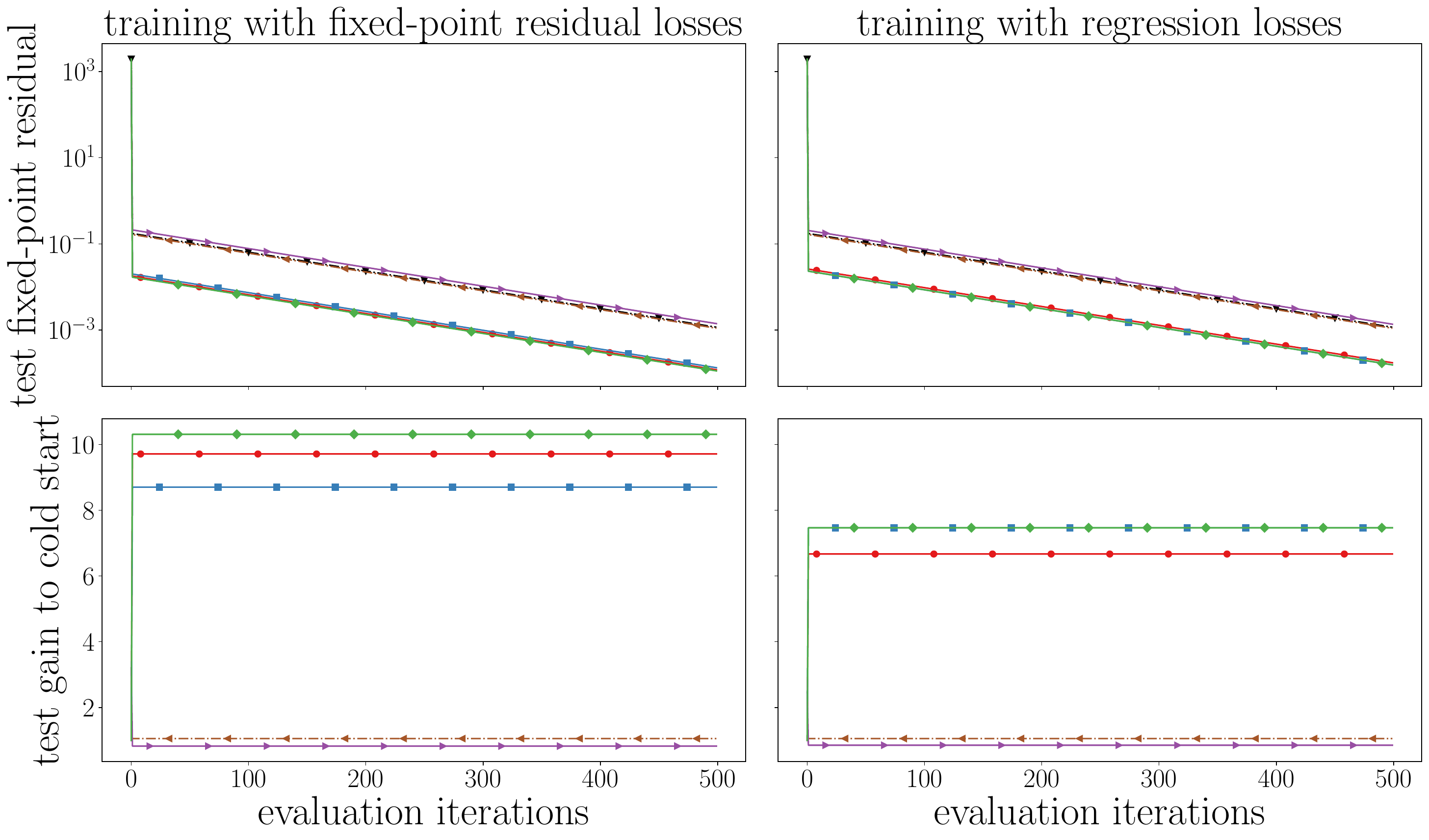}
    \\
    \legend\\
    \caption{Unconstrained QP results.
      All of the learned warm starts provide large improvements over the cold start and nearest neighbor initializations except for the ones learned with $k=0$.
    }
    \label{fig:unconstrained_qp_results}
\end{figure}

\begin{table}[!h]
  \centering
  \small
    \renewcommand*{\arraystretch}{1.0}
  \caption{Unconstrained QP.
  }
  \label{tab:unconstrained_qp_table}
  \vspace*{-3mm}
  \begin{tabular}{l}
  \iters
  \end{tabular}
  \adjustbox{max width=\textwidth}{
    \begin{tabular}{lllllllllllll}
    Fp res.&
    \begin{tabular}{@{}c@{}}Cold \\ Start\end{tabular}
    &
    \begin{tabular}{@{}c@{}}Nearest \\ Neighbor\end{tabular}
    &
    \begin{tabular}{@{}c@{}}Fp \\$k=0$ \end{tabular}
    &
    \begin{tabular}{@{}c@{}}Fp \\$k=5$ \end{tabular}
    &
    \begin{tabular}{@{}c@{}}Fp \\$k=15$ \end{tabular}
    &
    \begin{tabular}{@{}c@{}}Fp \\$k=30$ \end{tabular}
    &
    \begin{tabular}{@{}c@{}}Fp \\$k=60$ \end{tabular}
    &
    \begin{tabular}{@{}c@{}}Reg \\$k=0$ \end{tabular}
    &
    \begin{tabular}{@{}c@{}}Reg \\$k=5$ \end{tabular}
    &
    \begin{tabular}{@{}c@{}}Reg \\$k=15$ \end{tabular}
    &
    \begin{tabular}{@{}c@{}}Reg \\$k=30$ \end{tabular}
    &
    \begin{tabular}{@{}c@{}}Reg \\$k=60$ \end{tabular} \\
    \midrule
    \begin{filecontents*}{./data/iters_red/unconstrained_qp.csv}
,accuracies,cold_start_iters,cold_start_red,nearest_neighbor_iters,nearest_neighbor_red,obj_k0_iters,obj_k0_red,obj_k5_iters,obj_k5_red,obj_k15_iters,obj_k15_red,obj_k30_iters,obj_k30_red,obj_k60_iters,obj_k60_red,reg_k0_iters,reg_k0_red,reg_k5_iters,reg_k5_red,reg_k15_iters,reg_k15_red,reg_k30_iters,reg_k30_red,reg_k60_iters,reg_k60_red
0,0.1,57,0,51,0.11,76,-0.33,\bf{1},\bf{0.98},\bf{1},\bf{0.98},\bf{1},\bf{0.98},\bf{1},\bf{0.98},73,-0.28,\bf{1},\bf{0.98},\bf{1},\bf{0.98},\bf{1},\bf{0.98},\bf{1},\bf{0.98}
1,0.01,286,0,280,0.02,305,-0.07,59,0.79,70,0.76,55,0.81,\bf{50},\bf{0.83},302,-0.06,97,0.66,86,0.7,86,0.7,86,0.7
2,0.001,515,0,510,0.01,534,-0.04,289,0.44,299,0.42,285,0.45,\bf{279},\bf{0.46},531,-0.03,326,0.37,315,0.39,315,0.39,315,0.39
3,0.0001,744,0,739,0.01,763,-0.03,518,0.3,529,0.29,514,0.31,\bf{509},\bf{0.32},760,-0.02,555,0.25,544,0.27,544,0.27,544,0.27
\end{filecontents*}
\csvreader[head to column names, late after line=\\]{./data/iters_red/unconstrained_qp.csv}{
    accuracies=\colA,
    cold_start_iters=\colC,
    nearest_neighbor_iters=\colD,
    obj_k0_iters=\colE,
    obj_k5_iters=\colF,
    obj_k15_iters=\colG,
    obj_k30_iters=\colH,
    obj_k60_iters=\colI,
    reg_k0_iters=\colJ,
    reg_k5_iters=\colK,
    reg_k15_iters=\colL,
    reg_k30_iters=\colM,
    reg_k60_iters=\colN,
    }{\colA & \colC & \colD & \colE & \colF & \colG & \colH & \colI & \colJ & \colK & \colL & \colM & \colN}
    \bottomrule
  \end{tabular}}
  \begin{tabular}{l}
    \reduction
  \end{tabular}
  \adjustbox{max width=\textwidth}{
    \begin{tabular}{lllllllllllll}
    Fp res.&
    \begin{tabular}{@{}c@{}}Cold \\ Start\end{tabular}
    &
    \begin{tabular}{@{}c@{}}Nearest \\ Neighbor\end{tabular}
    &
    \begin{tabular}{@{}c@{}}Fp \\$k=0$ \end{tabular}
    &
    \begin{tabular}{@{}c@{}}Fp \\$k=5$ \end{tabular}
    &
    \begin{tabular}{@{}c@{}}Fp \\$k=15$ \end{tabular}
    &
    \begin{tabular}{@{}c@{}}Fp \\$k=30$ \end{tabular}
    &
    \begin{tabular}{@{}c@{}}Fp \\$k=60$ \end{tabular}
    &
    \begin{tabular}{@{}c@{}}Reg \\$k=0$ \end{tabular}
    &
    \begin{tabular}{@{}c@{}}Reg \\$k=5$ \end{tabular}
    &
    \begin{tabular}{@{}c@{}}Reg \\$k=15$ \end{tabular}
    &
    \begin{tabular}{@{}c@{}}Reg \\$k=30$ \end{tabular}
    &
    \begin{tabular}{@{}c@{}}Reg \\$k=60$ \end{tabular} \\
    \midrule
    \csvreader[head to column names, late after line=\\]{./data/iters_red/unconstrained_qp.csv}{
    accuracies=\colA,
    cold_start_red=\colC,
    nearest_neighbor_red=\colD,
    obj_k0_red=\colE,
    obj_k5_red=\colF,
    obj_k15_red=\colG,
    obj_k30_red=\colH,
    obj_k60_red=\colI,
    reg_k0_red=\colJ,
    reg_k5_red=\colK,
    reg_k15_red=\colL,
    reg_k30_red=\colM,
    reg_k60_red=\colN,
    }{\colA & \colC & \colD & \colE & \colF & \colG & \colH & \colI & \colJ & \colK & \colL & \colM & \colN}
    \bottomrule
  \end{tabular}}
\end{table}

\subsection{Proximal gradient descent}\label{sec:proxgd_experiments}
\subsubsection{Lasso}
We first consider the lasso problem
\begin{equation*}
  \begin{array}{ll}
  \label{prob:lasso}
  \mbox{minimize} & (1 / 2)\|Az-b\|_2^2 + \lambda \|z\|_1,
  \end{array}
\end{equation*}
where $A \in \reals^{m \times n}$, $b \in \reals^m$, and $\lambda \in \reals_{++}$ are problem data and $z \in \reals^n$ is the decision variable.
The parameter here is $\theta = b$.

\paragraph{Numerical example.}
We generate $A \in \reals^{500 \times 500}$ with i.i.d standard Gaussian entries and pick $\lambda = 10$.
We sample each $b$ vector from the uniform distribution $\mathcal{U}[0,30]$.

\paragraph{Results.}
Figure~\ref{fig:lasso_results} and Table~\ref{tab:lasso_table} show the convergence behavior of our method.
While most of the learned warm starts significantly improve upon the baselines, the warm starts learned with $k=5$ and the regression loss perform the best.

\begin{figure}[!h]
  \centering
    \includegraphics[width=\figsize\linewidth]{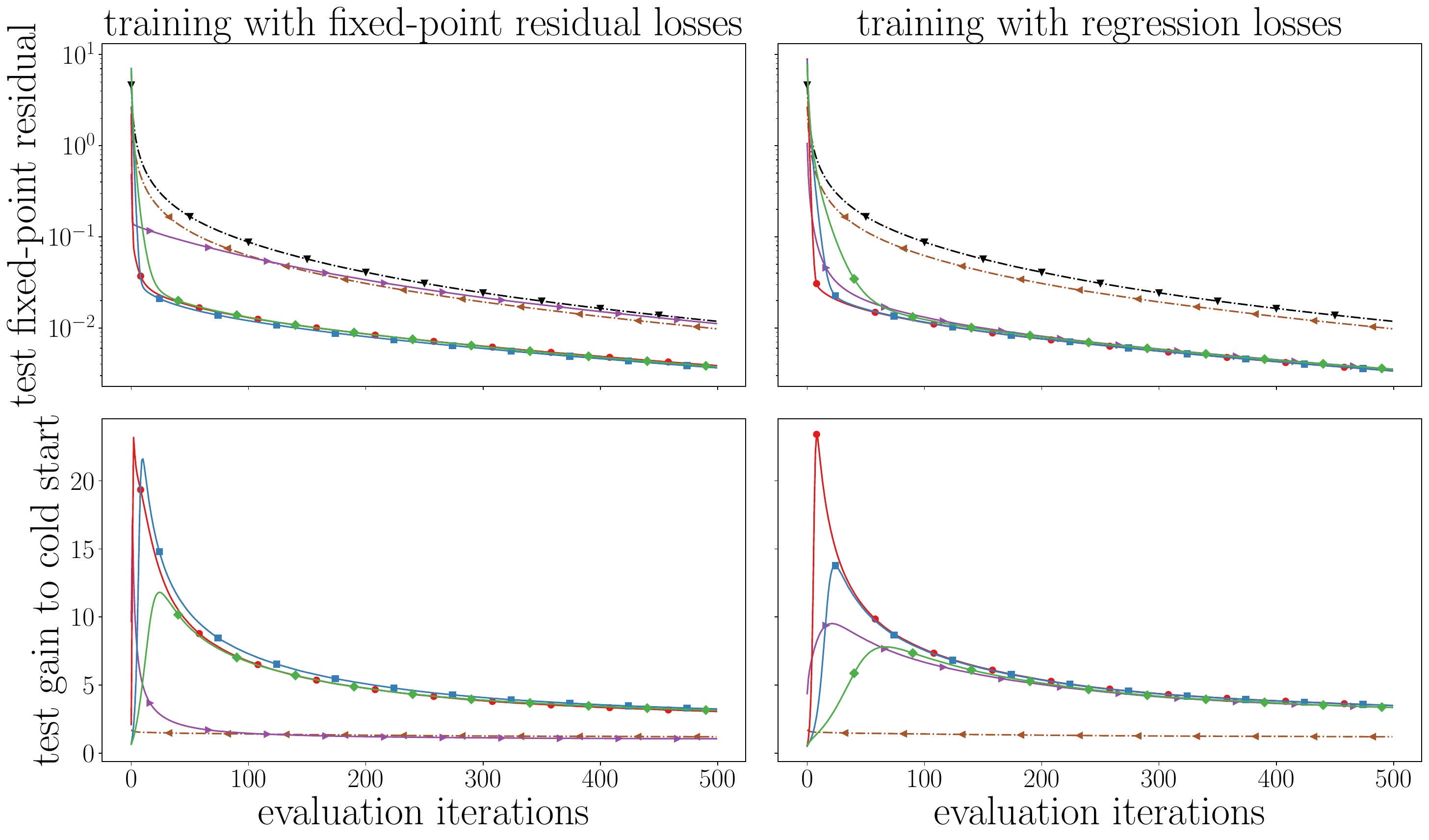}
    \\
  \legend\\
    \caption{Lasso results.
      All of the learned warm starts except for $k=0$ with fixed-point residual loss significantly improve upon the baselines for any number of steps up to $500$.
    }
    \label{fig:lasso_results}
\end{figure}

\begin{table}[!h]
  \centering
  \small
    \renewcommand*{\arraystretch}{1.0}
  \caption{Lasso.
  }
  \label{tab:lasso_table}
  \vspace*{-3mm}
  \begin{tabular}{l}
  \iters
  \end{tabular}
  \adjustbox{max width=\textwidth}{
    \begin{tabular}{lllllllllllll}
    Fp res.&
    \begin{tabular}{@{}c@{}}Cold \\ Start\end{tabular}
    &
    \begin{tabular}{@{}c@{}}Nearest \\ Neighbor\end{tabular}
    &
    \begin{tabular}{@{}c@{}}Fp \\$k=0$ \end{tabular}
    &
    \begin{tabular}{@{}c@{}}Fp \\$k=5$ \end{tabular}
    &
    \begin{tabular}{@{}c@{}}Fp \\$k=15$ \end{tabular}
    &
    \begin{tabular}{@{}c@{}}Fp \\$k=30$ \end{tabular}
    &
    \begin{tabular}{@{}c@{}}Fp \\$k=60$ \end{tabular}
    &
    \begin{tabular}{@{}c@{}}Reg \\$k=0$ \end{tabular}
    &
    \begin{tabular}{@{}c@{}}Reg \\$k=5$ \end{tabular}
    &
    \begin{tabular}{@{}c@{}}Reg \\$k=15$ \end{tabular}
    &
    \begin{tabular}{@{}c@{}}Reg \\$k=30$ \end{tabular}
    &
    \begin{tabular}{@{}c@{}}Reg \\$k=60$ \end{tabular} \\
    \midrule
    \begin{filecontents*}{./data/iters_red/lasso.csv}
,accuracies,cold_start_iters,cold_start_red,nearest_neighbor_iters,nearest_neighbor_red,obj_k0_iters,obj_k0_red,obj_k5_iters,obj_k5_red,obj_k15_iters,obj_k15_red,obj_k30_iters,obj_k30_red,obj_k60_iters,obj_k60_red,reg_k0_iters,reg_k0_red,reg_k5_iters,reg_k5_red,reg_k15_iters,reg_k15_red,reg_k30_iters,reg_k30_red,reg_k60_iters,reg_k60_red
0,0.1,88,0,59,0.33,33,0.62,\bf{2},\bf{0.98},6,0.93,10,0.89,11,0.88,7,0.92,5,0.94,12,0.86,15,0.83,24,0.73
1,0.01,559,0,492,0.12,540,0.03,159,0.72,141,0.75,138,0.75,159,0.72,150,0.73,\bf{129},\bf{0.77},131,0.77,136,0.76,142,0.75
2,0.001,1821,0,1725,0.05,1814,0.0,1274,0.3,1239,0.32,1233,0.32,1267,0.3,1201,0.34,\bf{1181},\bf{0.35},1182,0.35,1205,0.34,1210,0.34
3,0.0001,4041,0,3894,0.04,4034,0.0,3410,0.16,3347,0.17,3334,0.17,3352,0.17,3229,0.2,\bf{3208},\bf{0.21},3209,0.21,3243,0.2,3240,0.2
\end{filecontents*}
\csvreader[head to column names, late after line=\\]{./data/iters_red/lasso.csv}{
    accuracies=\colA,
    cold_start_iters=\colC,
    nearest_neighbor_iters=\colD,
    obj_k0_iters=\colE,
    obj_k5_iters=\colF,
    obj_k15_iters=\colG,
    obj_k30_iters=\colH,
    obj_k60_iters=\colI,
    reg_k0_iters=\colJ,
    reg_k5_iters=\colK,
    reg_k15_iters=\colL,
    reg_k30_iters=\colM,
    reg_k60_iters=\colN,
    }{\colA & \colC & \colD & \colE & \colF & \colG & \colH & \colI & \colJ & \colK & \colL & \colM & \colN}
    \bottomrule
  \end{tabular}}
  \begin{tabular}{l}
    \reduction
  \end{tabular}
  \adjustbox{max width=\textwidth}{
    \begin{tabular}{lllllllllllll}
    Fp res.&
    \begin{tabular}{@{}c@{}}Cold \\ Start\end{tabular}
    &
    \begin{tabular}{@{}c@{}}Nearest \\ Neighbor\end{tabular}
    &
    \begin{tabular}{@{}c@{}}Fp \\$k=0$ \end{tabular}
    &
    \begin{tabular}{@{}c@{}}Fp \\$k=5$ \end{tabular}
    &
    \begin{tabular}{@{}c@{}}Fp \\$k=15$ \end{tabular}
    &
    \begin{tabular}{@{}c@{}}Fp \\$k=30$ \end{tabular}
    &
    \begin{tabular}{@{}c@{}}Fp \\$k=60$ \end{tabular}
    &
    \begin{tabular}{@{}c@{}}Reg \\$k=0$ \end{tabular}
    &
    \begin{tabular}{@{}c@{}}Reg \\$k=5$ \end{tabular}
    &
    \begin{tabular}{@{}c@{}}Reg \\$k=15$ \end{tabular}
    &
    \begin{tabular}{@{}c@{}}Reg \\$k=30$ \end{tabular}
    &
    \begin{tabular}{@{}c@{}}Reg \\$k=60$ \end{tabular} \\
    \midrule
    \csvreader[head to column names, late after line=\\]{./data/iters_red/lasso.csv}{
    accuracies=\colA,
    cold_start_red=\colC,
    nearest_neighbor_red=\colD,
    obj_k0_red=\colE,
    obj_k5_red=\colF,
    obj_k15_red=\colG,
    obj_k30_red=\colH,
    obj_k60_red=\colI,
    reg_k0_red=\colJ,
    reg_k5_red=\colK,
    reg_k15_red=\colL,
    reg_k30_red=\colM,
    reg_k60_red=\colN,
    }{\colA & \colC & \colD & \colE & \colF & \colG & \colH & \colI & \colJ & \colK & \colL & \colM & \colN}
    \bottomrule
  \end{tabular}}
\end{table}

\subsection{OSQP}\label{sec:osqp_experiments}
In this subsection, we apply our learning framework to the OSQP \citep{osqp} algorithm from Table~\ref{table:fp_algorithms} to solve convex quadratic programs (QPs).
We compare solve times using OSQP code written in C for our learned warm starts against the baselines.
Table~\ref{table:prob_sizes_osqp} shows the sizes of the problems we run: model predictive control of a quadcopter in \Sec~\ref{sec:quadcopter} and image deblurring in \Sec~\ref{sec:image_deblur}.

\begin{table}[!h]
  \small
  \centering
\renewcommand*{\arraystretch}{1.0}
\caption{Sizes of conic problems from Table~\ref{table:fp_algorithms} that we use OSQP to solve.
We give the number of primal constraints ($m$), size of the primal variable ($n$), and the parameter size, $d$.
}
\label{table:prob_sizes_osqp}
\begin{tabular}{lll}
\toprule
~ & Quadcopter & Image deblurring \\
\midrule
\begin{filecontents*}{data/prob_sizes_osqp.csv}
dim,quadcopter,image_deblur
constraints $m$,600,2102
variables $n$,550,802
parameter size $d$,44,784
\end{filecontents*}
\csvreader[head to column names, late after line=\\]{data/prob_sizes_osqp.csv}{
dim=\colA,
quadcopter=\colB,
image_deblur=\colC,
}{\colA & \colB & \colC}
\bottomrule
\end{tabular}
\end{table}

\subsubsection{Model predictive control of a quadcopter}\label{sec:quadcopter}
In our next example, we use model predictive control~\citep{borrelli_mpc_book} to control a quadcopter to follow a reference trajectory.
The idea of MPC is to optimize over a finite horizon length, but then to only implement the first control before optimizing again.
Since the family of problems is sequential in nature, we add an additional baseline called the \emph{previous-solution warm start} where we shift the solution of the previous problem by one time index to warm-start the current problem.

We model the quadcopter as a rigid body controlled by four motors as in \citet{quadcopter_pendulum}.
The state vector is $x = (p, v, q) \in \reals^{n_x}$ where the state size is $n_x = 10$.
The position vector $p = (p_x, p_y, p_z) \in \reals^3$ and the velocity vector $v = (v_x, v_y, v_z) \in \reals^3$ indicate the coordinates and velocities of the center of the quadcopter respectively.
The vector $q = (q_w, q_x, q_y, q_z) \in \reals^4$ is the quaternion vector indicating the orientation of the quadcopter.
The inputs are $u = (c, \omega_x, \omega_y, \omega_z) \in \reals^{n_u}$ where the input size is $n_u = 4$.
The first input is the vertical thrust, and the last three are the angular velocities in the body frame.
The dynamics are
\begin{align*}
\small
  \dot{p} &= v, &
  \dot{v} &= \begin{bmatrix}
    2 (q_w q_y + q_x q_z) c  \\
    2 (q_w q_y + q_x q_z) c \\
    (q_w^2 - q_x^2 - q_y^2 + q_z^2) c - g
  \end{bmatrix}, &
  \dot{q} &= \frac{1}{2}\begin{bmatrix}
    -w_x q_x - w_y q_y - w_z q_z\\
    w_x q_w - w_y q_z + w_z q_y\\
    w_x q_z + w_y q_w - w_z q_x\\
    w_x q_y + w_y q_x + w_z q_w\\
  \end{bmatrix},
\end{align*}
where $g$ is the gravitational constant.
At each time step, the goal is to track a reference trajectory given by $x^{\rm{ref}} = (x^{\rm{ref}}_1, \dots, x^{\rm{ref}}_T)$, while satisfying constraints on the states and the controls.
We discretize the system with $\Delta t$ and solve the QP
\begin{equation*}
  \begin{array}{ll}
  \label{eq:quadcopter_qp}
  \mbox{minimize} & (x_{T} - x_T^{\rm{ref}})^T Q_{T} (x_{T} - x_T^{\rm{ref}}) + \sum_{t=1}^{T-1} (x_{t} - x_t^{\rm{ref}})^T Q (x_{t} - x_t^{\rm{ref}}) + u_t^T R u_t \\
  \mbox{subject to} & x_{t+1} = A x_t + B u_t \quad t=0, \dots, T-1 \\
  & u_{\textrm{min}} \leq u_t \leq u_{\textrm{min}} \quad t=0, \dots, T-1 \\
  & x_{\textrm{min}} \leq x_t \leq x_{\textrm{max}} \quad t=1, \dots, T\\
  & |u_{t+1} - u_t| \leq \Delta u \quad t=1, \dots, T-1.\\
  \end{array}
\end{equation*}
Here, the decision variables are the states $(x_1, \dots, x_{T})$ where $x_t \in \reals^{n_x}$ and the controls $(u_1, \dots, u_{T-1})$ where $u_t \in \reals^{n_u}$.
The dynamics matrices $A$ and $B$ are determined by linearizing the dynamics around the current state $x_0$, and the previous control input $u_0$~\citep{nonlinear_mpc}.
The matrices $Q, Q_T \in \symm_+^{n_x}$ penalize the distance of the states to the reference trajectory, $(x_1^{\rm{ref}}, \dots, x_T^{\rm{ref}})$.
The matrix $R \in \symm_{++}^{n_u}$ regularizes the controls.
The parameter is $\theta = (x_0, u_0, x_1^{\rm{ref}}, \dots, x_T^{\rm{ref}}) \in \reals^{(T + 1) n_x + n_u}$.
We generate many different trajectories where the simulation length is larger than the time horizon $T$.




\paragraph{Numerical example.}
We discretize our continuous time model with a value of $\Delta t = 0.05$ seconds.
The gravitational constant is $9.8$.
Each trajectory has a length of $100$, and the horizon we consider at each timestep for each QP is $T=10$.
We use state bounds of $x_{\rm{max}} = -x_{\rm{min}} = (1, 1, 1, 10, 10, 10, 1, 1, 1, 1)$.
We constrain the controls with $u_{\rm{max}} = (20, 6, 6, 6)$ and $u_{\rm{min}} = (2, -6, -6, -6)$, and set $\Delta u = (18, 6, 6, 6)$.
For each simulation, the quadcopter is initialized at $p = v = 0$ and $q = (1, 0, 0, 0)$.
We sample $5$ waypoints for each of the $(x,y,z)$ coordinates for each trajectory from the uniform distribution, $\mathcal{U}[-0.5, 0.5]$.
Then we use a B-spline~\citep{b_spline} to smoothly interpolate between the waypoints to generate $100$ points for the entire trajectory.
Since each reference trajectory is made up of $(x,y,z)$ coordinates rather than the full state vector, we shorten the parameter size to $\theta \in \reals^{n_x + n_u + 3T}$.

\paragraph{Results.}
Figure~\ref{fig:quadcopter_results} and Table~\ref{tab:quadcopter_times} show the convergence behavior of our method.
While all of the warm starts learned with the regression losses deliver substantial improvements over the baselines, our method using $k=60$ with the fixed-point residual loss stands out as the best for a larger number of steps.
To simulate a strict latency requirement, we also compare various initialization techniques in a closed-loop system where only $15$ OSQP iterations are allowed per QP in Figure~\ref{fig:quadcopter_visualize}.
The learned warm start can more accurately track the reference trajectory compared with the other two methods.

\begin{table}[!h]
  \centering
  \small
    \renewcommand*{\arraystretch}{1.0}
  \caption{Quadcopter.
  }
  \label{tab:quadcopter_times}
  \vspace*{-3mm}
  \begin{tabular}{l}
  \iters
  \end{tabular}
  \adjustbox{max width=\textwidth}{
    \begin{tabular}{llllllllllllll}
    Fp res.&
    \begin{tabular}{@{}c@{}}Cold \\ Start\end{tabular}
    &
    \begin{tabular}{@{}c@{}}Nearest \\ Neighbor\end{tabular}
    &
    \begin{tabular}{@{}c@{}}Previous \\ Solution\end{tabular}
    &
    \begin{tabular}{@{}c@{}}Fp \\$k=0$ \end{tabular}
    &
    \begin{tabular}{@{}c@{}}Fp \\$k=5$ \end{tabular}
    &
    \begin{tabular}{@{}c@{}}Fp \\$k=15$ \end{tabular}
    &
    \begin{tabular}{@{}c@{}}Fp \\$k=30$ \end{tabular}
    &
    \begin{tabular}{@{}c@{}}Fp \\$k=60$ \end{tabular}
    &
    \begin{tabular}{@{}c@{}}Reg \\$k=0$ \end{tabular}
    &
    \begin{tabular}{@{}c@{}}Reg \\$k=5$ \end{tabular}
    &
    \begin{tabular}{@{}c@{}}Reg \\$k=15$ \end{tabular}
    &
    \begin{tabular}{@{}c@{}}Reg \\$k=30$ \end{tabular}
    &
    \begin{tabular}{@{}c@{}}Reg \\$k=60$ \end{tabular} \\
    \midrule
    \begin{filecontents*}{./data/iters_red/quadcopter.csv}
,accuracies,cold_start_iters,cold_start_red,nearest_neighbor_iters,nearest_neighbor_red,prev_sol_iters,prev_sol_red,obj_k0_iters,obj_k0_red,obj_k5_iters,obj_k5_red,obj_k15_iters,obj_k15_red,obj_k30_iters,obj_k30_red,obj_k60_iters,obj_k60_red,reg_k0_iters,reg_k0_red,reg_k5_iters,reg_k5_red,reg_k15_iters,reg_k15_red,reg_k30_iters,reg_k30_red,reg_k60_iters,reg_k60_red
0,0.1,103,0,26,0.75,28,0.73,86,0.17,\bf{9},\bf{0.91},10,0.9,15,0.85,34,0.67,14,0.86,14,0.86,15,0.85,26,0.75,39,0.62
1,0.01,682,0,127,0.81,115,0.83,696,-0.02,480,0.3,79,0.88,\bf{24},\bf{0.96},50,0.93,25,0.96,26,0.96,26,0.96,35,0.95,54,0.92
2,0.001,2262,0,1210,0.47,1416,0.37,2582,-0.14,4268,-0.89,5728,-1.53,693,0.69,\bf{277},\bf{0.88},596,0.74,673,0.7,604,0.73,616,0.73,636,0.72
3,0.0001,9958,0,6573,0.34,6906,0.31,9980,-0.0,12938,-0.3,14000,-0.41,7212,0.28,\bf{5681},\bf{0.43},6041,0.39,5957,0.4,6015,0.4,6027,0.39,6053,0.39
\end{filecontents*}
\csvreader[head to column names, late after line=\\]{./data/iters_red/quadcopter.csv}{
    accuracies=\colA,
    cold_start_iters=\colC,
    nearest_neighbor_iters=\colD,
    prev_sol_iters=\colB,
    obj_k0_iters=\colE,
    obj_k5_iters=\colF,
    obj_k15_iters=\colG,
    obj_k30_iters=\colH,
    obj_k60_iters=\colI,
    reg_k0_iters=\colJ,
    reg_k5_iters=\colK,
    reg_k15_iters=\colL,
    reg_k30_iters=\colM,
    reg_k60_iters=\colN,
    }{\colA & \colC & \colD & \colB & \colE & \colF & \colG & \colH & \colI & \colJ & \colK & \colL & \colM & \colN}
    \bottomrule
  \end{tabular}}
  \begin{tabular}{l}
    \reduction
  \end{tabular}
  \adjustbox{max width=\textwidth}{
    \begin{tabular}{llllllllllllll}
    Fp res.&
    \begin{tabular}{@{}c@{}}Cold \\ Start\end{tabular}
    &
    \begin{tabular}{@{}c@{}}Nearest \\ Neighbor\end{tabular}
    &
    \begin{tabular}{@{}c@{}}Previous \\ Solution\end{tabular}
    &
    \begin{tabular}{@{}c@{}}Fp \\$k=0$ \end{tabular}
    &
    \begin{tabular}{@{}c@{}}Fp \\$k=5$ \end{tabular}
    &
    \begin{tabular}{@{}c@{}}Fp \\$k=15$ \end{tabular}
    &
    \begin{tabular}{@{}c@{}}Fp \\$k=30$ \end{tabular}
    &
    \begin{tabular}{@{}c@{}}Fp \\$k=60$ \end{tabular}
    &
    \begin{tabular}{@{}c@{}}Reg \\$k=0$ \end{tabular}
    &
    \begin{tabular}{@{}c@{}}Reg \\$k=5$ \end{tabular}
    &
    \begin{tabular}{@{}c@{}}Reg \\$k=15$ \end{tabular}
    &
    \begin{tabular}{@{}c@{}}Reg \\$k=30$ \end{tabular}
    &
    \begin{tabular}{@{}c@{}}Reg \\$k=60$ \end{tabular} \\
    \midrule
    \csvreader[head to column names, late after line=\\]{./data/iters_red/quadcopter.csv}{
    accuracies=\colA,
    cold_start_red=\colC,
    nearest_neighbor_red=\colD,
    prev_sol_red=\colB,
    obj_k0_red=\colE,
    obj_k5_red=\colF,
    obj_k15_red=\colG,
    obj_k30_red=\colH,
    obj_k60_red=\colI,
    reg_k0_red=\colJ,
    reg_k5_red=\colK,
    reg_k15_red=\colL,
    reg_k30_red=\colM,
    reg_k60_red=\colN,
    }{\colA & \colC & \colD & \colB & \colE & \colF & \colG & \colH & \colI & \colJ & \colK & \colL & \colM & \colN}
    \bottomrule
  \end{tabular}}
  \begin{tabular}{l}
  \osqptiming
  \end{tabular}
  \adjustbox{max width=\textwidth}{
    \begin{tabular}{llllllllllllll}
    tol.&
    \begin{tabular}{@{}c@{}}Cold \\ Start\end{tabular}
    &
    \begin{tabular}{@{}c@{}}Nearest \\ Neighbor\end{tabular}
    &
    \begin{tabular}{@{}c@{}}Previous \\ Solution\end{tabular}
    &
    \begin{tabular}{@{}c@{}}Fp \\$k=0$ \end{tabular}
    &
    \begin{tabular}{@{}c@{}}Fp \\$k=5$ \end{tabular}
    &
    \begin{tabular}{@{}c@{}}Fp \\$k=15$ \end{tabular}
    &
    \begin{tabular}{@{}c@{}}Fp \\$k=30$ \end{tabular}
    &
    \begin{tabular}{@{}c@{}}Fp \\$k=60$ \end{tabular}
    &
    \begin{tabular}{@{}c@{}}Reg \\$k=0$ \end{tabular}
    &
    \begin{tabular}{@{}c@{}}Reg \\$k=5$ \end{tabular}
    &
    \begin{tabular}{@{}c@{}}Reg \\$k=15$ \end{tabular}
    &
    \begin{tabular}{@{}c@{}}Reg \\$k=30$ \end{tabular}
    &
    \begin{tabular}{@{}c@{}}Reg \\$k=60$ \end{tabular} \\
    \midrule
    \begin{filecontents*}{./data/timings/quadcopter.csv}
,rel_tols,abs_tols,cold_start,nearest_neighbor,prev_sol,obj_k0,obj_k5,obj_k15,obj_k30,obj_k60,reg_k0,reg_k5,reg_k15,reg_k30,reg_k60
0,0.1,0.1,0.19,0.17,0.17,0.32,0.17,0.18,0.18,0.19,\bf{0.16},\bf{0.16},\bf{0.16},0.17,0.18
1,0.01,0.01,2.44,0.23,0.21,2.32,0.68,0.25,0.18,0.21,\bf{0.16},0.17,\bf{0.16},0.17,0.28
2,0.001,0.001,12.0,4.69,6.33,11.77,17.65,35.89,1.37,\bf{0.77},1.18,1.23,1.23,1.3,1.45
3,0.0001,0.0001,54.59,24.63,28.66,50.15,68.56,104.18,29.14,\bf{14.34},18.26,18.86,18.51,18.9,18.6
4,1e-05,1e-05,114.62,76.03,78.21,103.42,128.59,174.69,85.6,\bf{63.35},65.23,67.0,65.88,68.24,65.43
\end{filecontents*}
\csvreader[head to column names, late after line=\\]{./data/timings/quadcopter.csv}{
    rel_tols=\colA,
    cold_start=\colC,
    nearest_neighbor=\colD,
    prev_sol=\colB,
    obj_k0=\colE,
    obj_k5=\colF,
    obj_k15=\colG,
    obj_k30=\colH,
    obj_k60=\colI,
    reg_k0=\colJ,
    reg_k5=\colK,
    reg_k15=\colL,
    reg_k30=\colM,
    reg_k60=\colN,
    }{\colA & \colC & \colD & \colB & \colE & \colF & \colG & \colH & \colI & \colJ & \colK & \colL & \colM & \colN}
    \bottomrule
  \end{tabular}}
\end{table}

\begin{figure}[!h]
  \centering
    \includegraphics[width=\figsize\linewidth]{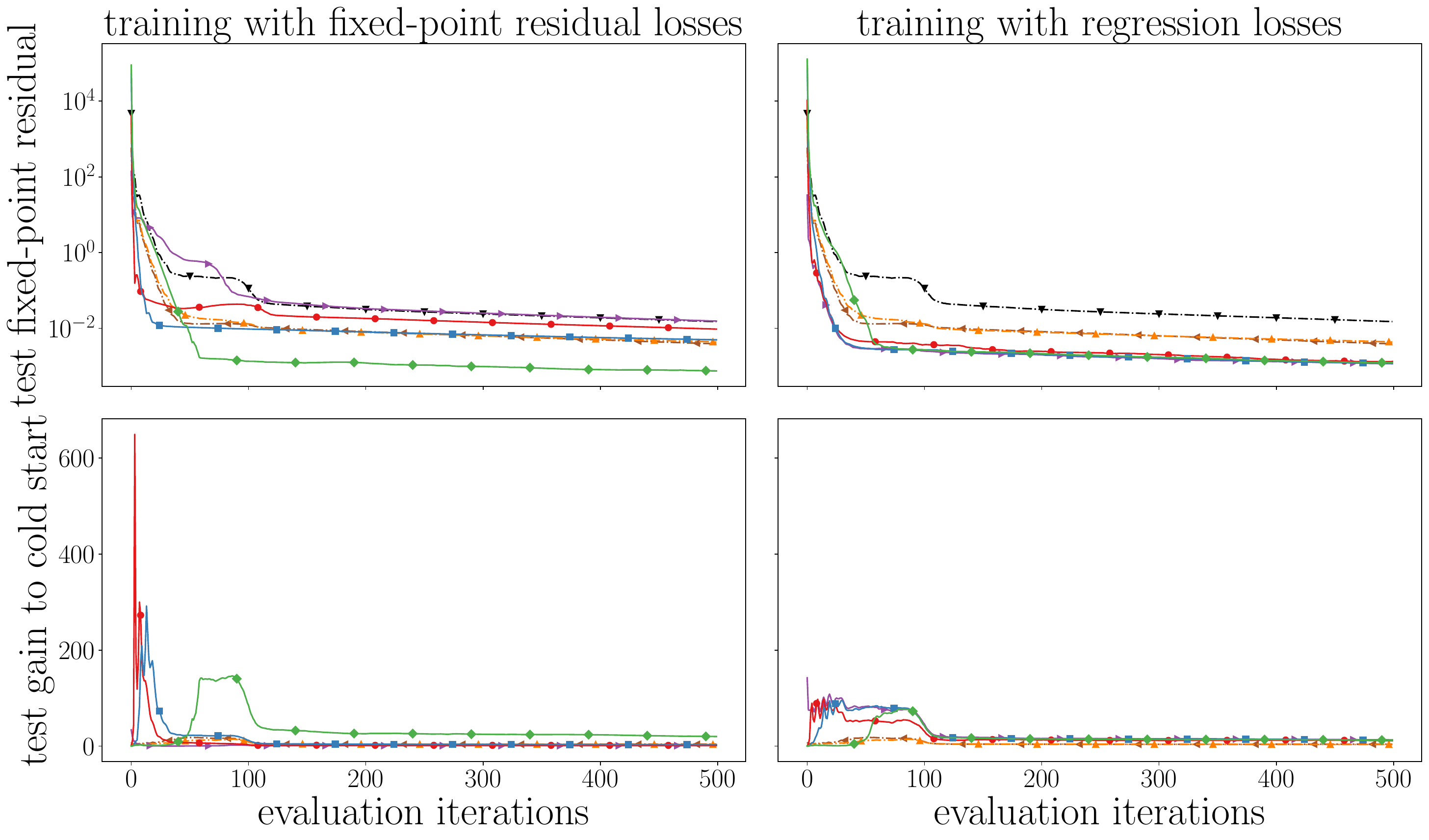}
    \\
    \legendprevsol\\
    \caption{Quadcopter results.
      Learned warm starts offer substantial improvements over the baselines.
      In particular, warm starts learned with $k=60$ and the fixed-point residual loss have the largest gain for evaluation steps over about $50$.
    }
    \label{fig:quadcopter_results}
\end{figure}

\ifpreprint
\newcommand{\quadcopterwidth}{\quadcopterwidthpreprint}
\fi

\begin{figure}[!h]
  \vspace*{-5mm}
  \centering
  \begin{subfigure}[t]{\quadcopterwidth\linewidth}
    \centering
    \includegraphics[width=1\textwidth]{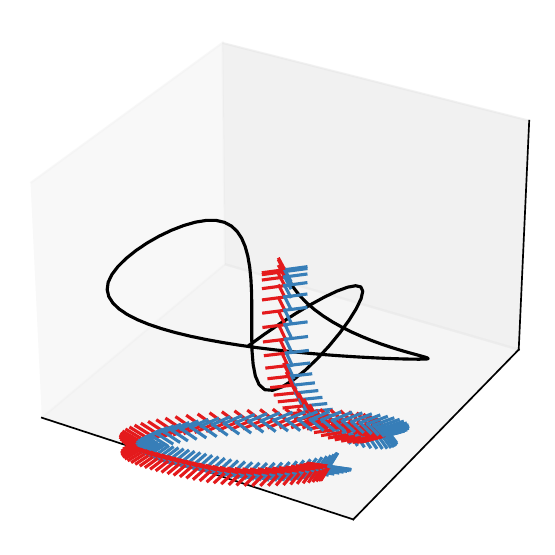}
    \label{fig:quad_prev_sol2}
  \end{subfigure}%
  \hspace*{4mm}
  \begin{subfigure}[t]{\quadcopterwidth\linewidth}
    \centering
    \includegraphics[width=1\linewidth]{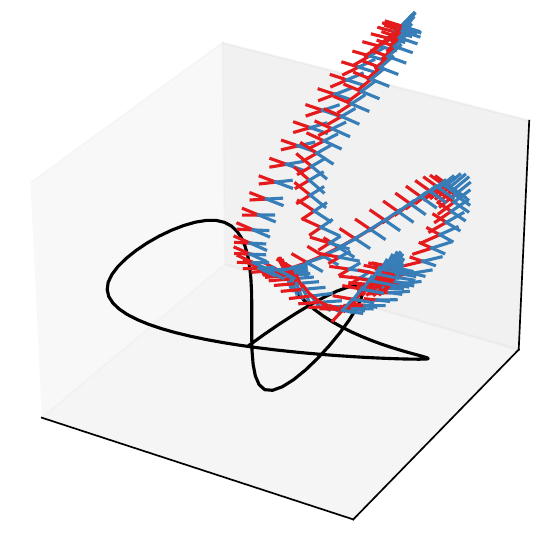}
    \label{fig:quad_nearest_neighbor2}
  \end{subfigure}
  \hspace*{4mm}
  \begin{subfigure}[t]{\quadcopterwidth\linewidth}
    \centering
    \includegraphics[width=1\linewidth]{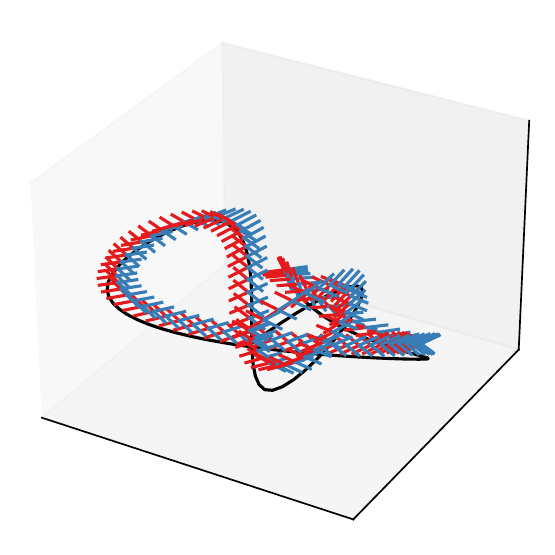}
    \label{fig:quad_k5_reg2}
  \end{subfigure}\\[-20pt]
  \begin{subfigure}[t]{\quadcopterwidth\linewidth}
    \centering
    \includegraphics[width=1\linewidth]{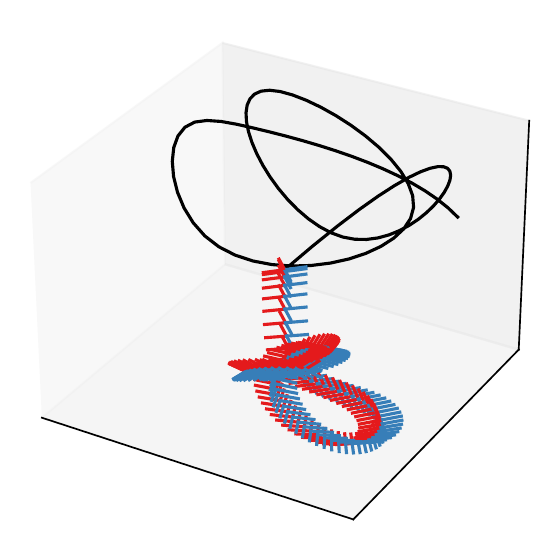}
    \caption{previous solution}
    \label{fig:quad_prev_sol3}
  \end{subfigure}%
  \hspace*{4mm}
  \begin{subfigure}[t]{\quadcopterwidth\linewidth}
    \centering
    \includegraphics[width=1\linewidth]{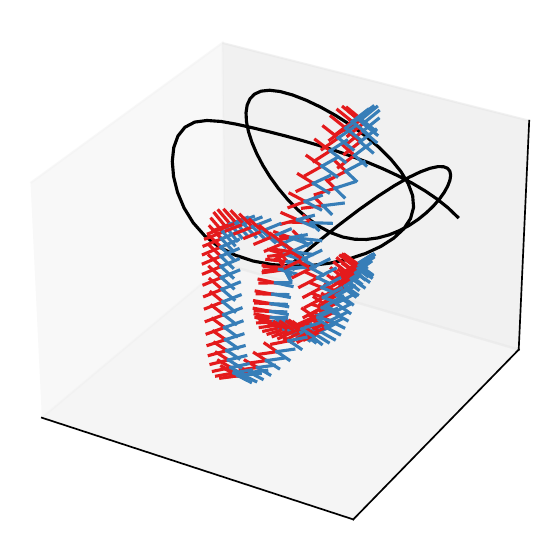}
    \caption{nearest neighbor}
    \label{fig:quad_nearest_neighbor}
  \end{subfigure}
  \hspace*{4mm}
  \begin{subfigure}[t]{\quadcopterwidth\linewidth}
    \centering
    \includegraphics[width=1\linewidth]{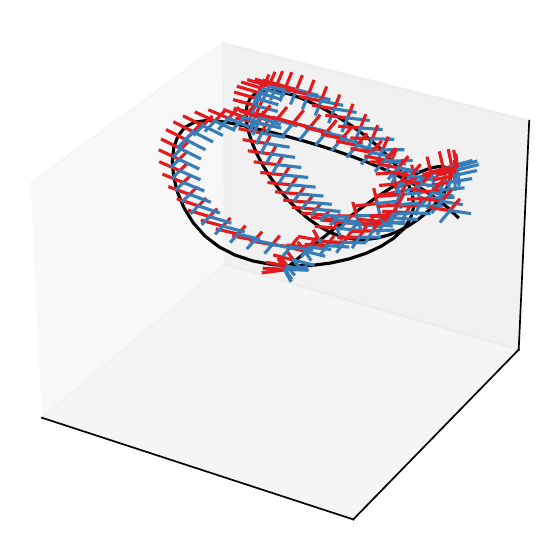}
    \caption{learned}
    \label{fig:quad_k5_reg3}
  \end{subfigure}
  \caption{Visualizing closed-loop MPC of flying a quadcopter to track a reference trajectory.
  Each row corresponds to a different unseen reference trajectory.
  Each column uses a different initialization scheme to track the same unseen black reference trajectory in a closed-loop.
  Each technique is given a budget of $15$ OSQP iterations to solve each QP.
  The learned approach which is trained on $k=5$ with the regression loss tracks the trajectory well compared against the other two.
  }
  \label{fig:quadcopter_visualize}
\end{figure}




\subsubsection{Image deblurring}\label{sec:image_deblur}
We turn our attention to the task of image deblurring.
Given a blurry image $b \in \reals^{n}$, the goal is to recover the original image $x \in \reals^{n}$.
Both the noisy vector $b$ and the target vector $x$ are formed by stacking the columns of their respective images.
We formulate this well-studied problem~\citep{fista,nonneg_deblur} as
\begin{equation*}
  \begin{array}{ll}
  \label{prob:img_deblur}
  \mbox{minimize} & \|A x - b\|_2^2 + \lambda \|x\|_1 \\
  \mbox{subject to} & 0 \leq x \leq 1. \\
  \end{array}
\end{equation*}
Here, the matrix $A \in \reals^{n \times n}$ is the blur operator which represents a two-dimensional convolutional operator.
The regularization hyperparameter $\lambda \in \reals_{++}$,  weights the fidelity term $\|A x - b\|_2^2$, relative to the $\ell_1$ penalty.
The $\ell_1$ penalty is used as it less sensitive to outliers and encourages sparsity \citep{fista}.
The constraints ensure that the deblurred image has pixel values within its domain.


\paragraph{Numerical example.}
We consider handwritten letters from the EMNIST dataset~\citep{emnist}.
We apply a Gaussian blue of size $8$ to each letter and then add i.i.d. Gaussian noise with standard deviation $0.001$.
The hyperparameter weighting term is $\lambda = 1e-4$.

\paragraph{Results.}
Figure~\ref{fig:image_deblur_results} and Table~\ref{tab:mnist_times} show the convergence behavior of our method.
Learned warm starts with the regression loss tend to outperform the learned warm starts with the fixed-point residual loss.
We show visualizations of our method in Figure~\ref{fig:mnist_vis}.
For images that are particularly challenging, the image quality after $50$ OSQP steps is significantly better for the learned warm start than the baseline initializations.
\begin{figure}[!h]
  \centering
    \includegraphics[width=\figsize\linewidth]{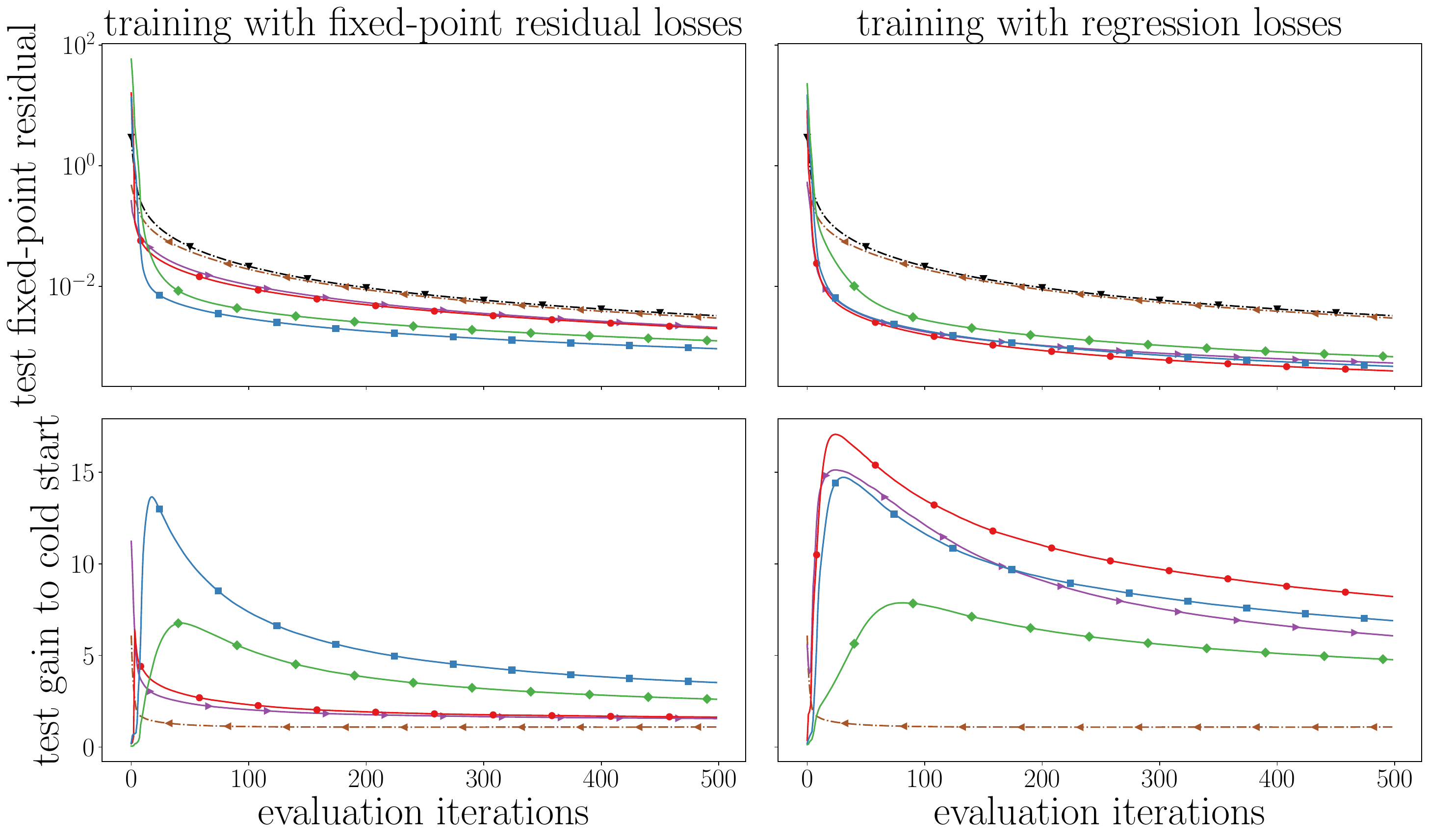}
    \\
    \legend
    \caption{Image deblurring.
      Warm starts learned with the regression loss provide bigger gains compared with those learned with the fixed-point residual loss.
    }
    \label{fig:image_deblur_results}
\end{figure}

\begin{table}[!h]
  \centering
    \renewcommand*{\arraystretch}{1.0}
  \caption{Image deblurring.
  }
  \label{tab:mnist_times}
  \small
  \vspace*{-3mm}
  \begin{tabular}{l}
    \iters
  \end{tabular}
  \adjustbox{max width=\textwidth}{
    \begin{tabular}{lllllllllllll}
    Fp res.&
    \begin{tabular}{@{}c@{}}Cold \\ Start\end{tabular}
    &
    \begin{tabular}{@{}c@{}}Nearest \\ Neighbor\end{tabular}
    &
    \begin{tabular}{@{}c@{}}Fp \\$k=0$ \end{tabular}
    &
    \begin{tabular}{@{}c@{}}Fp \\$k=5$ \end{tabular}
    &
    \begin{tabular}{@{}c@{}}Fp \\$k=15$ \end{tabular}
    &
    \begin{tabular}{@{}c@{}}Fp \\$k=30$ \end{tabular}
    &
    \begin{tabular}{@{}c@{}}Fp \\$k=60$ \end{tabular}
    &
    \begin{tabular}{@{}c@{}}Reg \\$k=0$ \end{tabular}
    &
    \begin{tabular}{@{}c@{}}Reg \\$k=5$ \end{tabular}
    &
    \begin{tabular}{@{}c@{}}Reg \\$k=15$ \end{tabular}
    &
    \begin{tabular}{@{}c@{}}Reg \\$k=30$ \end{tabular}
    &
    \begin{tabular}{@{}c@{}}Reg \\$k=60$ \end{tabular} \\
    \midrule
    \begin{filecontents*}{./data/iters_red/mnist.csv}
,accuracies,cold_start_iters,cold_start_red,nearest_neighbor_iters,nearest_neighbor_red,obj_k0_iters,obj_k0_red,obj_k5_iters,obj_k5_red,obj_k15_iters,obj_k15_red,obj_k30_iters,obj_k30_red,obj_k60_iters,obj_k60_red,reg_k0_iters,reg_k0_red,reg_k5_iters,reg_k5_red,reg_k15_iters,reg_k15_red,reg_k30_iters,reg_k30_red,reg_k60_iters,reg_k60_red
0,0.1,24,0,16,0.33,6,0.75,\bf{5},\bf{0.79},8,0.67,10,0.58,12,0.5,\bf{5},\bf{0.79},6,0.75,8,0.67,9,0.62,11,0.54
1,0.01,194,0,181,0.07,107,0.45,93,0.52,17,0.91,27,0.86,36,0.81,16,\bf{0.92},\bf{15},\bf{0.92},18,0.91,27,0.86,42,0.78
2,0.001,1348,0,1253,0.07,982,0.27,953,0.29,454,0.66,589,0.56,658,0.51,215,0.84,\bf{171},\bf{0.87},208,0.85,266,0.8,321,0.76
3,0.0001,7767,0,7607,0.02,6613,0.15,6689,0.14,5668,0.27,6359,0.18,6812,0.12,4238,0.45,\bf{3292},\bf{0.58},3779,0.51,4399,0.43,4744,0.39
\end{filecontents*}
\csvreader[head to column names, late after line=\\]{./data/iters_red/mnist.csv}{
    accuracies=\colA,
    cold_start_iters=\colC,
    nearest_neighbor_iters=\colD,
    obj_k0_iters=\colE,
    obj_k5_iters=\colF,
    obj_k15_iters=\colG,
    obj_k30_iters=\colH,
    obj_k60_iters=\colI,
    reg_k0_iters=\colJ,
    reg_k5_iters=\colK,
    reg_k15_iters=\colL,
    reg_k30_iters=\colM,
    reg_k60_iters=\colN,
    }{\colA & \colC & \colD & \colE & \colF & \colG & \colH & \colI & \colJ & \colK & \colL & \colM & \colN}
    \bottomrule
  \end{tabular}}
  \begin{tabular}{l}
    \reduction
  \end{tabular}
  \adjustbox{max width=\textwidth}{
    \begin{tabular}{lllllllllllll}
    Fp res.&
    \begin{tabular}{@{}c@{}}Cold \\ Start\end{tabular}
    &
    \begin{tabular}{@{}c@{}}Nearest \\ Neighbor\end{tabular}
    &
    \begin{tabular}{@{}c@{}}Fp \\$k=0$ \end{tabular}
    &
    \begin{tabular}{@{}c@{}}Fp \\$k=5$ \end{tabular}
    &
    \begin{tabular}{@{}c@{}}Fp \\$k=15$ \end{tabular}
    &
    \begin{tabular}{@{}c@{}}Fp \\$k=30$ \end{tabular}
    &
    \begin{tabular}{@{}c@{}}Fp \\$k=60$ \end{tabular}
    &
    \begin{tabular}{@{}c@{}}Reg \\$k=0$ \end{tabular}
    &
    \begin{tabular}{@{}c@{}}Reg \\$k=5$ \end{tabular}
    &
    \begin{tabular}{@{}c@{}}Reg \\$k=15$ \end{tabular}
    &
    \begin{tabular}{@{}c@{}}Reg \\$k=30$ \end{tabular}
    &
    \begin{tabular}{@{}c@{}}Reg \\$k=60$ \end{tabular} \\
    \midrule
    \csvreader[head to column names, late after line=\\]{./data/iters_red/mnist.csv}{
    accuracies=\colA,
    cold_start_red=\colC,
    nearest_neighbor_red=\colD,
    obj_k0_red=\colE,
    obj_k5_red=\colF,
    obj_k15_red=\colG,
    obj_k30_red=\colH,
    obj_k60_red=\colI,
    reg_k0_red=\colJ,
    reg_k5_red=\colK,
    reg_k15_red=\colL,
    reg_k30_red=\colM,
    reg_k60_red=\colN,
    }{\colA & \colC & \colD & \colE & \colF & \colG & \colH & \colI & \colJ & \colK & \colL & \colM & \colN}
    \bottomrule
  \end{tabular}}
  \begin{tabular}{l}
  \osqptiming
  \end{tabular}
  \adjustbox{max width=\textwidth}{
    \begin{tabular}{lllllllllllll}
    tol.&
    \begin{tabular}{@{}c@{}}Cold \\ Start\end{tabular}
    &
    \begin{tabular}{@{}c@{}}Nearest \\ Neighbor\end{tabular}
    &
    \begin{tabular}{@{}c@{}}Fp \\$k=0$ \end{tabular}
    &
    \begin{tabular}{@{}c@{}}Fp \\$k=5$ \end{tabular}
    &
    \begin{tabular}{@{}c@{}}Fp \\$k=15$ \end{tabular}
    &
    \begin{tabular}{@{}c@{}}Fp \\$k=30$ \end{tabular}
    &
    \begin{tabular}{@{}c@{}}Fp \\$k=60$ \end{tabular}
    &
    \begin{tabular}{@{}c@{}}Reg \\$k=0$ \end{tabular}
    &
    \begin{tabular}{@{}c@{}}Reg \\$k=5$ \end{tabular}
    &
    \begin{tabular}{@{}c@{}}Reg \\$k=15$ \end{tabular}
    &
    \begin{tabular}{@{}c@{}}Reg \\$k=30$ \end{tabular}
    &
    \begin{tabular}{@{}c@{}}Reg \\$k=60$ \end{tabular} \\
    \midrule
    \begin{filecontents*}{./data/timings/mnist.csv}
,rel_tols,abs_tols,cold_start,nearest_neighbor,obj_k0,obj_k5,obj_k15,obj_k30,obj_k60,reg_k0,reg_k5,reg_k15,reg_k30,reg_k60
0,0.1,0.1,6.34,6.37,6.32,6.31,\bf{6.30},6.38,6.32,6.33,6.34,6.32,6.35,6.33
1,0.01,0.01,9.13,9.62,6.32,6.33,6.32,\bf{6.31},6.33,6.33,6.33,6.35,6.32,6.33
2,0.001,0.001,62.37,68.25,48.0,38.51,9.9,14.78,18.38,9.51,\bf{7.55},8.21,11.13,15.54
3,0.0001,0.0001,464.0,472.7,398.8,363.2,230.1,290.2,333.9,115.9,\bf{80.8},106.0,140.8,168.6
4,1e-05,1e-05,3048,2997,2630,2778,2500,2761,2934,1957,\bf{1691},1832,1973,2090
\end{filecontents*}
\csvreader[head to column names, late after line=\\]{./data/timings/mnist.csv}{
    rel_tols=\colA,
    cold_start=\colC,
    nearest_neighbor=\colD,
    obj_k0=\colE,
    obj_k5=\colF,
    obj_k15=\colG,
    obj_k30=\colH,
    obj_k60=\colI,
    reg_k0=\colJ,
    reg_k5=\colK,
    reg_k15=\colL,
    reg_k30=\colM,
    reg_k60=\colN,
    }{\colA & \colC & \colD & \colE & \colF & \colG & \colH & \colI & \colJ & \colK & \colL & \colM & \colN}
    \bottomrule
  \end{tabular}}
\end{table}

\begin{figure}[!h]
  \centering
    \includegraphics[width=0.55\linewidth]{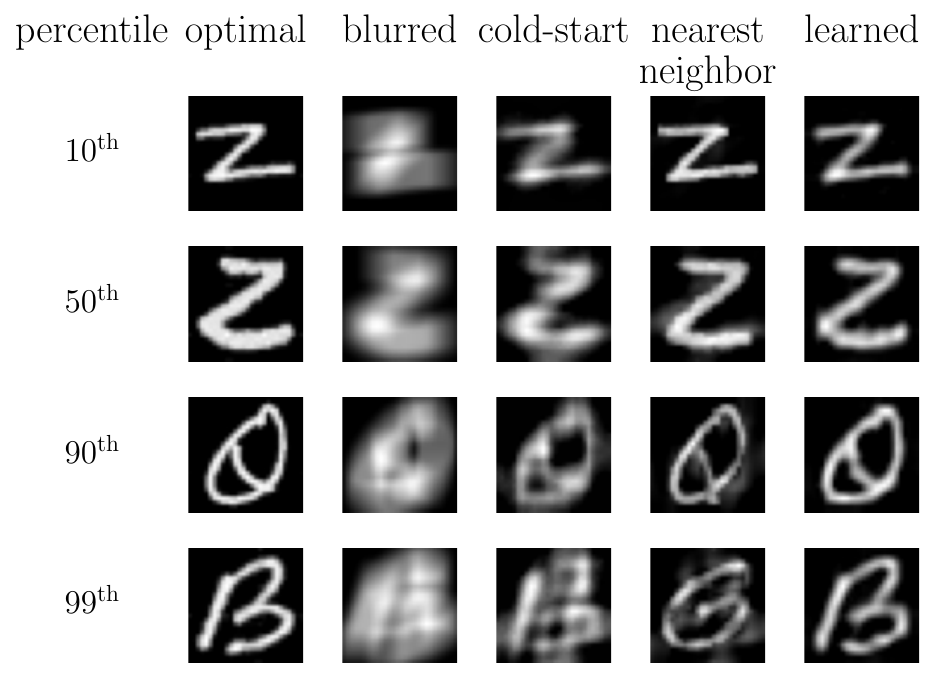}
    \\
    \caption{EMNIST image deblurring.
      Each row corresponds to an unseen sample from the EMNIST dataset.
      The last three columns depict several different initialization techniques after $50$ OSQP steps.
      In the learned column, we use the regression loss with $k=5$.
      To adjust the difficulty of the images displayed, we select images corresponding to different percentiles of distance from the nearest neighbor to the optimal solution.
    }
    \label{fig:mnist_vis}
\end{figure}

\subsection{SCS}\label{sec:scs_experiments}
In this subsection, we apply our learning framework to the SCS \citep{scs} algorithm from Table~\ref{table:fp_algorithms} to solve convex conic optimization problems.
We compare solve times using SCS code written in C for our learned warm starts against the baselines.
We run our experiments on two second-order cone programs (SOCPs) in robust Kalman filtering in \Sec~\ref{sec:rkf_experiment} and robust non-negative least squares in \Sec~\ref{sec:rls_experiment} and two semidefinite programs (SDPs) in phase retrieval in \Sec~\ref{sec:phase_retrieval_experiment} and sparse PCA in \Sec~\ref{sec:sparse_pca_experiment}.

\begin{table}[!h]
  \small
  \centering
\renewcommand*{\arraystretch}{1.0}
\caption{Sizes of conic problems from Table~\ref{table:fp_algorithms} that we use SCS to solve.
We give the number of primal constraints ($m$), size of the primal variable ($n$), and the parameter size, $d$.
Then, we provide the sizes of the cones for each conic program.
For the second-order and the positive semidefinite cones, we supply arrays specifying the lengths of each respective cone.
The notation $100 \times [3]$ means that there are $100$ second-order cones each of size $3$.
}
\label{table:prob_sizes}
\begin{tabular}{lllll}
\toprule
~ & Kalman filter & robust least squares & phase retrieval & sparse PCA\\
\midrule
\begin{filecontents*}{data/prob_sizes_scs.csv}
dim;theflag;kalman_filter;robust_ls;phase_retrieval;sparse_pca
constraints $m$;;600;2102;3480;4022
variables $n$;;550;802;1600;2420
parameter size $d$;;100;500;120;55
zero;X;$600$;0;240;1
non-negative;;550;800;0;3201
second-order;;$100 \times [3]$;[801,501];0;0
positive semidefinite;;0;0;[80];[40]
\end{filecontents*}
\csvreader[head to column names, /csv/separator=semicolon, late after line=\\, before line=\ifthenelse{\equal{\theflag}{X}}{\vspace{-4mm} \\ \hline \vspace{0mm}}{},]{data/prob_sizes_scs.csv}{
dim=\colA,
kalman_filter=\colB,
robust_ls=\colC,
phase_retrieval=\colD,
sparse_pca=\colE
}{\colA & \colB & \colC & \colD & \colE}
\bottomrule
\end{tabular}
\end{table}


\subsubsection{Robust Kalman filtering}\label{sec:rkf_experiment}
Kalman filtering \citep{kalman_filter} is a widely used technique for predicting system states in the presence of noise in dynamic systems.
In our first SOCP example, we use robust Kalman filtering \citep{rkf} which mitigates the impact of outliers on the filtering process and model misspecifications to track a moving vehicle from noisy data location as in~\citet{neural_fp_accel_amos}.
The dynamical system is modeled by
\begin{equation}\label{eq:rkf_dynamics_eqn}
    x_{t+1} = Ax_t + Bw_t, \quad y_t = C x_t + v_t,\quad \text{for}\quad t=0,1,\dots,
\end{equation}
where $x_t \in \reals^{n_x}$ is the state, $y_t \in \reals^{n_o}$ is the observation, $w_t \in \reals^{n_u}$ is the input, and $v_t \in \reals^{n_o}$ is a perturbation to the observation.
The matrices $A \in \reals^{n_x \times n_x}$, $B \in \reals^{n_x \times n_u}$, and~$C \in \reals^{n_o \times n_x}$ give the dynamics of the system.
Our goal is to recover the state $x_t$ from the noisy measurements $y_t$.
To do so, we solve the problem
\begin{equation*}
\begin{array}{ll}
\label{prob:rkf}
\mbox{minimize} & \sum_{t=1}^{T-1} \|w_t\|_2^2 + \mu \psi_{\rho} (v_t) \\
\mbox{subject to} & x_{t+1} = A x_t + B w_t \quad t=0, \dots, T-1 \\
& y_t = C x_t + v_t \quad t=0, \dots, T-1.\\
\end{array}
\end{equation*}
Here, the Huber penalty function~\citep{huber} parametrized by $\rho \in \reals_{++}$ that robustifies against outliers is
\[\psi_{\rho}(a) = \begin{cases}
      \|a\|_2 & \|a\|_2 \leq \rho \\
      2 \rho \|a\|_2 - \rho^2 & \|a\|_2 \geq \rho \\
   \end{cases},
\]
and $\mu \in \reals_{++}$ weights this penalty term.
The decision variables are the $x_t$'s, $w_t$'s, and $v_t$'s.
The parameters are the observed $y_t$'s, \ie, $\theta = (y_0, \dots, y_{T-1})$.
In this example, we take advantage of rotational invariance of the problem.
We rotate the noisy trajectory so that $y_T$ is on the x-axis for every problem.
After solving the transformed problem (for any initialization) we reverse the rotation to obtain the solution of the original problem.

\paragraph{Numerical example.}
As in \citet{neural_fp_accel_amos}, we set $n_x=4$, $n_o=2$, $n_u=2$, $\mu=2$, $\rho=2$, and $T=50$.
The dynamics matrices are
\begin{equation*}
\small
    \label{eq:rkf_dynamics}
    A = \begin{bmatrix}
        1 & 0 & (1 - (\gamma/2)\Delta t) \Delta t & 0\\
        0 & 1 & 0 & (1 - (\gamma/2)\Delta t) \Delta t\\
        0 & 0 & 1 - \gamma \Delta t & 0\\
        0 & 0 & 0 & 1 - \gamma \Delta t
    \end{bmatrix},\hspace{0mm}
    B = \begin{bmatrix}
        1 / 2\Delta t^2 & 0 \\
        0 & 1 / 2\Delta t^2\\
        \Delta t & 0 \\
        0 & \Delta t
    \end{bmatrix}, \hspace{0mm}
    C = \begin{bmatrix}
        1 & 0 & 0 & 0\\
        0 & 1 & 0 & 0\\
    \end{bmatrix},
\end{equation*}
where $\Delta t=0.5$ and $\gamma=0.05$ are fixed to be respectively the sampling time and the velocity dampening parameter.
We generate the problem instances in the following way.
We generate true trajectories $\{x_0^*, \dots, x_{T-1}^*\}$ of the vehicle by first letting $x_0^* = 0$.
Then we sample the inputs as $w_t \sim \mathcal{N}(0, 0.01)$ and $v_t \sim \mathcal{N}(0, 0.01)$.
The trajectories are then fully defined via the dynamics equations in \Eqn~\eqref{eq:rkf_dynamics_eqn} with the sampled $w_t$'s and $v_t$'s.

\paragraph{Results.}

\begin{table}[!h]
  \centering
    \renewcommand*{\arraystretch}{1.0}
\small
  \caption{Robust Kalman filtering.
  }
  \label{tab:rkf_table}
  \vspace*{-3mm}
  \begin{tabular}{l}
    \iters
  \end{tabular}
  \adjustbox{max width=\textwidth}{
    \begin{tabular}{lllllllllllll}
    Fp res.&
    \begin{tabular}{@{}c@{}}Cold \\ Start\end{tabular}
    &
    \begin{tabular}{@{}c@{}}Nearest \\ Neighbor\end{tabular}
    &
    \begin{tabular}{@{}c@{}}Fp \\$k=0$ \end{tabular}
    &
    \begin{tabular}{@{}c@{}}Fp \\$k=5$ \end{tabular}
    &
    \begin{tabular}{@{}c@{}}Fp \\$k=15$ \end{tabular}
    &
    \begin{tabular}{@{}c@{}}Fp \\$k=30$ \end{tabular}
    &
    \begin{tabular}{@{}c@{}}Fp \\$k=60$ \end{tabular}
    &
    \begin{tabular}{@{}c@{}}Reg \\$k=0$ \end{tabular}
    &
    \begin{tabular}{@{}c@{}}Reg \\$k=5$ \end{tabular}
    &
    \begin{tabular}{@{}c@{}}Reg \\$k=15$ \end{tabular}
    &
    \begin{tabular}{@{}c@{}}Reg \\$k=30$ \end{tabular}
    &
    \begin{tabular}{@{}c@{}}Reg \\$k=60$ \end{tabular} \\
    \midrule
    \begin{filecontents*}{./data/iters_red/robust_kalman.csv}
,accuracies,cold_start_iters,cold_start_red,nearest_neighbor_iters,nearest_neighbor_red,obj_k0_iters,obj_k0_red,obj_k5_iters,obj_k5_red,obj_k15_iters,obj_k15_red,obj_k30_iters,obj_k30_red,obj_k60_iters,obj_k60_red,reg_k0_iters,reg_k0_red,reg_k5_iters,reg_k5_red,reg_k15_iters,reg_k15_red,reg_k30_iters,reg_k30_red,reg_k60_iters,reg_k60_red
0,0.1,40,0,36,0.1,34,0.15,\bf{4},\bf{0.9},8,0.8,32,0.2,34,0.15,24,0.4,9,0.78,33,0.18,32,0.2,34,0.15
1,0.01,90,0,85,0.06,87,0.03,33,0.63,\bf{25},\bf{0.72},81,0.1,79,0.12,77,0.14,26,0.71,81,0.1,79,0.12,79,0.12
2,0.001,142,0,139,0.02,140,0.01,84,0.41,\bf{73},\bf{0.49},134,0.06,131,0.08,131,0.08,75,0.47,134,0.06,131,0.08,131,0.08
3,0.0001,195,0,193,0.01,195,0.0,137,0.3,\bf{126},\bf{0.35},187,0.04,185,0.05,185,0.05,128,0.34,187,0.04,185,0.05,185,0.05
\end{filecontents*}
\csvreader[head to column names, late after line=\\]{./data/iters_red/robust_kalman.csv}{
    accuracies=\colA,
    cold_start_iters=\colC,
    nearest_neighbor_iters=\colD,
    obj_k0_iters=\colE,
    obj_k5_iters=\colF,
    obj_k15_iters=\colG,
    obj_k30_iters=\colH,
    obj_k60_iters=\colI,
    reg_k0_iters=\colJ,
    reg_k5_iters=\colK,
    reg_k15_iters=\colL,
    reg_k30_iters=\colM,
    reg_k60_iters=\colN,
    }{\colA & \colC & \colD & \colE & \colF & \colG & \colH & \colI & \colJ & \colK & \colL & \colM & \colN}
    \bottomrule
  \end{tabular}}
  \begin{tabular}{l}
    \reduction
  \end{tabular}
  \adjustbox{max width=\textwidth}{
    \begin{tabular}{lllllllllllll}
    Fp res.&
    \begin{tabular}{@{}c@{}}Cold \\ Start\end{tabular}
    &
    \begin{tabular}{@{}c@{}}Nearest \\ Neighbor\end{tabular}
    &
    \begin{tabular}{@{}c@{}}Fp \\$k=0$ \end{tabular}
    &
    \begin{tabular}{@{}c@{}}Fp \\$k=5$ \end{tabular}
    &
    \begin{tabular}{@{}c@{}}Fp \\$k=15$ \end{tabular}
    &
    \begin{tabular}{@{}c@{}}Fp \\$k=30$ \end{tabular}
    &
    \begin{tabular}{@{}c@{}}Fp \\$k=60$ \end{tabular}
    &
    \begin{tabular}{@{}c@{}}Reg \\$k=0$ \end{tabular}
    &
    \begin{tabular}{@{}c@{}}Reg \\$k=5$ \end{tabular}
    &
    \begin{tabular}{@{}c@{}}Reg \\$k=15$ \end{tabular}
    &
    \begin{tabular}{@{}c@{}}Reg \\$k=30$ \end{tabular}
    &
    \begin{tabular}{@{}c@{}}Reg \\$k=60$ \end{tabular} \\
    \midrule
    \csvreader[head to column names, late after line=\\]{./data/iters_red/robust_kalman.csv}{
    accuracies=\colA,
    cold_start_red=\colC,
    nearest_neighbor_red=\colD,
    obj_k0_red=\colE,
    obj_k5_red=\colF,
    obj_k15_red=\colG,
    obj_k30_red=\colH,
    obj_k60_red=\colI,
    reg_k0_red=\colJ,
    reg_k5_red=\colK,
    reg_k15_red=\colL,
    reg_k30_red=\colM,
    reg_k60_red=\colN,
    }{\colA & \colC & \colD & \colE & \colF & \colG & \colH & \colI & \colJ & \colK & \colL & \colM & \colN}
    \bottomrule
  \end{tabular}}
  \begin{tabular}{l}
  \scstiming
  \end{tabular}
  \adjustbox{max width=\textwidth}{
    \begin{tabular}{lllllllllllll}
    tol.&
    \begin{tabular}{@{}c@{}}Cold \\ Start\end{tabular}
    &
    \begin{tabular}{@{}c@{}}Nearest \\ Neighbor\end{tabular}
    &
    \begin{tabular}{@{}c@{}}Fp \\$k=0$ \end{tabular}
    &
    \begin{tabular}{@{}c@{}}Fp \\$k=5$ \end{tabular}
    &
    \begin{tabular}{@{}c@{}}Fp \\$k=15$ \end{tabular}
    &
    \begin{tabular}{@{}c@{}}Fp \\$k=30$ \end{tabular}
    &
    \begin{tabular}{@{}c@{}}Fp \\$k=60$ \end{tabular}
    &
    \begin{tabular}{@{}c@{}}Reg \\$k=0$ \end{tabular}
    &
    \begin{tabular}{@{}c@{}}Reg \\$k=5$ \end{tabular}
    &
    \begin{tabular}{@{}c@{}}Reg \\$k=15$ \end{tabular}
    &
    \begin{tabular}{@{}c@{}}Reg \\$k=30$ \end{tabular}
    &
    \begin{tabular}{@{}c@{}}Reg \\$k=60$ \end{tabular} \\
    \midrule
    \begin{filecontents*}{./data/timings/robust_kalman.csv}
,rel_tols,abs_tols,cold_start,nearest_neighbor,obj_k0,obj_k5,obj_k15,obj_k30,obj_k60,reg_k0,reg_k5,reg_k15,reg_k30,reg_k60
0,0.1,0.1,0.49,0.49,0.45,0.49,0.49,0.49,0.49,\bf{0.35},0.52,0.49,0.5,0.49
1,0.01,0.01,0.67,0.49,0.56,0.49,\bf{0.48},0.49,0.5,0.49,0.52,0.49,0.5,0.51
2,0.001,0.001,1.36,1.31,1.47,0.5,\bf{0.48},1.14,1.12,1.23,0.53,1.13,1.13,1.11
3,0.0001,0.0001,2.23,2.19,2.33,1.18,\bf{0.94},2.03,1.95,2.12,1.06,1.99,1.99,1.95
4,1e-05,1e-05,3.11,3.13,3.31,2.05,\bf{1.76},2.85,2.82,3.03,1.97,2.86,2.83,2.8
\end{filecontents*}
\csvreader[head to column names, late after line=\\]{./data/timings/robust_kalman.csv}{
    rel_tols=\colA,
    cold_start=\colC,
    nearest_neighbor=\colD,
    obj_k0=\colE,
    obj_k5=\colF,
    obj_k15=\colG,
    obj_k30=\colH,
    obj_k60=\colI,
    reg_k0=\colJ,
    reg_k5=\colK,
    reg_k15=\colL,
    reg_k30=\colM,
    reg_k60=\colN,
    }{\colA & \colC & \colD & \colE & \colF & \colG & \colH & \colI & \colJ & \colK & \colL & \colM & \colN}
    \bottomrule
  \end{tabular}}
\end{table}

\begin{figure}[!h]
  \centering
    \includegraphics[width=\figsize\linewidth]{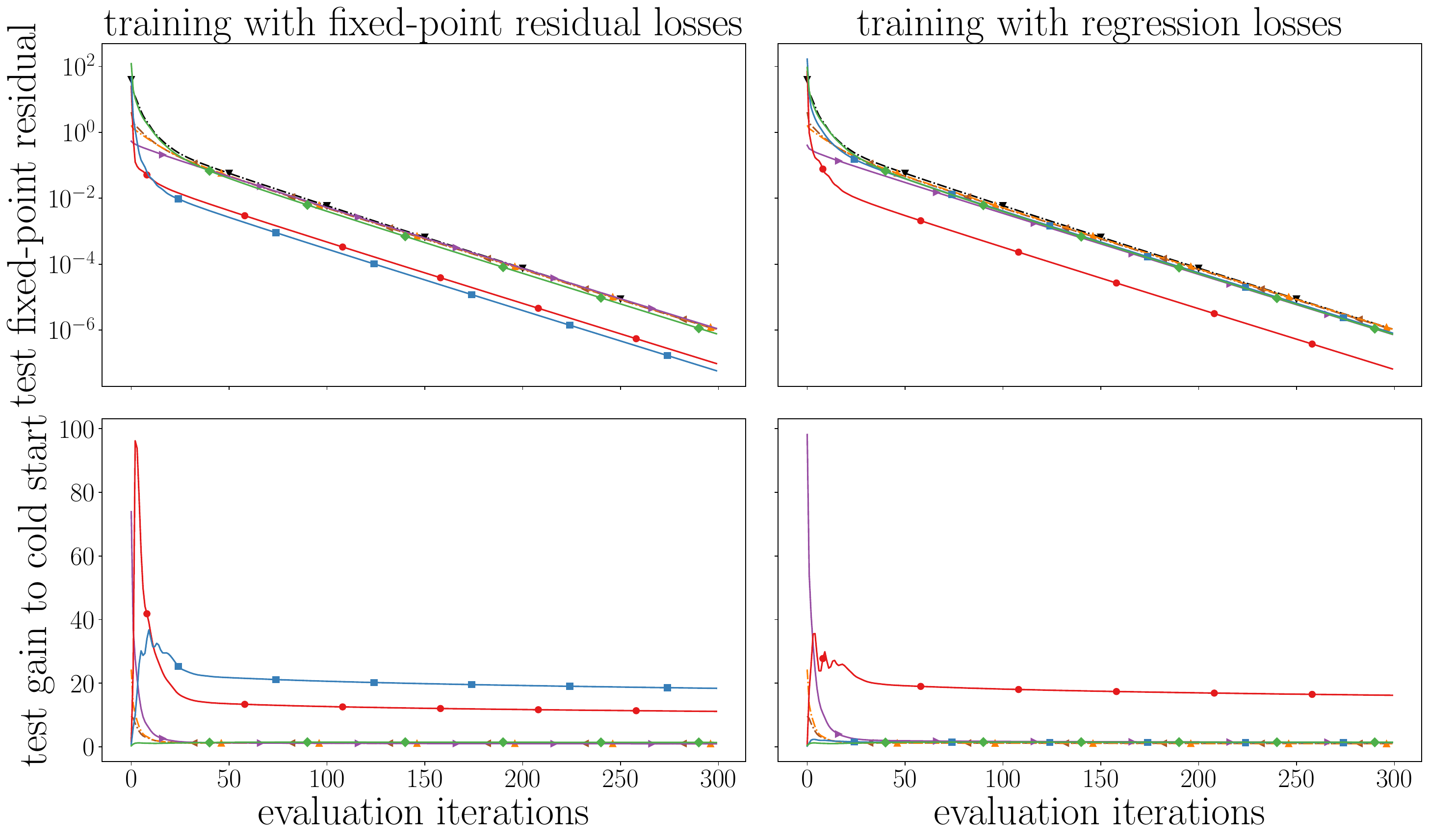}
    \\
    \legendprevsol\\
    \caption{Robust Kalman filtering.
        The learned warm starts that train with $k=5$ for both losses and with $k=15$ for the fixed-point residual loss have significant gains over the baselines. 
    }
    \label{fig:robust_kalman_results}
\end{figure}


\begin{figure}
  \centering
  \begin{subfigure}{.37\linewidth}
    \centering
    \includegraphics[width=1\linewidth]{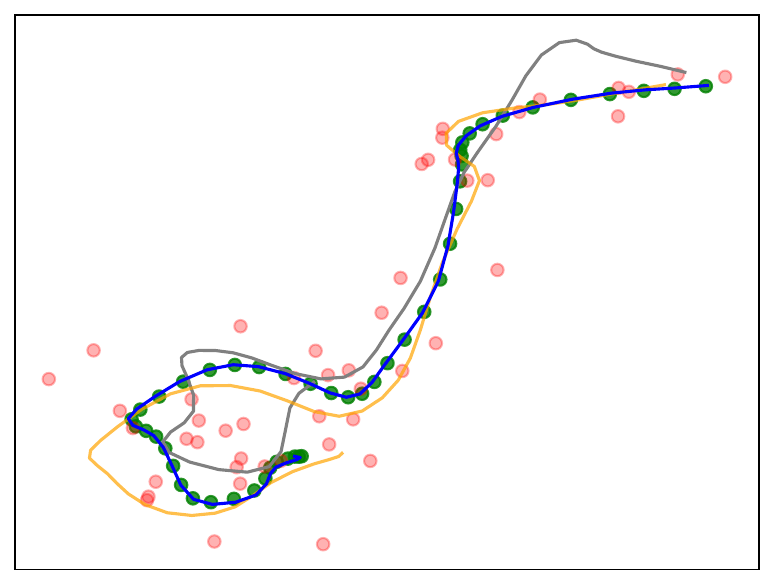}
    \label{fig:rkf_visualize1}
  \end{subfigure}%
  \begin{subfigure}{.37\linewidth}
    \centering
    \includegraphics[width=1\linewidth]{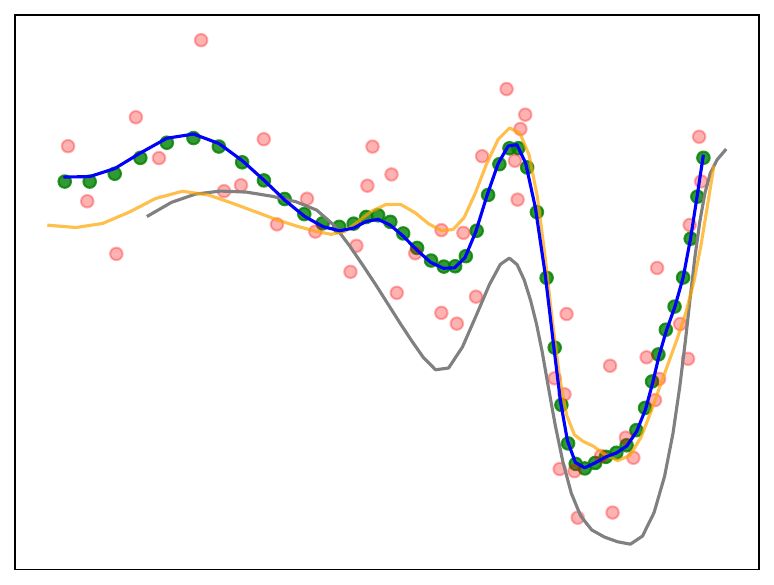}
    \label{fig:rkf_visualize2}
  \end{subfigure}\\
  \ccircle{44}{160}{44} optimal solution \hspace{1mm}
  \ccircle{245}{125}{125} noisy trajectory \hspace{1mm}
  \\
  \cblock{245}{201}{44} nearest neighbor \hspace{1mm}
  \cblock{138}{135}{129} previous solution \hspace{1mm}
  \cblock{0}{0}{255} learned \hspace{1mm}
  \caption{Visualizing test problems for robust Kalman filtering.
  Each plot is a separate test problem.
  The noisy, observed trajectory are the red points which serve as problem data for the SOCP.
  The robust Kalman filtering recovery, the optimal solution of the SOCP, is shown as green dots.
  After $5$ iterations, SCS with our learned warm start using the regression loss with $k=5$ is very close to the optimal solution while SCS initialized with both the shifted previous solution and the nearest neighbor still is noticeably far away from optimality.
  }
  \label{fig:rkf_visualize}
\end{figure}
Since this is a control example, we use the shifted previous solution as a warm start from \Sec~\ref{sec:quadcopter}.
Figure~\ref{fig:robust_kalman_results} and Table~\ref{tab:rkf_table} show the convergence behavior of our method.
In this example, the learned warm starts do well with the fixed-point residual loss for $k=5$ and $k=15$ and the regression loss for $k=5$, but hardly improve in the other cases.
In all cases, the gains relative to the cold start remain nearly constant throughout the evaluation iterations.
Figure~\ref{fig:rkf_visualize} illustrates how our learned solutions after $5$ iterations outperforms the solution returned after $5$ iterations from the baselines.

\subsubsection{Robust non-negative least squares}\label{sec:rls_experiment}
We consider the problem of non-negative least squares with matrix uncertainty
\begin{equation*}
    \begin{array}{ll}
    \label{eq:robust_ls_minmax}
    \mbox{min}_{x \geq 0} & \mbox{max}_{\|\Delta A\| \leq \rho} \|(\hat{A} + \Delta A)x - b\|_2, \\
    \end{array}
\end{equation*}
where the right-hand side vector $b \in \reals^m$, nominal matrix $\hat{A} \in \reals^{m \times n}$, and maximum perturbation $\rho \in \reals_{++}$ are the problem data.
The decision variable of the minimizer and maximizer are $x \in \reals^n$ and $\Delta A \in \reals^{m \times n}$ respectively.
Here, $\|\Delta A\|$ denotes the largest singular value of the perturbation matrix $\Delta A$.
\citet{ElGhaoui_robust_ls} provide an SOCP formulation for this problem
\begin{equation*}
    \begin{array}{ll}
    \mbox{minimize} & u + \rho v \\
    \mbox{subject to} & \|\hat{A} x - b\|_2 \leq u \\
    & \|x\|_2 \leq v\\
    & x \geq 0,
    \end{array}
\end{equation*}
where $x \in \reals^n$, $u \in \reals$, and $v \in \reals$ are the decision variables.
The parameter is $\theta = b$.

\paragraph{Numerical example.}
We pick $\rho=4$ and $\hat{A} \in \reals^{500 \times 800}$ where the entries of $\hat{A}$ are sampled the uniform distribution $\mathcal{U}[-1,1]$.
We sample $b$ in an i.i.d. fashion from $\mathcal{U}[1,2]$.

\paragraph{Results.}
Figure~\ref{fig:robust_ls_results} and table~\ref{tab:rls_times} show the convergence behavior of our method.
The learned warm starts with positive $k$ substantially improve upon the baselines for both losses.
Figure~\ref{fig:robust_ls_results} show linear convergence of our method; this results in the gains from the learned warm starts staying roughly constant as the number of evaluation steps increases.

\begin{figure}[!h]
  \centering
    \includegraphics[width=\figsize\linewidth]{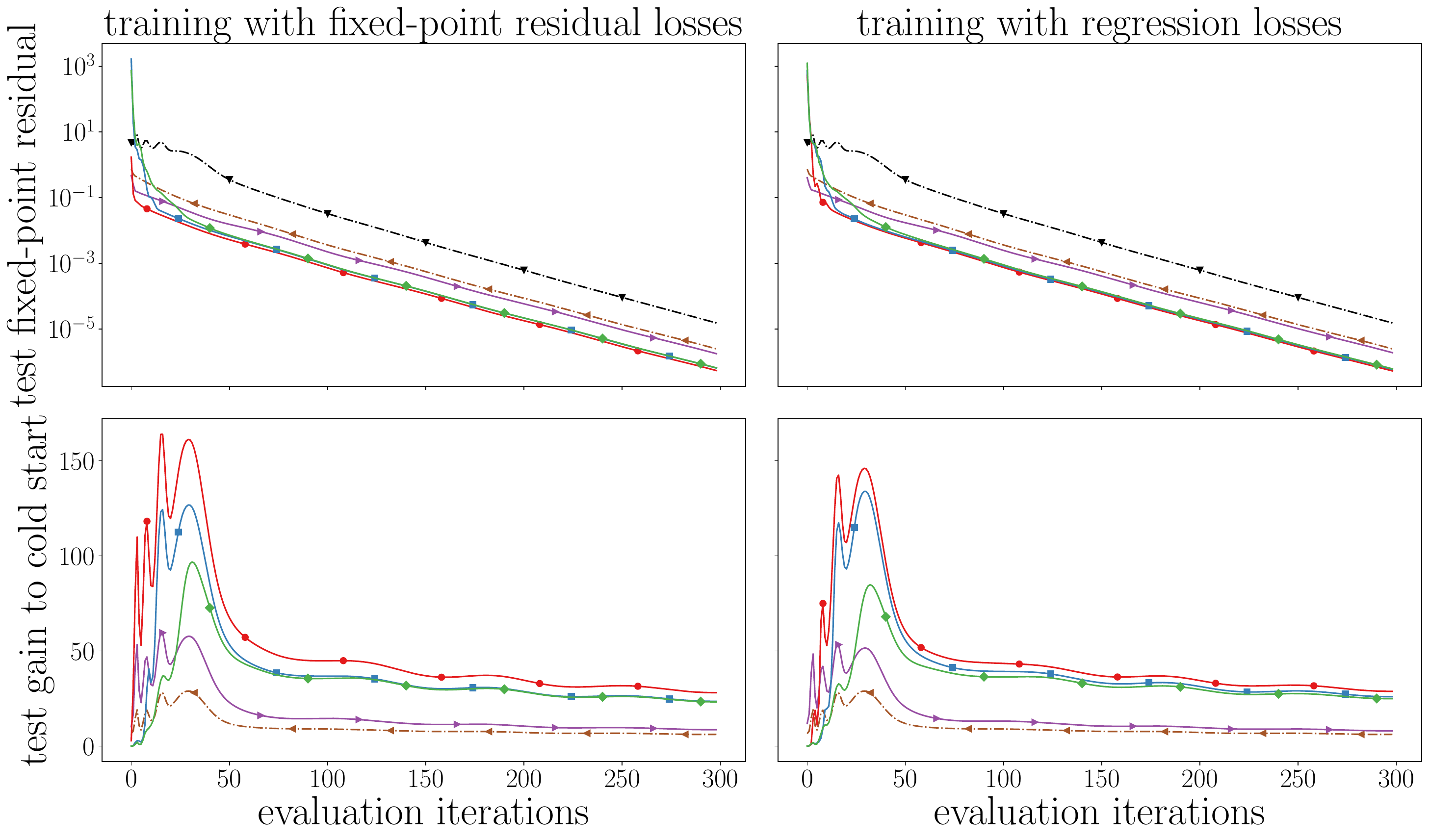}
    \\
    \legend
    \caption{Robust non-negative least squares.
    All of the learned warm starts apart from the ones with $k=0$ substantially improve the gain over the cold start.
    }
    \label{fig:robust_ls_results}
\end{figure}

\begin{table}[!h]
  \centering
  \small
    \renewcommand*{\arraystretch}{1.0}
  \caption{Robust non-negative least squares.
  }
  \label{tab:rls_times}
  \vspace*{-3mm}
  \begin{tabular}{l}
    \iters
  \end{tabular}
  \adjustbox{max width=\textwidth}{
    \begin{tabular}{lllllllllllll}
    Fp res.&
    \begin{tabular}{@{}c@{}}Cold \\ Start\end{tabular}
    &
    \begin{tabular}{@{}c@{}}Nearest \\ Neighbor\end{tabular}
    &
    \begin{tabular}{@{}c@{}}Fp \\$k=0$ \end{tabular}
    &
    \begin{tabular}{@{}c@{}}Fp \\$k=5$ \end{tabular}
    &
    \begin{tabular}{@{}c@{}}Fp \\$k=15$ \end{tabular}
    &
    \begin{tabular}{@{}c@{}}Fp \\$k=30$ \end{tabular}
    &
    \begin{tabular}{@{}c@{}}Fp \\$k=60$ \end{tabular}
    &
    \begin{tabular}{@{}c@{}}Reg \\$k=0$ \end{tabular}
    &
    \begin{tabular}{@{}c@{}}Reg \\$k=5$ \end{tabular}
    &
    \begin{tabular}{@{}c@{}}Reg \\$k=15$ \end{tabular}
    &
    \begin{tabular}{@{}c@{}}Reg \\$k=30$ \end{tabular}
    &
    \begin{tabular}{@{}c@{}}Reg \\$k=60$ \end{tabular} \\
    \midrule
    \begin{filecontents*}{./data/iters_red/robust_ls.csv}
,accuracies,cold_start_iters,cold_start_red,nearest_neighbor_iters,nearest_neighbor_red,obj_k0_iters,obj_k0_red,obj_k5_iters,obj_k5_red,obj_k15_iters,obj_k15_red,obj_k30_iters,obj_k30_red,obj_k60_iters,obj_k60_red,reg_k0_iters,reg_k0_red,reg_k5_iters,reg_k5_red,reg_k15_iters,reg_k15_red,reg_k30_iters,reg_k30_red,reg_k60_iters,reg_k60_red
0,0.1,85,0,33,0.61,20,0.76,\bf{5},\bf{0.94},12,0.86,16,0.81,20,0.76,15,0.82,8,0.91,14,0.84,15,0.82,22,0.74
1,0.01,139,0,87,0.37,75,0.46,45,0.68,50,0.64,52,0.63,53,0.62,68,0.51,39,0.72,41,0.71,\bf{38},\bf{0.73},48,0.65
2,0.001,199,0,146,0.27,134,0.33,103,0.48,108,0.46,109,0.45,109,0.45,127,0.36,95,0.52,97,0.51,\bf{93},\bf{0.53},100,0.5
3,0.0001,258,0,208,0.19,196,0.24,166,0.36,171,0.34,172,0.33,171,0.34,189,0.27,156,0.4,158,0.39,\bf{153},\bf{0.41},161,0.38
\end{filecontents*}
\csvreader[head to column names, late after line=\\]{./data/iters_red/robust_ls.csv}{
    accuracies=\colA,
    cold_start_iters=\colC,
    nearest_neighbor_iters=\colD,
    obj_k0_iters=\colE,
    obj_k5_iters=\colF,
    obj_k15_iters=\colG,
    obj_k30_iters=\colH,
    obj_k60_iters=\colI,
    reg_k0_iters=\colJ,
    reg_k5_iters=\colK,
    reg_k15_iters=\colL,
    reg_k30_iters=\colM,
    reg_k60_iters=\colN,
    }{\colA & \colC & \colD & \colE & \colF & \colG & \colH & \colI & \colJ & \colK & \colL & \colM & \colN}
    \bottomrule
  \end{tabular}}
  \begin{tabular}{l}
    \reduction
  \end{tabular}
  \adjustbox{max width=\textwidth}{
    \begin{tabular}{lllllllllllll}
    Fp res.&
    \begin{tabular}{@{}c@{}}Cold \\ Start\end{tabular}
    &
    \begin{tabular}{@{}c@{}}Nearest \\ Neighbor\end{tabular}
    &
    \begin{tabular}{@{}c@{}}Fp \\$k=0$ \end{tabular}
    &
    \begin{tabular}{@{}c@{}}Fp \\$k=5$ \end{tabular}
    &
    \begin{tabular}{@{}c@{}}Fp \\$k=15$ \end{tabular}
    &
    \begin{tabular}{@{}c@{}}Fp \\$k=30$ \end{tabular}
    &
    \begin{tabular}{@{}c@{}}Fp \\$k=60$ \end{tabular}
    &
    \begin{tabular}{@{}c@{}}Reg \\$k=0$ \end{tabular}
    &
    \begin{tabular}{@{}c@{}}Reg \\$k=5$ \end{tabular}
    &
    \begin{tabular}{@{}c@{}}Reg \\$k=15$ \end{tabular}
    &
    \begin{tabular}{@{}c@{}}Reg \\$k=30$ \end{tabular}
    &
    \begin{tabular}{@{}c@{}}Reg \\$k=60$ \end{tabular} \\
    \midrule
    \csvreader[head to column names, late after line=\\]{./data/iters_red/robust_ls.csv}{
    accuracies=\colA,
    cold_start_red=\colC,
    nearest_neighbor_red=\colD,
    obj_k0_red=\colE,
    obj_k5_red=\colF,
    obj_k15_red=\colG,
    obj_k30_red=\colH,
    obj_k60_red=\colI,
    reg_k0_red=\colJ,
    reg_k5_red=\colK,
    reg_k15_red=\colL,
    reg_k30_red=\colM,
    reg_k60_red=\colN,
    }{\colA & \colC & \colD & \colE & \colF & \colG & \colH & \colI & \colJ & \colK & \colL & \colM & \colN}
    \bottomrule
  \end{tabular}}
  \begin{tabular}{l}
  \scstiming
  \end{tabular}
  \adjustbox{max width=\textwidth}{
    \begin{tabular}{lllllllllllll}
    tol.&
    \begin{tabular}{@{}c@{}}Cold \\ Start\end{tabular}
    &
    \begin{tabular}{@{}c@{}}Nearest \\ Neighbor\end{tabular}
    &
    \begin{tabular}{@{}c@{}}Fp \\$k=0$ \end{tabular}
    &
    \begin{tabular}{@{}c@{}}Fp \\$k=5$ \end{tabular}
    &
    \begin{tabular}{@{}c@{}}Fp \\$k=15$ \end{tabular}
    &
    \begin{tabular}{@{}c@{}}Fp \\$k=30$ \end{tabular}
    &
    \begin{tabular}{@{}c@{}}Fp \\$k=60$ \end{tabular}
    &
    \begin{tabular}{@{}c@{}}Reg \\$k=0$ \end{tabular}
    &
    \begin{tabular}{@{}c@{}}Reg \\$k=5$ \end{tabular}
    &
    \begin{tabular}{@{}c@{}}Reg \\$k=15$ \end{tabular}
    &
    \begin{tabular}{@{}c@{}}Reg \\$k=30$ \end{tabular}
    &
    \begin{tabular}{@{}c@{}}Reg \\$k=60$ \end{tabular} \\
    \midrule
    \begin{filecontents*}{./data/timings/robust_ls.csv}
,rel_tols,abs_tols,cold_start,nearest_neighbor,obj_k0,obj_k5,obj_k15,obj_k30,obj_k60,obj_k120,reg_k0,reg_k5,reg_k15,reg_k30,reg_k60
0,0.1,0.1,61.79,4.16,\bf{2.23},22.13,26.95,29.2,26.73,32.15,2.32,22.44,22.39,22.48,22.14
1,0.01,0.01,101.73,46.01,42.08,24.13,31.61,36.61,36.37,60.92,42.59,26.17,25.47,\bf{23.97},37.31
2,0.001,0.001,141.53,101.75,89.95,61.95,79.9,82.98,76.75,107.01,95.76,63.65,64.69,\bf{62.16},67.32
3,0.0001,0.0001,185.46,145.54,148.4,107.75,142.34,158.18,149.7,173.23,143.77,112.38,115.18,\bf{107.6},119.79
4,1e-05,1e-05,241.16,199.63,187.8,169.57,197.37,210.4,198.4,235.69,195.24,168.44,167.4,\bf{161.76},164.64
\end{filecontents*}
\csvreader[head to column names, late after line=\\]{./data/timings/robust_ls.csv}{
    rel_tols=\colA,
    cold_start=\colC,
    nearest_neighbor=\colD,
    obj_k0=\colE,
    obj_k5=\colF,
    obj_k15=\colG,
    obj_k30=\colH,
    obj_k60=\colI,
    reg_k0=\colJ,
    reg_k5=\colK,
    reg_k15=\colL,
    reg_k30=\colM,
    reg_k60=\colN,
    }{\colA & \colC & \colD & \colE & \colF & \colG & \colH & \colI & \colJ & \colK & \colL & \colM & \colN}
    \bottomrule
  \end{tabular}}
\end{table}


\subsubsection{Phase retrieval}\label{sec:phase_retrieval_experiment}
Our first SDP example is the problem of phase retrieval~\citep{phase_retrieval} where the goal is to recover an unknown signal $x \in \complex^n$ from observations.
This problem has applications in X-ray crystallogpraphy \citep{phase_retrieval_crystallography} and coherent diffractive imaging \citep{phase_retrieval_imaging}.
Specifically, for known vectors $a_i \in \complex^n$, we have $m$ scalar measurements: $b_i = \lvert \langle a_i, x \rangle \rvert^2, \quad i = 1,\dots, m$.
Since the values are complex, we denote the conjugate transpose of a matrix $A$ by $A^*$.
Noting that $\lvert \langle a_i, x \rangle \rvert^2 = (a_i^* x)(x^* a_i),$  we introduce a matrix variable $X\in\symm^n_{+}$ and matrices $A_i = a_i a_i^*$.
The exact phase retrieval problem becomes a feasibility problem over the matrix variable with a rank constraint
\begin{equation*}
  \begin{array}{ll}
  \mbox{find} & X \\
  \mbox{subject to} & \Tr(A_i X) = b_i, \quad i=1,\dots,m \\
  & \Rank(X) = 1,\quad X \succeq 0.
  \end{array}
\end{equation*}
We arrive at the following SDP relaxation by dropping the rank constraint:
\begin{equation*}
  \begin{array}{ll}
  \mbox{minimize} & \Tr(X) \\
  \mbox{subject to} & \Tr(A_i X) = b_i, \quad i=1,\dots,m \\
  & X \succeq 0.
  \end{array}
\end{equation*}
To parameterize each problem, we let $\theta = b \in \reals^m$.

\paragraph{Numerical example.}
For the signal, we sample $x$ from a complex normal distribution, {\it i.e.}, we sample the real and imaginary parts of each component independently from $\mathcal{N}(\mu, \sigma^2)$.
To construct the constraint matrices, we use the coded diffraction pattern model \citep{candes2015phase}.
The specific modulating waveforms follow the setup from \citet[\Sec~F.1]{yurtsever2021scalable}.
For a signal of size $n$, we generate $d = 3n$ measurements.
Specifically, we draw $3$ independent modulating waveforms $\psi_j \in \reals^n, ~ j = 1,\dots,3$.
Each component of $\psi_j$ is the product of two random variables, with one drawn uniformly from $\{1, i, -1, -i\}$ and the other drawn from $\{\sqrt{2}/2, \sqrt{3}\}$ with probabilities $0.8$ and $0.2$, respectively.
Then, each $a_i$ corresponds to computing a single entry of the Fourier transform of $x$ after being modulated by the waveforms.
Letting $W$ be the $n \times n$ discrete Fourier transform matrix, the $a_i$'s can be written explicitly as $a_{(j-1)n + l} = W^T_l(\diag(\psi_j))^\star$ where $W_l^T$ is the $l$-th row of $W$.
We take $n=40$, $\mu=5$, and $\sigma=1$.

\paragraph{Results.}
Figure~\ref{fig:phase_retrieval_results} and Table~\ref{tab:pr_table} show the convergence behavior of our method.
In this case, while the decoupled approach with $k=0$ offers the largest gains over the first few iterations, the gain degrades as $t$ increases to the point where it's performance becomes worse than that of the nearest-neighbor initialization.
The learned warm starts with the regression loss for positive $k$ tend to sustain their gains for a larger value of $t$ compared with the learned warm starts that use the fixed-point residual loss.

\begin{table}[!h]
  \centering
  \small
    \renewcommand*{\arraystretch}{1.0}
  \caption{Phase retrieval.
  }
  \label{tab:pr_table}
  \vspace*{-3mm}
  \begin{tabular}{l}
  \iters
  \end{tabular}
  \adjustbox{max width=\textwidth}{
    \begin{tabular}{lllllllllllll}
    Fp res.&
    \begin{tabular}{@{}c@{}}Cold \\ Start\end{tabular}
    &
    \begin{tabular}{@{}c@{}}Nearest \\ Neighbor\end{tabular}
    &
    \begin{tabular}{@{}c@{}}Fp \\$k=0$ \end{tabular}
    &
    \begin{tabular}{@{}c@{}}Fp \\$k=5$ \end{tabular}
    &
    \begin{tabular}{@{}c@{}}Fp \\$k=15$ \end{tabular}
    &
    \begin{tabular}{@{}c@{}}Fp \\$k=30$ \end{tabular}
    &
    \begin{tabular}{@{}c@{}}Fp \\$k=60$ \end{tabular}
    &
    \begin{tabular}{@{}c@{}}Reg \\$k=0$ \end{tabular}
    &
    \begin{tabular}{@{}c@{}}Reg \\$k=5$ \end{tabular}
    &
    \begin{tabular}{@{}c@{}}Reg \\$k=15$ \end{tabular}
    &
    \begin{tabular}{@{}c@{}}Reg \\$k=30$ \end{tabular}
    &
    \begin{tabular}{@{}c@{}}Reg \\$k=60$ \end{tabular} \\
    \midrule
    \begin{filecontents*}{./data/iters_red/phase_retrieval.csv}
,accuracies,cold_start_iters,cold_start_red,nearest_neighbor_iters,nearest_neighbor_red,obj_k0_iters,obj_k0_red,obj_k5_iters,obj_k5_red,obj_k15_iters,obj_k15_red,obj_k30_iters,obj_k30_red,obj_k60_iters,obj_k60_red,reg_k0_iters,reg_k0_red,reg_k5_iters,reg_k5_red,reg_k15_iters,reg_k15_red,reg_k30_iters,reg_k30_red,reg_k60_iters,reg_k60_red
0,1.0,112,0,47,0.58,11,0.9,\bf{4},\bf{0.96},6,0.95,7,0.94,11,0.9,12,0.89,15,0.87,16,0.86,19,0.83,23,0.79
1,0.1,531,0,254,0.52,320,0.4,280,0.47,209,0.61,96,0.82,\bf{76},\bf{0.86},322,0.39,122,0.77,115,0.78,117,0.78,129,0.76
2,0.01,2558,0,1629,0.36,2259,0.12,2710,-0.06,2890,-0.13,2626,-0.03,1982,0.23,2249,0.12,666,0.74,669,0.74,\bf{641},\bf{0.75},686,0.73
3,0.001,6478,0,5547,0.14,6133,0.05,6803,-0.05,7113,-0.1,6861,-0.06,6127,0.05,6098,0.06,3609,0.44,3611,0.44,\bf{3408},\bf{0.47},3667,0.43
4,0.0001,11512,0,10837,0.06,11281,0.02,12021,-0.04,12379,-0.08,12107,-0.05,11303,0.02,11223,0.03,8322,0.28,8268,0.28,\bf{7888},\bf{0.31},8217,0.29
\end{filecontents*}
\csvreader[head to column names, late after line=\\]{./data/iters_red/phase_retrieval.csv}{
    accuracies=\colA,
    cold_start_iters=\colC,
    nearest_neighbor_iters=\colD,
    obj_k0_iters=\colE,
    obj_k5_iters=\colF,
    obj_k15_iters=\colG,
    obj_k30_iters=\colH,
    obj_k60_iters=\colI,
    reg_k0_iters=\colJ,
    reg_k5_iters=\colK,
    reg_k15_iters=\colL,
    reg_k30_iters=\colM,
    reg_k60_iters=\colN,
    }{\colA & \colC & \colD & \colE & \colF & \colG & \colH & \colI & \colJ & \colK & \colL & \colM & \colN}
    \bottomrule
  \end{tabular}}
  \begin{tabular}{l}
    \reduction
  \end{tabular}
  \adjustbox{max width=\textwidth}{
    \begin{tabular}{lllllllllllll}
    Fp res.&
    \begin{tabular}{@{}c@{}}Cold \\ Start\end{tabular}
    &
    \begin{tabular}{@{}c@{}}Nearest \\ Neighbor\end{tabular}
    &
    \begin{tabular}{@{}c@{}}Fp \\$k=0$ \end{tabular}
    &
    \begin{tabular}{@{}c@{}}Fp \\$k=5$ \end{tabular}
    &
    \begin{tabular}{@{}c@{}}Fp \\$k=15$ \end{tabular}
    &
    \begin{tabular}{@{}c@{}}Fp \\$k=30$ \end{tabular}
    &
    \begin{tabular}{@{}c@{}}Fp \\$k=60$ \end{tabular}
    &
    \begin{tabular}{@{}c@{}}Reg \\$k=0$ \end{tabular}
    &
    \begin{tabular}{@{}c@{}}Reg \\$k=5$ \end{tabular}
    &
    \begin{tabular}{@{}c@{}}Reg \\$k=15$ \end{tabular}
    &
    \begin{tabular}{@{}c@{}}Reg \\$k=30$ \end{tabular}
    &
    \begin{tabular}{@{}c@{}}Reg \\$k=60$ \end{tabular} \\
    \midrule
    \csvreader[head to column names, late after line=\\]{./data/iters_red/phase_retrieval.csv}{
    accuracies=\colA,
    cold_start_red=\colC,
    nearest_neighbor_red=\colD,
    obj_k0_red=\colE,
    obj_k5_red=\colF,
    obj_k15_red=\colG,
    obj_k30_red=\colH,
    obj_k60_red=\colI,
    reg_k0_red=\colJ,
    reg_k5_red=\colK,
    reg_k15_red=\colL,
    reg_k30_red=\colM,
    reg_k60_red=\colN,
    }{\colA & \colC & \colD & \colE & \colF & \colG & \colH & \colI & \colJ & \colK & \colL & \colM & \colN}
    \bottomrule
  \end{tabular}}
  \begin{tabular}{l}
    \scstiming
  \end{tabular}
  \adjustbox{max width=\textwidth}{
    \begin{tabular}{lllllllllllll}
    tol.&
    \begin{tabular}{@{}c@{}}Cold \\ Start\end{tabular}
    &
    \begin{tabular}{@{}c@{}}Nearest \\ Neighbor\end{tabular}
    &
    \begin{tabular}{@{}c@{}}Fp \\$k=0$ \end{tabular}
    &
    \begin{tabular}{@{}c@{}}Fp \\$k=5$ \end{tabular}
    &
    \begin{tabular}{@{}c@{}}Fp \\$k=15$ \end{tabular}
    &
    \begin{tabular}{@{}c@{}}Fp \\$k=30$ \end{tabular}
    &
    \begin{tabular}{@{}c@{}}Fp \\$k=60$ \end{tabular}
    &
    \begin{tabular}{@{}c@{}}Reg \\$k=0$ \end{tabular}
    &
    \begin{tabular}{@{}c@{}}Reg \\$k=5$ \end{tabular}
    &
    \begin{tabular}{@{}c@{}}Reg \\$k=15$ \end{tabular}
    &
    \begin{tabular}{@{}c@{}}Reg \\$k=30$ \end{tabular}
    &
    \begin{tabular}{@{}c@{}}Reg \\$k=60$ \end{tabular} \\
    \midrule
    \begin{filecontents*}{./data/timings/phase_retrieval.csv}
,rel_tols,abs_tols,cold_start,nearest_neighbor,obj_k0,obj_k5,obj_k15,obj_k30,obj_k60,reg_k0,reg_k5,reg_k15,reg_k30,reg_k60
0,0.1,0.1,72.6,48.5,35.3,\bf{34.9},35.3,35.4,35.2,35.3,35.1,35.2,35.4,35.4
1,0.01,0.01,490.7,257.6,495.5,512.8,364.8,208.8,\bf{118.4},500.3,140.4,127.9,137.4,146.9
2,0.001,0.001,3230,1979,3152.5,3732,3926,3575,2813,3111,\bf{1001},1044,1014,1047
3,0.0001,0.0001,8359,6365,7965,8830,9191,8851,7883,7924,4888,5003,\bf{4844},5071
4,1e-05,1e-05,12673,11113,12237,12514,12698,12712,12491,12232,10954,10845,\bf{10147},10865
\end{filecontents*}
\csvreader[head to column names, late after line=\\]{./data/timings/phase_retrieval.csv}{
    rel_tols=\colA,
    cold_start=\colC,
    nearest_neighbor=\colD,
    obj_k0=\colE,
    obj_k5=\colF,
    obj_k15=\colG,
    obj_k30=\colH,
    obj_k60=\colI,
    reg_k0=\colJ,
    reg_k5=\colK,
    reg_k15=\colL,
    reg_k30=\colM,
    reg_k60=\colN,
    }{\colA & \colC & \colD & \colE & \colF & \colG & \colH & \colI & \colJ & \colK & \colL & \colM & \colN}
    \bottomrule
  \end{tabular}}
\end{table}
\begin{figure}[!h]
  \centering
    \includegraphics[width=\figsize\linewidth]{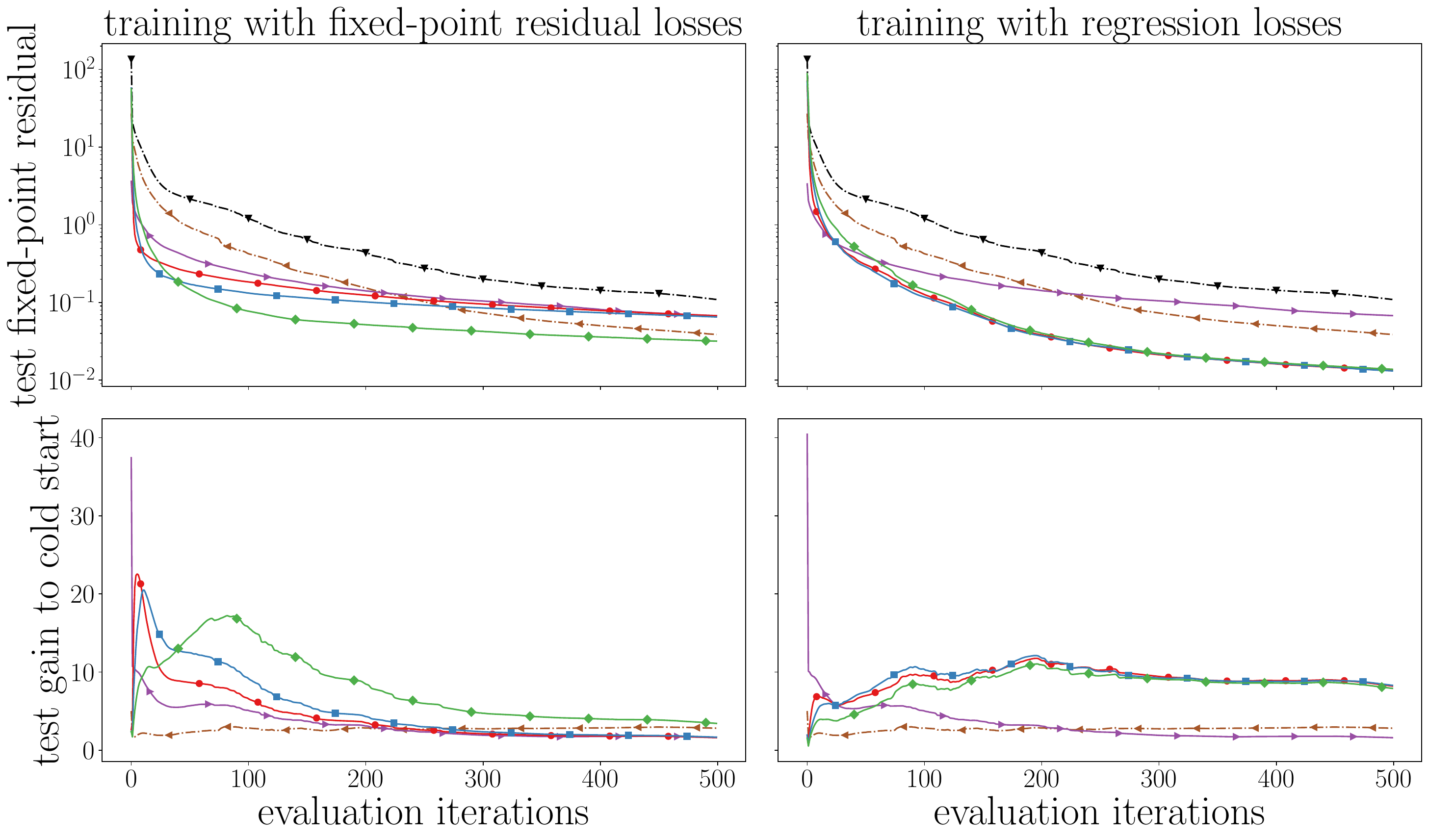}
    \\
    \legend\\
    \caption{Phase retrieval results.
        Other than the $k=0$ case, the learned warm starts with regression loss improvements are maintained for many evaluation steps.
    }
    \label{fig:phase_retrieval_results}
\end{figure}

\subsubsection{Sparse PCA}\label{sec:sparse_pca_experiment}
Next, we examine the problem of sparse PCA~\citep{sparse_PCA}.
Unlike standard PCA~\citep{pca}, which typically finds principal components that depend on all observed variables, sparse PCA identifies principal components that rely on only a small subset of the variables.
The Sparse PCA problem is
\begin{equation}\label{prob:sPCA_noncvx}
  \begin{array}{ll}
  \mbox{maximize} & x^T A x \\
  \mbox{subject to} & \|x\|_2 \leq 1, \quad \card(x) \leq c, \\
  \end{array}
\end{equation}
where $x \in \reals^n$ is the decision variable and $\card(x)$ is the number of nonzero terms of vector $x$.
The covariance matrix $A \in \symm_+^{n}$ and desired cardinality $c \in \reals_+$ are problem data.
We consider an SDP relaxation of the non-convex problem~\eqref{prob:sPCA_noncvx} which takes the form
\begin{equation}\label{prob:sPCA_sdp}
  \begin{array}{ll}
  \mbox{maximize} & \Tr(A X) \\
  \mbox{subject to} & \Tr(X) = 1 \\
  & \mathbf{1}^T |X| \mathbf{1} \leq c\\
  & X \succeq 0,
  \end{array}
\end{equation}
where the decision variable is $X \in \symm_+^{n}$.
We use an $r$-factor model~\citep{cvx_portfolio} and set $A = F \Sigma F^T$ where $F \in \reals^{n \times r}$ is the factor loading matrix and $\Sigma \in \mathbf{S}^n_+$ is a matrix that holds the factor scores.
The parameter is $\theta = \vect(\Sigma)$.

\paragraph{Numerical example.}
We run our experiments with matrix size of $n=40$, a factor size of $r=10$, and a cardinality size of $c=10$.
To generate the covariance matrices, we first generate a random nominal matrix $A_0$, whose entries are sampled as an i.i.d. standard Gaussian.
We then take the singular value decomposition of $A_0 = U \Sigma_0 U^T$, and let the shared factor loading matrix $F \in \reals^{n \times r}$ be the first $r$ singular vectors of $U$.
Let $B_0 \in \symm_+^{r}$ be the diagonal matrix found by taking the square root of the first $r$ singular values of $A_0$.
Then, for each problem, we take $\Sigma = B B^T$ where $B = \Delta + B_0$.
Here, the elements of $\Delta \in \reals^{r \times r}$ are sampled i.i.d. from the uniform distribution $\mathcal{U}[-0.1, 0.1]$.

\paragraph{Results.}

\begin{table}[!h]
  \centering
  \small
    \renewcommand*{\arraystretch}{1.0}
  \caption{Sparse PCA.
  }
  \label{tab:sparse_pca_tab}
  \vspace*{-3mm}
  \begin{tabular}{l}
  \iters
  \end{tabular}
  \adjustbox{max width=\textwidth}{
    \begin{tabular}{lllllllllllll}
    Fp res.&
    \begin{tabular}{@{}c@{}}Cold \\ Start\end{tabular}
    &
    \begin{tabular}{@{}c@{}}Nearest \\ Neighbor\end{tabular}
    &
    \begin{tabular}{@{}c@{}}Fp \\$k=0$ \end{tabular}
    &
    \begin{tabular}{@{}c@{}}Fp \\$k=5$ \end{tabular}
    &
    \begin{tabular}{@{}c@{}}Fp \\$k=15$ \end{tabular}
    &
    \begin{tabular}{@{}c@{}}Fp \\$k=30$ \end{tabular}
    &
    \begin{tabular}{@{}c@{}}Fp \\$k=60$ \end{tabular}
    &
    \begin{tabular}{@{}c@{}}Reg \\$k=0$ \end{tabular}
    &
    \begin{tabular}{@{}c@{}}Reg \\$k=5$ \end{tabular}
    &
    \begin{tabular}{@{}c@{}}Reg \\$k=15$ \end{tabular}
    &
    \begin{tabular}{@{}c@{}}Reg \\$k=30$ \end{tabular}
    &
    \begin{tabular}{@{}c@{}}Reg \\$k=60$ \end{tabular} \\
    \midrule
    \begin{filecontents*}{./data/iters_red/sparse_pca.csv}
,accuracies,cold_start_iters,cold_start_red,nearest_neighbor_iters,nearest_neighbor_red,obj_k0_iters,obj_k0_red,obj_k5_iters,obj_k5_red,obj_k15_iters,obj_k15_red,obj_k30_iters,obj_k30_red,obj_k60_iters,obj_k60_red,reg_k0_iters,reg_k0_red,reg_k5_iters,reg_k5_red,reg_k15_iters,reg_k15_red,reg_k60_iters,reg_k60_red,reg_k30_iters,reg_k30_red
0,0.1,26,0,1,0.96,\bf{0},\bf{1.0},2,0.92,8,0.69,10,0.62,10,0.62,\bf{0},1.0,5,0.81,8,0.69,17,0.35,12,0.54
1,0.01,122,0,62,0.49,262,-1.15,92,0.25,70,0.43,\bf{33},\bf{0.73},38,0.69,262,-1.15,\bf{33},\bf{0.73},37,0.7,55,0.55,44,0.64
2,0.001,338,0,269,0.2,491,-0.45,313,0.07,289,0.14,169,0.5,160,0.53,490,-0.45,\bf{145},\bf{0.57},151,0.55,172,0.49,154,0.54
3,0.0001,982,0,822,0.16,1000,-0.02,881,0.1,935,0.05,766,0.22,738,0.25,1000,-0.02,\bf{681},\bf{0.31},698,0.29,709,0.28,689,0.3
\end{filecontents*}
\csvreader[head to column names, late after line=\\]{./data/iters_red/sparse_pca.csv}{
    accuracies=\colA,
    cold_start_iters=\colC,
    nearest_neighbor_iters=\colD,
    obj_k0_iters=\colE,
    obj_k5_iters=\colF,
    obj_k15_iters=\colG,
    obj_k30_iters=\colH,
    obj_k60_iters=\colI,
    reg_k0_iters=\colJ,
    reg_k5_iters=\colK,
    reg_k15_iters=\colL,
    reg_k30_iters=\colM,
    reg_k60_iters=\colN,
    }{\colA & \colC & \colD & \colE & \colF & \colG & \colH & \colI & \colJ & \colK & \colL & \colM & \colN}
    \bottomrule
  \end{tabular}}
  \begin{tabular}{l}
    \reduction
  \end{tabular}
  \adjustbox{max width=\textwidth}{
    \begin{tabular}{lllllllllllll}
    Fp res.&
    \begin{tabular}{@{}c@{}}Cold \\ Start\end{tabular}
    &
    \begin{tabular}{@{}c@{}}Nearest \\ Neighbor\end{tabular}
    &
    \begin{tabular}{@{}c@{}}Fp \\$k=0$ \end{tabular}
    &
    \begin{tabular}{@{}c@{}}Fp \\$k=5$ \end{tabular}
    &
    \begin{tabular}{@{}c@{}}Fp \\$k=15$ \end{tabular}
    &
    \begin{tabular}{@{}c@{}}Fp \\$k=30$ \end{tabular}
    &
    \begin{tabular}{@{}c@{}}Fp \\$k=60$ \end{tabular}
    &
    \begin{tabular}{@{}c@{}}Reg \\$k=0$ \end{tabular}
    &
    \begin{tabular}{@{}c@{}}Reg \\$k=5$ \end{tabular}
    &
    \begin{tabular}{@{}c@{}}Reg \\$k=15$ \end{tabular}
    &
    \begin{tabular}{@{}c@{}}Reg \\$k=30$ \end{tabular}
    &
    \begin{tabular}{@{}c@{}}Reg \\$k=60$ \end{tabular} \\
    \midrule
    \csvreader[head to column names, late after line=\\]{./data/iters_red/sparse_pca.csv}{
    accuracies=\colA,
    cold_start_red=\colC,
    nearest_neighbor_red=\colD,
    obj_k0_red=\colE,
    obj_k5_red=\colF,
    obj_k15_red=\colG,
    obj_k30_red=\colH,
    obj_k60_red=\colI,
    reg_k0_red=\colJ,
    reg_k5_red=\colK,
    reg_k15_red=\colL,
    reg_k30_red=\colM,
    reg_k60_red=\colN,
    }{\colA & \colC & \colD & \colE & \colF & \colG & \colH & \colI & \colJ & \colK & \colL & \colM & \colN}
    \bottomrule
  \end{tabular}}
  \begin{tabular}{l}
    \scstiming
  \end{tabular}
  \adjustbox{max width=\textwidth}{
    \begin{tabular}{lllllllllllll}
    tol.&
    \begin{tabular}{@{}c@{}}Cold \\ Start\end{tabular}
    &
    \begin{tabular}{@{}c@{}}Nearest \\ Neighbor\end{tabular}
    &
    \begin{tabular}{@{}c@{}}Fp \\$k=0$ \end{tabular}
    &
    \begin{tabular}{@{}c@{}}Fp \\$k=5$ \end{tabular}
    &
    \begin{tabular}{@{}c@{}}Fp \\$k=15$ \end{tabular}
    &
    \begin{tabular}{@{}c@{}}Fp \\$k=30$ \end{tabular}
    &
    \begin{tabular}{@{}c@{}}Fp \\$k=60$ \end{tabular}
    &
    \begin{tabular}{@{}c@{}}Reg \\$k=0$ \end{tabular}
    &
    \begin{tabular}{@{}c@{}}Reg \\$k=5$ \end{tabular}
    &
    \begin{tabular}{@{}c@{}}Reg \\$k=15$ \end{tabular}
    &
    \begin{tabular}{@{}c@{}}Reg \\$k=30$ \end{tabular}
    &
    \begin{tabular}{@{}c@{}}Reg \\$k=60$ \end{tabular} \\
    \midrule
    \begin{filecontents*}{./data/timings/sparse_pca.csv}
,rel_tols,abs_tols,cold_start,nearest_neighbor,obj_k0,obj_k5,obj_k15,obj_k30,obj_k60,reg_k0,reg_k5,reg_k15,reg_k60,reg_k30
0,0.1,0.1,9.66,0.59,\bf{0.65},10.62,8.35,8.22,8.22,0.73,10.06,8.27,8.4,8.25
1,0.01,0.01,26.34,9.63,97.97,14.31,8.81,8.25,\bf{8.22},88.35,10.09,8.25,9.94,8.3
2,0.001,0.001,67.88,45.91,121.94,71.17,43.23,17.39,17.23,110.4,20.09,\bf{16.92},25.31,18.56
3,0.0001,0.0001,136.21,107.61,187.43,140.64,100.76,66.36,63.95,170.05,73.65,\bf{61.11},70.46,62.77
4,1e-05,1e-05,255.52,204.37,309.51,264.94,201.32,156.11,151.28,279.27,174.63,\bf{144.53},153.85,146.15
\end{filecontents*}
\csvreader[head to column names, late after line=\\]{./data/timings/sparse_pca.csv}{
    rel_tols=\colA,
    cold_start=\colC,
    nearest_neighbor=\colD,
    obj_k0=\colE,
    obj_k5=\colF,
    obj_k15=\colG,
    obj_k30=\colH,
    obj_k60=\colI,
    reg_k0=\colJ,
    reg_k5=\colK,
    reg_k15=\colL,
    reg_k30=\colM,
    reg_k60=\colN,
    }{\colA & \colC & \colD & \colE & \colF & \colG & \colH & \colI & \colJ & \colK & \colL & \colM & \colN}
    \bottomrule
  \end{tabular}}
\end{table}

\begin{figure}[!h]
  \centering
    \includegraphics[width=0.7\linewidth]{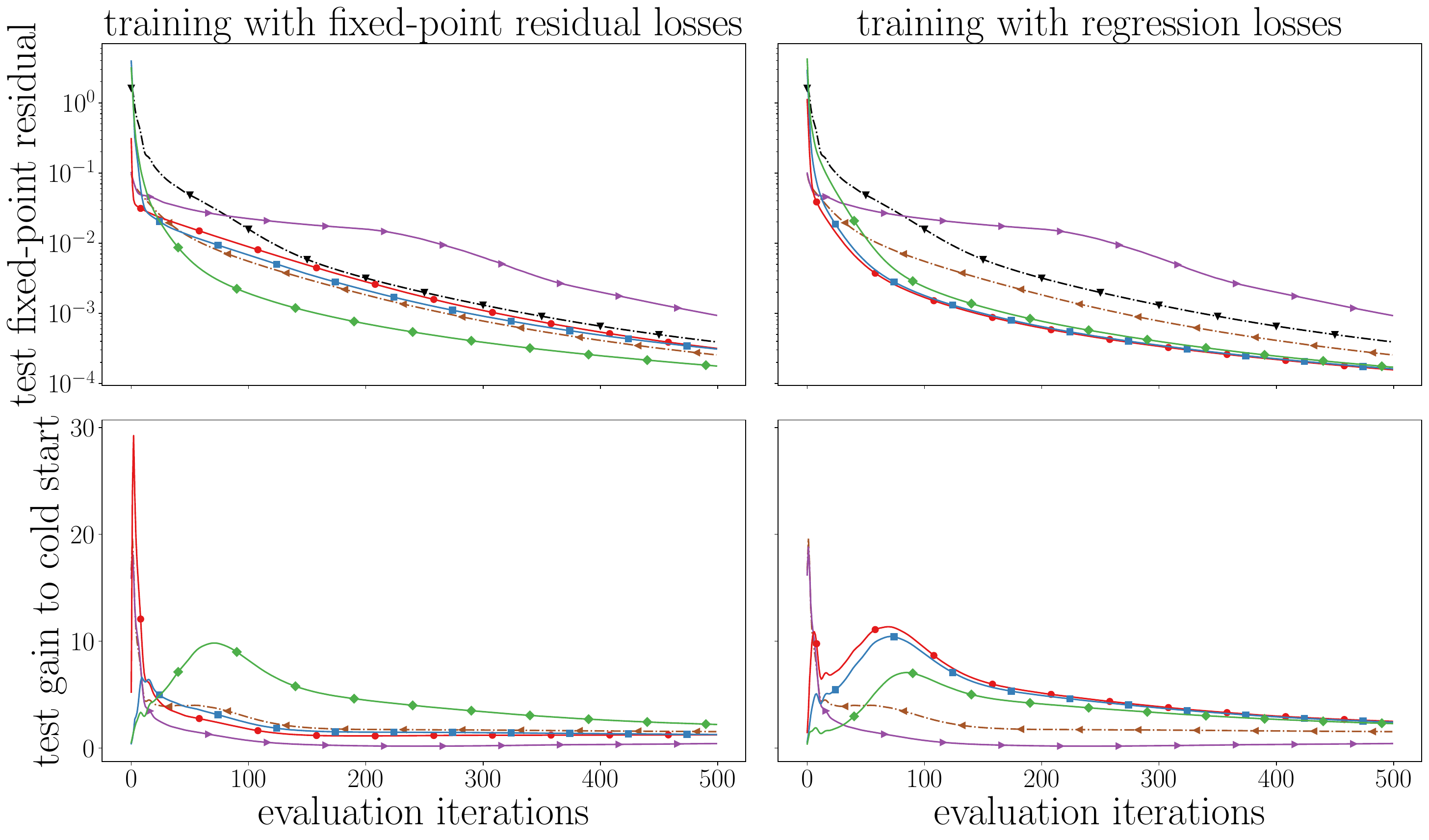}
    \\
    \legend\\
    \caption{Sparse PCA.
      The learned warm starts with positive $k$ that use the regression loss provide large gains.
    }
    \label{fig:sparse_pca_results}
\end{figure}

Figure~\ref{fig:sparse_pca_results} and table~\ref{tab:sparse_pca_tab} show the convergence behavior of our method.
In this example, both the fixed-point residual loss and the regression loss perform with $k=0$.
All of the other learned warm starts with the regression loss and some with the fixed-point residual loss show good performance.

\appendix
\section{Examples of fixed-point algorithms}\label{app:fp_algos}
\paragraph{Gradient descent.} Here, $z \in \reals^n$ is the decision variable and $\fth$ is a convex and $L$-smooth function.
Recall that $f  : \reals^n \rightarrow \reals$ is $L$-smooth if $\|\nabla f(x) -\nabla f(y) \|_2 \leq L \|x - y\|_2 \quad \forall x, y \in \reals^n$.
If $\alpha \in (0, 2 / L)$, then the iterates of gradient descent are guaranteed to converge to an optimal solution~\citep{mon_primer}.
If $\fth$ is strongly convex, then the fixed-point operator is a contraction~\citep{mon_primer}.

\paragraph{Proximal gradient descent.}
Here, $z \in \reals^n$ is the decision variable, $\fth$ is a convex and $L$-smooth function, and $\gth$ is a convex but possibly non-smooth function.
The iterations of proximal gradient descent converge to a solution if $\alpha \in (0, 2 / L)$ \citep{prox_algos}.

\paragraph{Alternating direction method of multipliers (ADMM).}
Here, $u \in \reals^n$ is the decision variable and $\fth$ and $\gth$ are closed, convex, proper, and possibly non-smooth functions.
The iterations of ADMM generate a sequence of iterates, resulting in the convergence of both $\tilde{u}^i$ and $u^i$ to each other and to a solution of the problem.
The $z \in \reals^n$ variable serves as the associated dual variable.
We use the equivalence of ADMM to Douglas-Rachford splitting \citep{admm_dr_equiv} and write the Douglas-Rachford splitting iterations in Table~\ref{table:fp_algorithms}.
While the associated fixed-point operator to ADMM is averaged~\citep{mon_primer}, ADMM is known to converge linearly under certain conditions~\citep{eckstein_thesis,giselsson_lin_conv}.

\paragraph{OSQP.}
The operator splitting quadratic program (OSQP)~\citep{osqp} solver is based on ADMM.
Here, $P \in \symm_+^n$, $A \in \reals^{m \times n}$, $c \in \reals^n$, $l \in \reals^m$, and $u \in \reals^m$ are problem data, and $\Pi_{[l,u]}$ is the projection onto the box, $[l,u]$.
The decision variable is $x \in \reals^n$.
While the algorithm uses $x$, $w$, and $y$ variables, the fixed-point operator is represented as an operator on a smaller vector, as shown in~\citet{infeas_detection}.

\paragraph{SCS.}
The splitting conic solver (SCS)~\citep{scs_quadratic} is also based on ADMM.
Here, $P \in \symm_+^n$, $A \in \reals^{m \times n}$, $c \in \reals^n$, and $b \in \reals^m$ are problem data, and $\Pi_{\mathcal{C}}$ is the projection onto the cone $\mathcal{C}$.
The decision variables are $x \in \reals^n$ and $s \in \reals^m$.
For simplicity, Table~\ref{table:fp_algorithms} includes the simplified version of the SCS algorithm without the homogeneous self-dual embedding.
The SCS algorithm, the one we use in the numerical experiments in \Sec~\ref{sec:scs_experiments}, is based on the homogeneous self-dual embedding; see~\citet{scs_quadratic} for the details.
As in~\citet{neural_fp_accel_amos}, our implementation normalizes the fixed-point residual by the $\tau$ scaling factor to ensure that the fixed-point residual is not artificially small.

\section{Proofs}

\subsection{Proof of \Lem~\texorpdfstring{\ref{lemma:marg}}{}}\label{proof:marg}
Let $w' = w + u$ and let $S_w$ be the set of perturbations $w'$ such that
\begin{equation*}
\textstyle
  S_w \subset \left\{w' \mid \max_{\theta \in \Theta} \|h_{w'}(\theta) - \hw(\theta)\|_2 \leq \gamma / 2\right\}.
\end{equation*}
Let $q$ be the probability density function over $w'$.
We construct a new distribution $\tilde{Q}$ over predictors $h_{\tilde{w}}$ where $\tilde{w}$ is restricted to $S_w$ with the probability density function $\tilde{q}(\tilde{w}) = (1/Z)q(\tilde{w})$ if $\tilde{w} \in S_{w}$ and otherwise $0$,
where $Z$ is a normalizing constant.
By the assumption of the lemma, $Z = \mathbf{P}(w' \ \in S_w) \geq 1 / 2$.
By the definition of $\tilde{Q}$, we have
\begin{equation*}
\textstyle
  \max_{\theta \in \Theta} \|h_{\tilde{w}}(\theta) - \hw(\theta)\|_2 \leq \gamma / 2.
\end{equation*}
Therefore, $\Lot[\hw(\theta)] \leq g_{\gamma / 2, \theta}^t(h_{\tilde{w}}(\theta)) \leq g_{\gamma, \theta}^t(\hw(\theta))$ almost surely for every $\theta \in \Theta$.
Hence, for every $\tilde{w}$ drawn from the probability density function $\tilde{Q}$, almost surely,
\begin{equation}\label{eq:lemm_pf_bd1}
\textstyle
    R^t(w) \leq R^t_{\gamma/2}(\tilde{w}), \quad \hat{R}^t_{\gamma/2}(\tilde{w}) \leq \hat{R}^t_{\gamma}(w).
\end{equation}
Now using these two inequalities above and the PAC-Bayes theorem, we get
\begin{talign*}
  R^t(w) &\leq \mathbf{E}_{\tilde{w}}[R^t_{\gamma/2}(\tilde{w})] \\
  &\leq \mathbf{E}_{\tilde{w}}[\hat{R}^t_{\gamma/2}(\tilde{w})] + 2\Cgtt \sqrt{(2 KL(\tilde{w} || \prior) + \log(2N/\delta)) / (N - 1)} \\
  & \leq \hat{R}^t_{\gamma}(w) + 2 \Cgtt \sqrt{(2 KL (\tilde{w} || \prior) + \log(2N/\delta)) / (N - 1)} \\
  & \leq \hat{R}^t_{\gamma}(w) + 4 \Cgtt \sqrt{(2 KL (w' || \prior) + \log(6N/\delta)) / (N - 1)}.
\end{talign*}
The first and third inequalities come from~\eqref{eq:lemm_pf_bd1}, and the second inequality follows from~\eqref{eq:_pac_bayes}.
The last inequality comes from the following calculation which we repeat from~\citet[\Sec~4]{PAC_Bayes_nn}.
Let $S_w^c$ denote the complement of $S_w$ and $\tilde{q}^c$ denote the density function $q$ restricted to $S_w^c$ and normalized.
Then we get
\begin{equation*}
  {\rm KL}(q||p) = Z {\rm KL}(\tilde{q}||p) + (1 - Z){\rm KL}(\tilde{q}^c||p) - H(Z),
\end{equation*}
where $H(Z) = -Z \log Z - (1 - Z) \log(1 - Z)$ is the binary entropy function.
Since the KL-divergence is always positive,
\begin{equation*}
\textstyle
  {\rm KL}(\tilde{q}||p) =  [{\rm KL}(q||p) + H(Z) - (1 - Z)KL(\tilde{q}^c||p)] / Z \leq 2({\rm KL}(q||p) + 1).
\end{equation*}
Using the additive properties of logarithms, $1 + \log(2N / \delta) \leq \log(6N / \delta)$.


\subsection{Proof of \Thm~\texorpdfstring{\ref{thm:gen_pac}}{}}\label{proof:gen_pac}
Our proof follows a similar structure as the proof of~\citet[\Thm~1]{PAC_Bayes_nn}.
Let $\zeta = (\Pi_{i=1}^L \|W_i\|_2)^{1/L}$ and consider a neural network with weights $\tilde{W}_i = \zeta W_i / \|W_i\|_2$.
Due to the homogeneity of the ReLU, we have $\hw(\theta) = h_{\tilde{w}}(\theta)$ for all $\theta \in \Theta$~\citep{PAC_Bayes_nn}.
Since $(\Pi_{i=1}^L \|W_i\|_2)^{1/L} = (\Pi_{i=1}^L \|\tilde{W}_i\|_2)^{1/L}$ and $\|W_i\|_F / \|W_i\|_2 = \|\tilde{W}_i\|_F / \|\tilde{W}_i\|_2$, inequality~\eqref{eq:gen_pac} is the same for $w$ and $\tw$.
Therefore, it is sufficient to prove the theorem only for the normalized weights $\tilde{w}$ and we can assume that the spectral norm of the weight matrix is equal across all layers, \ie, $\|W_i\|_2=\zeta$.
Now, we break our proof into two cases depending on the product of the spectral norm of the weight matrices.
The main difference between our proof and the proof for~\citet[\Thm~1]{PAC_Bayes_nn} is that we introduce a secondary case.
The main case analysis is similar.

\paragraph{Main case.}
In the main case, $\zeta^L \geq \gamma / (2 B)$.
We choose the prior distribution $\prior$ to be $\Ns$ and consider the perturbation $u \sim \Ns$.
As in \citet{PAC_Bayes_nn}, since the prior distribution $\prior$ cannot depend on $\zeta$, we consider predetermined values of $\tilde{\zeta}$ on a grid and then do a union bound.
For now, we consider $\tilde{\zeta}$ fixed and consider all $\zeta$ such that $|\zeta - \tilde{\zeta}| \leq \zeta / L$.
This ensures that each relevant value of $\zeta$ is covered by some $\tilde{\zeta}$ on the grid.
Since $|\zeta - \tilde{\zeta}| \leq \zeta / L$ we get the inequalities
\begin{equation}\label{eq:zeta_e_bd}
\textstyle
  \zeta^{L-1} / e \leq \tilde{\zeta}^{L-1} \leq e\zeta^{L-1}.
\end{equation}
This follows from the inequalities $(1 + 1 / x)^{x-1} \leq e$ and $1 / e \leq (1 - 1 / x)^{x - 1}$ which themselves are consequences the inequality $1 + y \leq e^y$ for all $y$.
Since the entries of each $U_i$ are drawn from $\Ns$, we have the bound on the spectral norm of each $U_i$ \citep{Tropp_2011}
\begin{equation*}\label{eq:tropp}
\textstyle
  \prob_{U_i \sim \Ns}(\|U_i\|_2 > t) \leq 2 \bh e^{-t^2 / (2 \bh \sigma^2)}.
\end{equation*}
We can take a union bound to get
\begin{equation}\label{eq:tropp_union}
\textstyle
  \prob_{U_1, \dots, U_L\sim \Ns}(\|U_1\|_2 \leq t, \dots, \|U_L\|_2 \leq t) \geq 1 - 2 L \bh e^{-t^2 / (2 \bh \sigma^2)}.
\end{equation}
By setting the right hand side of~\eqref{eq:tropp_union} to $1 / 2$, we establish that with probability at least $1 / 2$, the spectral norm of every perturbation $U_i$ is bounded by $\sigma \sqrt{2 \bh \log(4 L \bh)}$ simultaneously.
We choose $\sigma = \gamma / (21 L B\tilde{\zeta}^{L-1}\sqrt{\bh \log(4 \bh L)})$ and now verify that with probability at least $1/2$, $\|U_i\|_2 \leq \|W_i\|_2 / L = \zeta / L$ holds, a condition of ~\citet[\Lem~2]{PAC_Bayes_nn}:
\begin{align*}
  \|U_i\|_2 &\leq \sigma \sqrt{2 \bh \log(4 L \bh)}
  = \gamma \sqrt{2} / (21 L B \tilde{\zeta}^{L-1})\\
  &\leq e 2 \sqrt{2} \gamma / (42 L B \zeta^{L-1})
  \leq 2 \sqrt{2} e \zeta / (21 L)  \leq \zeta / L.
\end{align*}
In the first line, the inequality comes from the perturbation bound on $\|U_i\|_2$, and the equality follows from plugging in for $\sigma$.
The second line follows from~\eqref{eq:zeta_e_bd}, and the assumption from the main case that $\zeta^L > \gamma / (2 B)$.
Now that the conditions are met, we apply~\citet[\Lem~2]{PAC_Bayes_nn}.
The following holds with probability at least $1 / 2$:
\begin{talign*}
\textstyle
    \max_{\theta \in \Theta}\|\hw(\theta) - \hwp(\theta)\|_2 &\leq e B \zeta^{L-1} \sum_{i=1}^L \|U_i\|_2\\
    &\leq e^2 L B \tilde{\zeta}^{L-1}\sigma \sqrt{2 \bh \log(4 L \bh)} \leq \gamma / 2.
\end{talign*}
In the second inequality, we use~\eqref{eq:zeta_e_bd}.
The last inequality follows from the choice of $\sigma$.
Now we calculate the KL-term with $\prior \sim \Ns$ and $u$ chosen with the above value of $\sigma$:
\begin{talign}
  {\rm KL}(w + u || \prior) &\leq \frac{\|w\|_2^2}{2\sigma^2}
  = \frac{21^2 L^2 B^2\tilde{\zeta}^{2L-2} \bh \log(4 \bh L)}{2 \gamma^2}\sum_{i=1}^L \|W_i\|_F^2 \notag \\
  &\leq \frac{21^2\zeta^{2L}}{2\gamma^2} B^2 L^2 \bh \log(4 L \bh)  \sum_{i=1}^L \frac{\|W_i\|_F^2}{\zeta^2}. \label{eq:kl_line2}
\end{talign}
What remains is to take a union bound over the different choices of $\tilde{\zeta}$.
We only need to consider values of $\zeta$ in the range of
\begin{equation}\label{eq:zeta_range}
\textstyle
  (\gamma / (2B))^{1/L} \leq \zeta \leq (\gamma \sqrt{N} / (2B))^{1/L}.
\end{equation}
Since we are in the main case, we do not have to consider $\zeta^L < \gamma / (2B)$.
Alternatively, if $\zeta^L > \gamma \sqrt{N} / (2 B)$, then the upper bound on the KL term in~\eqref{eq:kl_line2} is greater than $N$.
To see this, first note that the frobenius norm is always at least the operator norm of a given matrix, so $\|W_i\|_F \geq \zeta$ for $i=1, \dots, L$.
Then, the right hand side of~\eqref{eq:kl_line2} becomes at least $21^2 L^2 \bh \log(4 L \bh) N / 8$ which is greater than $N$.
\Thm~\ref{thm:gen_pac} is obtained by using the bound in the right hand side of~\eqref{eq:kl_line2} for the KL term in \Lem~\ref{lemma:marg}.
Therefore \Thm~\ref{thm:gen_pac} holds trivially since $\Cgtt$ upper bounds $R^t(w)$ and the entire square root term in \Lem~\ref{lemma:marg} is at least one.
Hence, we only need to consider $\zeta$ in the range of~\eqref{eq:zeta_range}.

The condition $L |\tilde{\zeta} - \zeta| \leq (\gamma / (2B))^{1/ L}$ is sufficient to satisfy the required condition that $|\tilde{\zeta} - \zeta| \leq \zeta / L$ since $\zeta^L \geq \gamma / (2B)$.
For each $\tilde{\zeta}$ that we pick, we consider $\zeta$ within a distance of $(\gamma / (2B))^{1/L} / L$.
We need to pick enough $\tilde{\zeta}$'s to cover the whole region in~\eqref{eq:zeta_range}.
Picking a cover size of $LN^{\frac{1}{2L}}$ satisfies this condition since
\begin{equation*}
  \frac{(\frac{\gamma \sqrt{N}}{2B})^{1/L} - (\frac{\gamma}{2B})^{1/L}}{\frac{1}{L}(\frac{\gamma}{2B})^{1/L}} =  L(N^{1 / (2L)} - 1).
\end{equation*}
Therefore, by using \Lem~\ref{lemma:marg}, with probability at most $\tilde{\delta}$ and for all $\tilde{w}$ such that $|\zeta - \tilde{\zeta}| \leq \zeta / L$, the following bound is violated:
\begin{equation*}
\textstyle
  R^t(w) \leq \hat{R}^t_{\gamma}(\tilde{w}) + \mathcal{O} \left(\sqrt{\frac{B^2 L^2 \log(L \bh) \Pi_{j=1}^L \|\tilde{W}_j\|_2^2 \sum_{i=1}^L \frac{\|\tilde{W}_i\|_F^2}{\|\tilde{W}_i\|_2^2} + \log(\frac{N}{\tilde{\delta}})}{\gamma^2 N}}\right).
\end{equation*}
By applying the union bound over the cover size, with probability at most $\tilde{\delta} LN^{1/ (2L)}$, the same bound is violated for at least one of the $\tilde{\zeta}$'s out of the cover.
Setting $\delta = \tilde{\delta} L N^{1 / (2L)}$ and recalling that the proof generalizes from normalized weights $\tilde{w}$ to weights $w$ gives the final result.

\paragraph{Secondary case.}
In this case, $\|\hw(\theta)\|_2 \leq B (\Pi_{i=1}^L \|W_i\|_2) \leq \gamma / 2$.
We get the following:
\begin{talign*}
  R^t(w) &\leq R^t_{\gamma/2}(0) \\
  &\leq \hat{R}^t_{\gamma/2}(0) + \Cgtt \sqrt{\log (1 / \delta) / (2N)} \quad \text{w.p. at least } 1 - \delta\\
  &\leq \hat{R}^t_{\gamma}(w) + \Cgtt \sqrt{\log (1 / \delta) / (2N)} \quad \text{w.p. at least } 1 - \delta
\end{talign*}
The first and third lines come from $\|\hw(\theta)\|_2 \leq \gamma / 2$ and the definition of the marginal fixed-point residual.
The second lines uses Hoeffding's inequality as in \citet[\Eqn~1.3]{pac_bayes_intro}, which is permissible since the prediction is the zero vector and is therefore independent of the data.

\subsection{Proof of \Lem~\texorpdfstring{\ref{lem:fp_reg_bound}}{}}\label{sec:fp_2_reg_proof}
First, let $z^\star(\theta)$ be the nearest fixed-point of the operator $\op$ to $z$ so that $r_{\theta}(z) = \|z - z^\star(\theta)\|_2$.
\begin{equation*}
\textstyle
\ell^{\rm{fp}}_{\theta}(z) = \|\op(z) - z\|_2  \leq \|\op(z) -z^\star(\theta) \|_2 + \|z -z^\star(\theta) \|_2 \leq 2 r_{\theta}(z)
\end{equation*}
The first inequality uses the triangle inequality, and the second inequality uses the non-expansiveness of $\op$.

\subsection{Proof of \Lem~\texorpdfstring{\ref{lemma:nonexp_r}}{}}\label{proof:nonexp_r}
\begin{talign*}
    |\rok[z] - \rok[w]| &= |\|\Tj[k][z] - \Pi_{\fix \op}(\Tj[k][z])\|_2 - \|\Tj[k][w] - \Pi_{\fix \op}(\Tj[k][w])\|_2| \\
    & \leq \|\Tj[k][z] - \Pi_{\fix \op}(\Tj[k][z]) + \Pi_{\fix \op}(\Tj[k][w]) - \Tj[k][w] \|_2 \\
    & \leq \|\Tj[k][z] - \Tj[k][w]\|_2 + \|\Pi_{\fix \op}(\Tj[k][z]) - \Pi_{\fix \op}(\Tj[k][w])\|_2 \\
    & \leq 2 \|\Tj[k][z]-\Tj[k][w]\|_2 \leq 2 \|z-w\|_2
\end{talign*}
The first two inequalities use the reverse triangle inequality and triangle inequality.
Since $\op$ is non-expansive, $\fix \op$  is a convex set~\citep[\Sec~2.4.1]{mon_primer}.
The third inequality follows since the projection onto a convex set is non-expansive \citep[\Sec~3.1]{mon_primer}.
In the last inequality, we use the non-expansiveness of $\op$.

\bibliography{bibliography_nourl}

\end{document}